%% file: grid.tex
\documentclass[11pt]{amsart}

\include{template}

\begin{document}

\title{GRID invariants in universally tight lens spaces}

\author[Lev Tovstopyat-Nelip]{Lev Tovstopyat-Nelip}
\address{Department of Mathematics \\ Michigan State University}
\email{tovstopy@msu.edu}

% \thanks{JAB was supported by NSF Grant DMS-1406383 and NSF CAREER Grant DMS-1454865.}

\begin{abstract} 
We define combinatorial invariants of Legendrian and transverse links in universally tight lens spaces using grid diagrams, generalizing \cite{grid} and prove that they are equivalent to the invariants defined in \cite{equiv} and \cite{LOSS}. We use these combinatorial invariants to characterize index one grid diagrams for knots in lens spaces which admit surgeries to the 3-sphere and discuss a potential application to the Berge conjecture.
%We prove that the transverse invariant of the (disconnected) binding of an open book union a braid about it is nonzero.

 \end{abstract}

\maketitle

\vspace{-1cm}
\section{introduction}
\label{sec:intro}

Using grid diagrams Ozsv\'{a}th, Szab\'{o} and Thurston \cite{grid} defined combinatorial invariants of Legendrian and transverse links in the tight 3-sphere. These invariants have been shown to be effective, meaning they distinguish some Legendrian and transverse knots having identical classical invariants. Grid diagrams for links in lens spaces have been studied in \cite{lensgridcomb}, and their relationship with Legendrian links in universally tight lens spaces laid out in \cite{lensgridleg}. 
Roughly speaking, such a grid diagram is just the usual toroidal Heegaard diagram for $L(p,q)$ with $n$ parallel copies of the $\alpha$ and $\beta$ curves, along with some choice of $2n$ basepoints in their complements; $n$ is the \emph{index} of the grid diagram. 
The invariants of \cite{grid} admit natural generalizations to the case of links in universally tight lens spaces:

\begin{theorem}
\label{thm:grid}
For a grid diagram $G$ encoding a link $K\subset L(p,q)$, let $L \subset (L(p,q),\xi_{UT})$ denote the corresponding oriented Legendrian representative of $K$. There are two associated cycles $\bold{x}^+,\bold{x}^-\in CFK^- (G)$ supported in Maslov gradings
\begin{align*}
M(\bold{x}^+) = tb_\mathbb{Q}(L) - rot_\mathbb{Q}(L) +\frac{1}{p} -d(p,q,q-1)\\
M(\bold{x}^-) = tb_\mathbb{Q}(L) + rot_\mathbb{Q}(L) +\frac{1}{p}-d(p,q,q-1).
\end{align*}
If $K$ is a knot, the cycles are supported in Alexander gradings
\begin{align*}
A(\bold{x}^+) = \frac{1}{2}\Big{(}tb_\mathbb{Q}(L) - rot_\mathbb{Q}(L) +1 \Big{)}\\
A(\bold{x}^-) =\frac{1}{2}\Big{(}tb_\mathbb{Q}(L) + rot_\mathbb{Q}(L) +1 \Big{)}.
\end{align*}
The homology classes $[x^+]$ and $[x^-]$ in $HFK^- (L(p,q),K)$, denoted $\lambda ^+ (L)$ and $\lambda ^- (L)$, are invariants of the oriented Legendrian isotopy class of $L$. Let $L^-$ (respectively $L^+$) denote the negative (respectively positive) Legendrian stabilization of $L$ in $(L(p,q),\xi_{UT})$. We have that 
\begin{align*}
\lambda^+ (L^-) = \lambda^+(L) \quad \quad \quad  \quad  \lambda^- (L^-) = U\cdot\lambda^-(L)\\
\lambda^+ (L^+) = U\cdot \lambda^+(L) \quad \quad \quad  \quad \lambda^- (L^+) = \lambda^-(L).
\end{align*}
\end{theorem}

In the case of a link, we compute the rational Alexander multi-gradings in Proposition \ref{prop:agradingcomp}.

%As we will later see in Subsection \ref{subsec:contactgrid}, the invariants $\lambda^+$ and $\lambda^-$, along with the operations of positive and negative Legendrian stabilization, are exchanged under co-orientation reversal of the contact structure. 

Transverse isotopy classes correspond to Legendrian isotopy classes modulo negative Legendrian (de)stabilizations \cite{transverseapprox}. Let $T$ denote the positive transverse push-off of $L$ and define $\theta (T)$ to be $\lambda ^+ (L)$, the corollary below follows immediately:

\begin{corollary}
\label{corollary:grid}
The homology class $\theta(T)$ is an invariant of the transverse isotopy class of $T$, and is supported in Maslov grading
\begin{align*}
M(\theta(T)) = sl_{\mathbb{Q}}(T) +\frac{1}{p}-d(p,q,q-1).
\end{align*}
If $T$ is a knot, then the invariant is supported in Alexander grading
\[
A(\theta(T)) = \frac{1}{2}\Big{(}sl_\mathbb{Q}(T)  +1\Big{)}.
\]
\end{corollary}

We refer to the invariants $\lambda^\pm, \theta$ as the GRID invariants.

Suppose that $(B,\pi)$ is an open book supporting a contact manifold $(Y,\xi)$. Any link $K$ braided about this open book admits a natural transverse representative, and in fact via the general transverse Markov theorem \cite{pav} all transverse links arise as braids about $(B,\pi)$.
Using this perspective Baldwin, Vela-Vick and V\'{e}rtesi \cite{equiv} defined the BRAID invariant of transverse links in $(Y,\xi)$:
\[
t(K) \in HFK^- (-Y,K).
\]
%Their construction is reminiscent of the reinterpretation of the contact class due to Honda, Kazez and Mati\'{c} \cite{HKM}.

The contact class in Heegaard Floer homology is characterized in terms of the Alexander filtration induced by the binding of any supporting open book on the Heegaard Floer chain complex \cite{contactclass}. The BRAID invariant admits a similar characterization. Suppose that $Y$ is a rational homology sphere. If $K$ is braided about $(B,\pi)$, we may consider the filtration on $CFK^- (-Y,K)$ induced by $-B$:
\[
\emptyset = \mathcal{F}^{-B}_i\subset\mathcal{F}^{-B}_{i+1}\subset \dots \subset \mathcal{F}^{-B}_j = {CFK}^-(-Y,K).
\] 

Set
$bot := min\{j| H_* (\mathcal{F}_j^{-B})\ne 0 \}$
and let $H_{top}(\mathcal{F}^{-B}_{bot})$ denote the summand of $H_*(\mathcal{F}^{-B}_{bot})$ of maximal rational Maslov grading. We prove in Subsection \ref{subsec:reformulation} that $H_{top}(\mathcal{F}^{-B}_{bot})$ is a rank one $\mathbb{F}[U_1,\dots, U_m]$-module, where $m$ is the number of components of $K$.
The invariant $t(K)$ is the image of a generator under the natural map $H_{top}(\mathcal{F}^{-B}_{bot}) \to HFK^- (-Y,K)$. \footnote{This is not exactly right, see Subsection \ref{subsec:reformulation} for precise statements.}

The above reformulation of $t(K)$ was first proven in the case of a braid about the unknot in $S^3$ \cite{equiv}. There, the authors show that the GRID invariant $\theta (K)$ for a transverse link $K\subset (S^3,\xi_{std})$ also admits such a reformulation, and use this to prove the equivalence of $t(K)$ and $\theta (K)$ in this special case. Generalizing their approach we prove the following:

\begin{theorem}
\label{thm:equivalence}
Let $K\subset (L(p,q),\xi_{UT})$ be a transverse link, then the GRID and BRAID invariants are equivalent. There exists a graded isomorphism of $\mathbb{F}[U_1,\dots,U_m]$-modules
\[
HFK^-(-L(p,q),K)\to HFK^-(-L(p,q),K)
\]
mapping the class $\theta(K)$ to $t(K)$.
\end{theorem}

Our proof involves generalizing the reformulation of $t(K)$ to braids about rational open books having connected binding. 
The lens space $L(p,q)$ can be obtained by $-p/q$ surgery on the unknot in $S^3$; the core of the filling torus is a rationally fibered knot $B\subset L(p,q)$ with $D^2$ fibers.  We let $(B,\pi)$ denote this rational open book; the monodromy $\pi$ is a $2\pi q/p$ boundary twist. $(B,\pi)$ supports a universally tight contact structure $\xi_{UT}$ on $L(p,q)$ \cite{CCSMCM}.  We refer to a braid about $(B,\pi)$ as a \emph{lens space braid}. %Any transverse link in $(L(p,q),\xi_{UT})$ can be realized as such a braid.

%In Section \ref{sec:diagram} we present a standard Heegaard diagram encoding $(-L(p,q),K)$, and specify a generator representing $t(K)$. 
In Section \ref{sec:altchara} we prove an alternative reformulation of $t(K)$ involving the Alexander filtration induced by the unknotted Seifert cable of the rational binding $B$. We show in Section \ref{sec:gridchara} that the GRID invariant $\theta(K)$ of Corollary \ref{corollary:grid} satisfies the same reformulation involving the Seifert cable of the binding, allowing us to prove Theorem \ref{thm:equivalence}.

Lisca, Ozsv\'{a}th, Stipsicz and Szab\'{o} \cite{LOSS} defined invariants of null-homologous Legendrian and transverse knots in a contact manifold $(Y,\xi)$.
\[
\mathcal{L}(K),\mathcal{T}(K)\in HFK^- (-Y,K).
\] 
These are referred to as the LOSS invariants. For a transverse knot $K$, it is proven directly in \cite{equiv} that $\mathcal{T}(K) = t(K)$; Theorem \ref{thm:equivalence} allows us to conclude that for a transverse knot $K\subset (L(p,q),\xi_{UT})$ the invariants $\theta (K)$ and $\mathcal{T}(K)$ are equivalent. The equivalence of Legendrian invariants also follows immediately:

\begin{corollary}
\label{corollary:same}
If $K\subset (L(p,q),\xi_{UT})$ is a Legendrian knot , the Legendrian GRID and LOSS invariants are equivalent, i.e. there 
exists a graded isomorphism of $\mathbb{F}[U]$-modules
\[
HFK^-(-L(p,q),K)\to HFK^-(-L(p,q),K)
\]
mapping the class $\lambda^{+} (K)$ to $\mathcal{L}(K)$ and the class $\lambda^- (K)$ to $\mathcal{L}(-K)$, where $-K$ denotes the knot with reversed orientation.
\end{corollary}

%UNEDITED

\subsection{The Berge conjecture and simple knots}
These new invariants may give some insight into the Berge conjecture, which proposes an exhaustive list of knots in the 3-sphere admitting Dehn surgeries to lens spaces.
Dually, the Berge conjecture is equivalent to showing that any knot in a lens space admitting a surgery to the 3-sphere is \emph{simple}; i.e. can be encoded with an index one grid diagram.
%Any knot $K\subset L(p,q)$ admitting an $S^3$ surgery is \emph{Floer simple} \cite{Hberge}, i.e. has $\widehat{HFK}$ as small rank as possible.
%The remaining step of the program of \cite{lensgridcomb} is to show that a Floer simple knot must be simple. 
Given a grid diagram for a knot $K\subset L(p,q)$, one can \emph{dualize} to obtain a grid diagram for the mirror $K\subset -L(p,q)$. Via this construction, a grid diagram gives rise to Legendrian representatives of the encoded knot $K$ and its mirror.
The GRID invariants can be used to characterize index one grid diagrams for knots admitting $S^3$ surgeries:
\begin{theorem} 
\label{thm:gridmirror}
Suppose $K\subset L(p,q)$ admits an $S^3$ surgery. If a grid diagram $\mathcal{G}$ for $K$ gives rise to Legendrian representatives $L_0\subset (L(p,q),\xi_{UT})$ and $L_1\subset (L(p,p-q),\xi_{UT})$ then
\[
\mathcal{G} \text{ is an index one diagram} \iff \widehat{\lambda}^+(L_0),\widehat{\lambda}^-(L_0),\widehat{\lambda}^+(L_1),\widehat{\lambda}^-(L_1)\ne 0.
\]
\end{theorem}

The forward direction is immediate, if $\mathcal{G}$ is an index one diagram the two complexes used to define the quadruple of invariants have trivial differential. The reverse direction utilizes the Floer simplicity of $K$, see Section \ref{sec:gridmirror}.
\noindent %The proof of the above appeals to the Floer simplicity of $K$.

We formulate two conjectures:

\begin{conjecture}
\label{conj:rep}
Any knot $K\subset L(p,q)$ admitting a surgery to the 3-sphere has a Legendrian representative $L\subset (L(p,q),\xi_{UT})$  such that $\widehat{\lambda}^+(L),\widehat{\lambda}^-(L)\ne 0$. 
\end{conjecture}

\begin{conjecture}
\label{conj:grid}
Given any Legendrian representatives (in universally tight lens spaces) of $K\subset L(p,q)$ and its mirror $m(K)\subset L(p,p-q)$ there exists a grid diagram giving rise to the pair of Legendrians.
\end{conjecture}

Assuming the validity of these conjectures, the Berge conjecture follows readily via an application of Theorem \ref{thm:gridmirror}. 
Suppose $K\subset L(p,q)$ admits a surgery to the 3-sphere. Conjecture \ref{conj:grid} produces a grid diagram $\mathcal{G}$ for $K$ giving rise to the Legendrian representatives of $K$ and its mirror guaranteed by Conjecture \ref{conj:rep}, where the associated quadruple of invariants is non-vanishing. Theorem \ref{thm:gridmirror} implies that $\mathcal{G}$ is an index one diagram, hence $K$ is a simple knot.

Any knot $K\subset L(p,q)$ which admits a surgery to the 3-sphere is rationally fibered and supports a tight contact structure. Progress on Conjecture \ref{conj:rep} is likely to occur first for knots supporting a universally tight contact structure, and will probably appeal to properties of the LOSS invariant $\widehat{\mathcal{L}}$ via Corollary \ref{corollary:same}. Conjecture \ref{conj:grid} is known to hold for knots in the tight 3-sphere, this is a result of Dynnikov and Prasolov \cite{jonescon}.

\subsection{Acknowledgements} We thank John Baldwin for many helpful discussions. 

\section{Contact Preliminaries}
\label{sec:contact}
\subsection{Contact Geometry}
\label{subsec:contact}
We assume the reader has a certain knowledge of contact geometry. 
For an introduction to the Giroux correspondence and open books consult the wonderful notes of Etnyre \cite{openbooks}. For a reference on transverse and Legendrian links we point the reader to Etnyre's survey \cite{legtransverse}. 

%We explain the connection between transverse links and braids, and illustrate how to braid the binding of an open book about itself. 

Let $(Y,\xi)$ be a contact 3-manifold. Suppose that $(B,\pi)$ is an open book supporting $(Y,\xi)$. $B$ sits naturally as a transverse link. Any link braided about $B$ is also naturally a transverse link, as the contact plane field is very close to the plane field tangent to the pages away from the binding $B$. The following is a generalization of a theorem of Bennequin \cite{ben}

\begin{theorem} \cite{pav}
Suppose $(B,\pi)$ is an open book supporting $(Y,\xi)$. Every transverse link in $(Y,\xi)$ is transversely isotopic to a braid with respect to $(B,\pi)$.
\end{theorem}

There is a notion of positive Markov stabilization for braids with respect to an arbitrary open book, defined in \cite{pav}. This operation increases the braid index by one, but preserves the transverse isotopy class of the braid. The following is a generalization of the transverse Markov theorem of Wrinkle \cite{wrinkle}.

\begin{theorem} \cite{pav}
Suppose $K_1$ and $K_2$ are braids with respect to an open book $(B,\pi)$ supporting $(Y,\xi)$. $K_1$ and $K_2$ are transversely isotopic if and only if they admit positive Markov stabilizations $K_1 ^+$ and $K_2 ^+$ which are braid isotopic with respect to $(B,\pi)$.
\end{theorem}

The binding $B$ of an open book supporting $(Y,\xi)$ sits naturally as a transverse link. 
A copy of $B$ may be braided about the underlying open book, resulting in a braid of index $n$, where $n$ is the number of components of $B$.

Recall that the neighborhood of a transverse knot in any contact manifold is standard. If $K$ is transverse it admits a neighborhood contactomorphic to 
\[
N_\epsilon = \{(r,\theta,z):r<\epsilon\} \subset \mathbb{R}^2 \times S^1
\]
where $\xi = ker(\alpha) = ker(dz+r^2 d\theta)$, and $K$ is identified with $(0,0)\times S^1$. In these coordinates, $K$ admits a parametrization $\gamma (t) = (0,t,t)$, where $t\in [0,2\pi)$. Consider the following transverse isotopy
\[
\Gamma _s (t) = (s,t,t)
\]
from $\gamma _0 (t)$ to $\gamma_{\epsilon /2} (t)$. Applying this isotopy to each component of $B$ realizes a copy of $B$ as an index $n$ braid.

We will also use the notion of a rational open book and how one supports a contact structure; the original reference is \cite{CCSMCM}. Whenever an open book is rational we will emphasize it, otherwise open books are assumed to be integral. Any link braided about a rational open book also sits naturally as a transverse link in the supported contact manifold.

The unknot $U\subset S^3$ is fibered with disk pages. We take the convention that the lens space $L(p,q)$ is obtained by $-p/q$ surgery on the unknot. Let $B\subset L(p,q)$ be the core of the filling torus in the Dehn surgery.
$B$ is the binding of a rational open book decomposition $(B,\pi)$ for $L(p,q)$. The open book has $D^2$ pages, and the monodromy $\pi$ is a counter-clockwise $2\pi q/p$-rotation, which we denote by $\delta ^{q/p}$.

Honda \cite{contact1} classifies the universally tight contact structures (those which lift to the tight contact structure on $S^3$) on lens spaces. There are at most two such contact structures on $L(p,q)$. The structure is unique if  $q=p-1$, otherwise the two universally tight contact structures are related by co-orientation reversal. 

We let $\xi _{UT}$ denote the contact structure on $L(p,q)$ supported by the rational open book $(B,\pi)$, see \cite{CCSMCM} for a proof that the contact structure is universally tight. $\xi_{UT}$ is also constructed explicitly as the kernel of a globally defined 1-form in section 3 of \cite{lensgridleg}.

We will need the classical invariants for rationally null-homologous Legendrian and transverse links, studied in \cite{RLCG}.

\begin{definition}
Let $L\subset Y$ be an oriented, rationally null-homologous link, which is partitioned into sub links $L_1\cup\dots \cup L_l$. Let $r$ denote the least common multiple of the orders of components of $L$ in $H_1 (Y;\mathbb{Z})$. Let $i:\Sigma \to Y$ be a \emph{uniform rational Seifert surface} for $L$, by which we mean an oriented surface whose boundary wraps $r$ times around each component of $L$ (even those components whose orders in homology are less than $r$).
\begin{itemize}
\item Given another oriented link $L'$, define the \emph{rational linking number} of $L$ with $L'$ to be
\[
lk_{\mathbb{Q}} (L,L') = \frac{1}{r}\Sigma \cdot  L'
\]
where $\Sigma \cdot L'$ is the algebraic intersection number. In general the rational linking number may depend on the relative homology class of rational Seifert surface $\Sigma$ for $L$. We will only consider rational homology spheres, so we continue to suppress $\Sigma$ from the notation.
\item If $L$ is a Legendrian link in $(Y,\xi)$ let $L'$ denote the longitude for $L$ specified by the contact framing $\xi |_L \cap \nu (L)$. We define the \emph{rational Thurston-Bennequin number} of $L$ to be
\[
tb_{\mathbb{Q}}(L) = lk_\mathbb{Q} (L,L').
\]
\item Consider a trivialization of $\xi|_{\Sigma} \simeq \Sigma \times \mathbb{R}^2$. Since $L$ is oriented, the positive unit tangent vectors to $L$ give rise to a nonzero section $\sigma$ of this trivial bundle restricted to $\partial \Sigma$. We define the \emph{rational rotation number} of $L$ to be the winding number of $\sigma$ divided by $r$.
\[
rot_\mathbb{Q} (L) = \frac{1}{r} winding(\sigma,\mathbb{R}^2)
\]
We also denote the winding of the section $\sigma$ restricted to the component of $\partial \Sigma$ which wraps around the sub-link $L_i\subset L$ by 
\[
rot^i_\mathbb{Q}(L).
\]
%Note that this is not the same as the rational rotation number of $L_i$.
\item Suppose $L$ is a transverse link in $(Y,\xi)$. The map $i|_{\partial \Sigma} : \partial \Sigma \to L$ is an $r$-fold covering map. The map $i$ induces a map from a small neighborhood of the zero-section of $i^* \xi$ to a small neighborhood of the zero section of $\xi |_L$, which in turn is naturally identified with a neighborhood $\nu(L)$ of $L$ (because $L$ is transverse). The bundle $i^* \xi$ is trivial,  so there exists a nonvanishing section $v$. A small and generic choice of section $v$ gives rise to a link $L'$ sitting in $\nu(L)\smallsetminus L$. We define the \emph{rational self-linking number} of $L$ to be
\[
sl_\mathbb{Q} (L) = \frac{1}{r}lk_\mathbb{Q}(L,L')
\]
the coefficient $\frac{1}{r}$ appears because $L'$ is an $r$-fold push off of $L$.
%Restricting the nonvanishing section $v$ to the component of $\Sigma$ which wraps around $L_i$ gives rise to a sublink link $L_i'\subset L'$. 
We define
\[
sl^i_\mathbb{Q}(L) = \frac{1}{r}lk_\mathbb{Q}(L_i,L').
\]

\end{itemize}
\end{definition}

\subsection{Grid diagrams for links in lens spaces}
\label{subsec:contactgrid}

Grid diagrams for links in lens spaces were first studied in \cite{lensgridcomb}, and subsequently from a contact geometric perspective in \cite{lensgridleg}. 

Let $T^2$ be the standard torus $\mathbb{R}^2/\mathbb{Z}^2$, where $\mathbb{Z}^2$ is the standard lattice generated by $(1,0)$ and $(0,1)$. Let $\pi :\mathbb{R}^2 \to \mathbb{R}^2/\mathbb{Z}^2 = T^2$ denote the quotient map.

\begin{definition}
Suppose that $0<q<p$. A \emph{grid diagram} for a link $K\subset L(p,q)$, with \emph{index} $n$ is a Heegaard diagram $\mathcal{G} = (T^2,\boldsymbol{\alpha},\boldsymbol{\beta},\bold{z},\bold{w})$ where
\begin{itemize}
\item $\boldsymbol{\alpha} = \{\alpha_0,\dots,\alpha_{n-1}\}$, where $\alpha_i$ is the image of the line $y=i/n$ under the map $\pi$. The $n$ annular components of $T^2 \smallsetminus \boldsymbol{\alpha}$ are called the \emph{rows} of $\mathcal{G}$. For $0\le i < n$ the row between $\alpha_{i}$ and $\alpha_{i+1}$ is called the $i^{th}$ row.
\item $\boldsymbol{\beta} = \{\beta_0,\dots,\beta_{n-1}\}$, where $\beta_i$ is the image of the line $y= -\frac{p}{q}(x-\frac{i}{pn})$ under the map $\pi$. The $n$ annular components of $T^2 \smallsetminus \boldsymbol{\beta}$ are called the \emph{columns} of $\mathcal{G}$. For $0\le i < n$ the row between $\beta_{i}$ and $\beta_{i+1}$ is called the $i^{th}$ column.
\item $\bold{z} = \{z_0,z_1,\dots,z_{n-1}\}$. For each $i$, the basepoint $z_i$ is in the $i^{th}$ column.
\item $\bold{w} = \{w_0,w_1,\dots,w_{n-1}\}$. For each $i$, the basepoint $w_i$ is in the $i^{th}$ column.
\item Each region of $T^2 \smallsetminus \boldsymbol{\alpha} \smallsetminus \boldsymbol{\beta}$ contains at most one basepoint of $\bold{z}\cup\bold{w}$.

\item Each row contains two basepoints, one from $\bold{z}$ and one from $\bold{w}$. 
\end{itemize}
\end{definition}

In Proposition 4.3 of \cite{lensgridcomb} is it shown that every link $K\subset L(p,q)$ is represented by a grid diagram. 
Baker and Grigsby \cite{lensgridleg} introduce the notion of a toroidal front diagram; these diagrams are similar to regular front diagrams for knots in $(\mathbb{R}^3, \xi_{std})$: 

\begin{proposition}(Proposition 3.3 of \cite{lensgridleg})
A toroidal front diagram uniquely specifies a Legendrian link in $(L(p,q),\xi_{UT})$ up to Legendrian isotopy.
\end{proposition}
\begin{proposition}(Proposition 3.4 of \cite{lensgridleg})
Every Legendrian isotopy class in ($L(p,q),\xi_{UT}$) admits a representative admitting a toroidal front projection.
\end{proposition}

To a grid diagram $\mathcal{G}$ for $K$ one can associate several rectilinear projections of $K$ onto $T^2$. Baker and Grigsby describe how to canonically perturb a rectilinear projection of $K$ into a Legendrian front, they prove:
\begin{proposition}(Lemma 4.5 and Proposition 4.6 of \cite{lensgridleg})
\label{prop:front}
A rectilinear projection associated to a grid diagram $\mathcal{G}$ for $K\subset L(p,q)$ uniquely specifies a toroidal front for $L$, a Legendrian in $(L(p,q),\xi_{UT})$. Moreover, the Legendrian isotopy class of 
$L$ is independent of choice of rectilinear projection coming from $\mathcal{G}$, hence a grid diagram uniquely specifies a Legendrian representative of $K$.
\end{proposition}
The assumption that $0<q<p$ is needed to obtain a toroidal front diagram, the pieces of the rectilinear projection inside the columns of the diagram must be negatively sloped.

\begin{remark}
\label{remark:ORIENTATIONS}
Baker and Grigsby actually show that a grid diagram for $K\subset L(p,q)$ induces a Legendrian representative of the mirror $K\subset -L(p,q)$. Their approach fits nicely with the conventions established in \cite{grid} and \cite{lensgridcomb} and can be easily modified to give the results stated in this paper. Our conventions will be more natural in proving Theorem \ref{thm:equivalence}.
\end{remark}

Baker and Grigsby classify the grid moves which preserve topological and Legendrian isotopy classes.

\begin{theorem}(Theorem 5.1 of \cite{lensgridleg})
\label{thm:leg}
Let $\mathcal{G}$ and $\mathcal{G} '$ be grid diagrams for $K$ and $K'$ in $L(p,q)$. Let $L$ and $L'$ denote the induced Legendrian representatives of $K$ and $K'$ by $\mathcal{G}$ and $\mathcal{G} '$ respectively. 
\begin{itemize}
\item $K$ and $K'$ are isotopic if and only if $\mathcal{G}$ and $\mathcal{G} '$ are related by a sequence of elementary grid moves.
\item $L$ and $L'$ are Legendrian isotopic if and only if $\mathcal{G}$ and $\mathcal{G}'$ are related by a sequence of elementary Legendrian grid moves.
\end{itemize}
\end{theorem}

The \emph{elementary grid moves} consist of eight different types of (de)stabilizations along with row and column commutations. See Figures \ref{fig:stab} and \ref{fig:commutation}.

\begin{figure}[h]
\def\svgwidth{175pt}
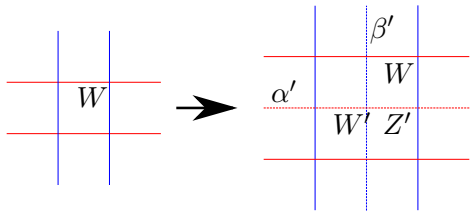
\caption{A W:NW stabilization is pictured. A new pair of curves, $\alpha '$ and $\beta '$, is added, in addition to a pair of basepoints. The ordinal direction of the stabilization indicates which slot is basepoint free after stabilizing. The corresponding destablization is the inverse of this procedure.}
\label{fig:stab}
\end{figure}

\begin{figure}[h]
\def\svgwidth{200pt}
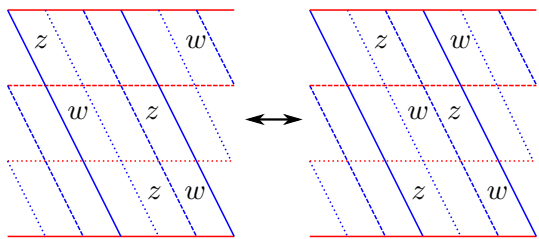
\caption{A commutation of the first and second columns on a index three diagram for a link in $L(2,1)$. Because the basepoint pairs in the columns do not interleave we are able to perform an exchange of the basepoints. Row commutations are the analogue for two adjacent rows. These commutations can be performed so long as the markings do not interleave.}
\label{fig:commutation}
\end{figure}

\begin{lemma} (Compare with Lemma 4.2 of \cite{grid})
A stabilization of type Z:SE (respectively Z:NE, Z:NW, or Z:SW) is equivalent to a stabilization of type W:NW(respectively W:SW, W:SE, or W:NE) followed by a sequence of commutation moves on the torus.
\end{lemma}
\begin{proof}
After performing a stabilization near a $Z$ basepoint, we can perform a sequence of commutation moves to get the desired diagram.
\end{proof}
The \emph{elementary Legendrian grid moves} are comprised of commutations along with (de)stabilizations of types W:NE, and W:SW. 

It is easy to see which elementary grid moves correspond to Legendrian (de)stabilization:

\begin{lemma}
(De)stabilizations of type W:SE and W:NW correspond to negative and positive Legendrian (de)stabilization respectively.
\end{lemma}

\begin{proof}
Consider the rectilinear projection $\lambda$ of $K$ constructed in the proof of Proposition \ref{prop:braiding}. This projection may be smoothed to a toroidal front projection of $L$, the associated Legendrian representative of $K$, having zero cusps. A stabilization of type W:SE has the effect of locally adding two upward oriented cusps, see Figure \ref{fig:negstab}. This can be thought of as taking place in a small Darboux ball, hence corresponds to negative Legendrian stabilization.

\begin{figure}[h]
\def\svgwidth{200pt}
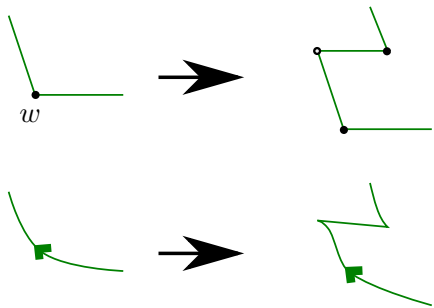
\caption{The effect of a W:SE stabilization on an associated rectilinear projection of $K$ is pictured on top. The effect on the associated toroidal front projection is pictured below.}
\label{fig:negstab}
\end{figure}

Likewise, by rotating the figures 180 degrees, a stabilization of type W:NW can be seen to have the effect of adding two downward oriented cusps, and hence corresponds to positive Legendrian stabilization.
\end{proof}

\begin{remark}
Since $(B,\pi)$ supports $\xi_{UT}$, the contact planes are oriented so that the ``upward" braid induced by a grid diagram is a positive transverse link, see Proposition \ref{prop:braiding}. Reversing co-orientation preserves the set of elementary Legendrian grid moves and exchanges the two moves corresponding to positive and negative Legendrian (de)stabilization; the set of elementary transverse grid moves is not preserved under co-orientation reversal. Let $\overline{\xi_{UT}}$ denote $\xi_{UT}$ with the reverse co-orientation.
\end{remark}

Transverse isotopy classes are in one to one correspondence with Legendrian isotopy classes up to negative Legendrian (de)stabilization \cite{transverseapprox}. The transverse isotopy class associated to a Legendrian link is obtained via positive transverse push-off. A grid diagram for $K$ gives rise to a transverse representative $T$ by taking the positive transverse push-off of $L$. (De)stabilizations of type W:SE in addition to the elementary Legendrian grid moves comprise the \emph{elementary transverse grid moves}. The following is evident:

\begin{proposition}
\label{prop:trans}
Let $\mathcal{G}$ and $\mathcal{G} '$ be grid diagrams for $K$ and $K'$ in $L(p,q)$. Let $T$ and $T'$ denote the induced transverse representatives of $K$ and $K'$ by  $\mathcal{G}$ and $\mathcal{G} '$ respectively. 
$T$ and $T'$ are transversely isotopic if and only if $\mathcal{G}$ and $\mathcal{G} '$ are related by a sequence of elementary transverse grid moves.
\end{proposition}

\begin{proposition}
\label{prop:braiding}
Let $(B,\pi)$ denote the rational open book supporting ($L(p,q),\xi_{UT}$) described in subsection \ref{subsec:contact}.
Each grid diagram $\mathcal{G}$ naturally induces a braiding $\mathcal{B}$ of $K$ about $B$, and hence a transverse representative $T$ of $K$.
\end{proposition}
\begin{proof}

Given a grid diagram $\mathcal{G} = (T^2,\boldsymbol{\alpha},\boldsymbol{\beta},\bold{z},\bold{w})$, we specify a longitude $\lambda$ for $K$ in the following way. In each column draw an oriented arc $\gamma _i ^\beta$ upward (i.e. having tangent vector $-q\frac{d}{dx} +p\frac{d}{dy}$) from $w_i$ to $z_i$. For each $i$, draw an arc in the $i^{th}$ row $\gamma^\alpha_i$ from a point of $\bold{z}$ to a point of $\bold{w}$ oriented right to left, (i.e. having tangent vector $-d/dx$). We push all horizontal arcs slightly into the $\boldsymbol{\alpha}$ handlebody, all vertical arcs into the $\boldsymbol{\beta}$ handlebody and set $\lambda = \bigcup_{i=0}^{n-1}( \gamma^\alpha_i \cup \gamma^\beta_i)$. The pages of $(B,\pi)$ meet $T^2$ in parallel copies of $\alpha_0$.
Note that $\lambda$ may be made positively transverse to all parallel copies of $\alpha_0$ in $T^2$ via a small isotopy, realizing $K$ as a braid about $B$.

%If $z_i$ and $w_i$ are in the same region of $\Sigma \smallsetminus (\boldsymbol{\alpha}\cup\boldsymbol{\beta})$, we adopt the convention that $\gamma_i ^\beta$ runs parallel to $\beta_i$, so that the corresponding unknot component of $\lambda$ comprised of $\gamma_i ^\beta \cup \gamma _l ^\alpha$ (where $l$ is the row containing $z_i$ and $w_i$) is a parallel copy of $\beta_i$. 
\end{proof}

%\begin{remark} 
%In the proof above we could have chosen the horizontal arcs $\gamma_i^\alpha$ to have any orientation, this gives the same braid (isotopic through braids). The rectilinear projection $\lambda$ we chose gives rise to a Legendrian front having no cusps, which simplifies the proof of Proposition \ref{prop:commute}. 
%\end{remark}

The \emph{elementary braid grid moves} consist of commutations along with (de)stabilizations of type W:SE and W:SW. These are precisely the grid moves which leave $\mathcal{B}$ unchanged. 
It is elementary to check the following proposition.
\begin{proposition}
\label{prop:braid}
Let $\mathcal{G}$ and $\mathcal{G} '$ be grid diagrams for $K$ in $L(p,q)$. Let $\mathcal{B}$ and $\mathcal{B}'$ denote the induced braid representatives of $K$ by  $\mathcal{G}$ and $\mathcal{G} '$ respectively. 
$\mathcal{B}$ and $\mathcal{B}'$ are braid isotopic if and only if $\mathcal{G}$ and $\mathcal{G} '$ are related by a sequence of elementary braid grid moves.
\end{proposition}

A (de)stabilization of type W:NE corresponds to a generalized positive Markov (de)stabilization in the rational setting. A generalized positive Markov stabilization has the effect of connect summing the braid with a braid index p unknot along a positively half twisted band.

The following proposition is a generalization of the analogous statement for grid diagrams in $S^3$, proven in \cite{gridcommutative}. Their proof generalizes to our setting.
\begin{proposition}
\label{prop:commute}
The two transverse links associated to a grid diagram coincide, i.e. the diagram
\[
\xymatrix{
\mathcal{G} \ar[d] \ar[r] & L \ar[d]\\
\mathcal{B} \ar[r]&T}
\]
commutes.
\end{proposition}

\begin{comment}

\begin{proof}

Let $\lambda$ denote the longitude for $K$ on $T^2$ constructed in the proof of Proposition \ref{prop:braiding}, and let $\Lambda$ denote the associated toroidal front of $L$. The construction of $L$ in Proposition 3.3 of \cite{lensgridleg} ensures that $L$ misses open neighborhoods of the cores of the $\boldsymbol{\alpha}$ and $\boldsymbol{\beta}$ handlebodies (this is expressed in their coordinates on $L(p,q)$ as $r_1 \in (0,1)$); in particular $L$ misses an open neighborhood $\nu (B)$ of the binding of $(B,\pi)$. 

Let $T$ denote a small positive transverse push-off of $L$ (so that the push-off takes place in the complement of $\nu(B)$); in what is to come we think of $T$ as a fixed embedding, not a transverse isotopy class. Since $(B,\pi)$ supports $\xi_{UT}$ there exists an isotopy through contact structures from $\xi_{UT}$ to $\xi$, such that the resulting contact structure $\xi$ is arbitrarily close to the tangent plane field of the pages of $(B,\pi)$ in the complement of $\nu (B)$. Transversality is an open condition, so we may assume that $T$ remains transverse throughout this isotopy. It is now clear that $\mathcal{B}$ gives rise to $T$ in $(L(p,q),\xi)$, and hence in $(L(p,q),\xi_{UT})$.

\end{proof}
\end{comment}

We have a rational transverse Markov theorem in $(L(p,q),\xi_{UT})$:
\begin{theorem}
\label{theorem:Markov}
Two braids about $(B,\pi)$ represent transversely isotopic knots in $(L(p,q),\xi_{UT})$ if and only if they are related by a sequence of braid isotopies and generalized positive Markov (de)stabilizations.
\end{theorem} 
\begin{proof}
The elementary transverse grid moves are the elementary braid grid moves in addition to (de)stabilization of type W:NE; this (de)stabilization corresponds to positive Markov (de)stabilization. Propositions \ref{prop:trans} and \ref{prop:braid} give the result.
\end{proof}

\subsection{Classical invariants and grid diagrams}

As in the previous subsection, let $\mathcal{G}$ denote a grid diagram for a link $K\subset L(p,q)$, and let $L$ and $T$ denote the induced Legendrian and transverse representatives. Cornwell \cite{BTI} has derived combinatorial formulas for the classical invariants of $L$ and $T$ coming from the grid diagram $\mathcal{G}$. These formulas generalize the well known formulas for classical invariants of Legendrian and transverse links in the tight three sphere coming from front projections.

Let $P$ denote a toroidal front projection of $K$ as in Proposition \ref{prop:front}. Let $w$ denote the writhe of this projection. Let $m$ denote the algebraic intersection of $\alpha _0$ with $\lambda$, and $l$ the algebraic intersection of $\beta_0$ with $\lambda$. There will be some cusps in the projection, let $c_d$ denote the number of downward oriented cusps, $c_u$ the number of upward oriented cusps, and $c$ the total number of cusps. 

\begin{proposition} (Propositions 3.2, 3.6, and Corollary 3.7 of \cite{BTI})
\label{prop:classical}
Let $L$ and $T$ denote the induced Legendrian and transverse representatives of $K$ in $(L(p,q),\xi_{UT})$. We have the following formulas for the classical invariants of $L$ and $T$:
\begin{align*}
tb_\mathbb{Q} (L) = w - \frac{c}{2} - \frac{ml}{p}\quad\quad\quad
rot_\mathbb{Q} (L) = \frac{1}{2} (c_d - c_u) - \frac {l-m}{p}\\
sl_\mathbb{Q} (T) = w - c_d - \frac{ml +(m-l)}{p}
\end{align*}
\end{proposition}
Let $L=L_1\cup\dots \cup L_l$ and 
let $w^i, c_d ^i, c_u ^i, l^i$ and $m^i$ denote the contributions to $w,c_d,c_u,l$ and $m$ coming from the $i^{th}$ component $L_i\subset L$. The proof of Proposition \ref{prop:classical} in fact shows us that
\begin{align*}
rot^i_\mathbb{Q} (L) = \frac{1}{2} (c_d^i - c_u ^i) - \frac {l^i-m^i}{p}\\
sl^i_\mathbb{Q} (T) = w^i - c_d^i - \frac{m^il^i +(m^i-l^i)}{p}.
\end{align*}

%Consider the projection $\lambda$ of $K$ onto the Heegaard torus constructed from $\mathcal{G}$ in the proof of Proposition \ref{prop:braiding}, note all horizontal arcs have only undercrossings. We may smooth all corners of this rectilinear projection, without introducing any cusps, and obtain a toroidal front projection. For this projection the formulas are particularly simple. 

Let $\pi :(S^3,\xi_{std})\to (L(p,q),\xi_{UT})$ denote the contact universal cover.
By taking $p$ copies of a grid diagram $\mathcal{G}$ for $K\subset L(p,q)$ and stacking them vertically (see Figure \ref{fig:stacking}), we obtain a grid diagram $\mathcal{G}'$ for $K' \subset S^3$. Let $L, T, L'$, and $T'$ denote the Legendrian and transverse representatives induced by $\mathcal{G}$ and $\mathcal{G}'$, respectively. By virtue of how a toroidal front projection induces Legendrian and transverse representatives (Proposition \ref{prop:front}) we have that $L' = \pi^{-1} (L)$ and $T' = \pi^{-1} (T)$. 

\begin{figure}[h]
\def\svgwidth{250pt}
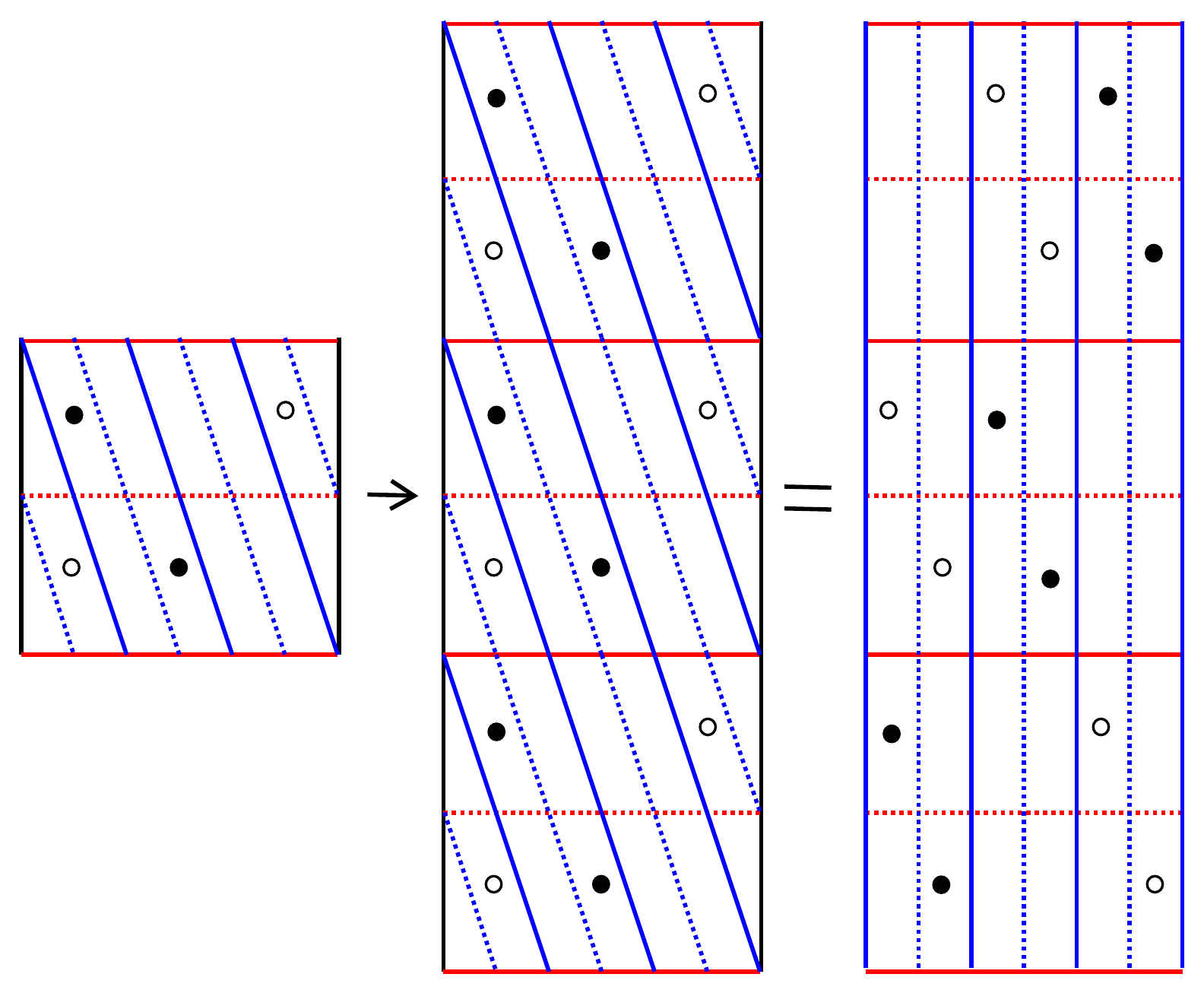
\caption{Stacking an index two diagram $\mathcal{G}$ for a link $K\subset L(3,1)$ to obtain an index six diagram $\mathcal{G}'$ for $K'\subset S^3$. The solid and hollow dots depict $\bold{w}$ and $\bold{z}$ basepoints, respectively.}
\label{fig:stacking}
\end{figure}

The writhe and the number of downward or upward oriented cusps are all multiplied by $p$ when passing from $\mathcal{G}$ to $\mathcal{G}'$; the quantities $m$ and $l$ are preserved. Combining these facts with the formulas of Proposition \ref{prop:classical} it is easy to see how the classical invariants behave under this contact universal cover:

\begin{lemma}
\label{lem:cover}
Let $\pi :(S^3,\xi_{std})\to (L(p,q),\xi_{UT})$ denote the contact universal cover. Let $L, T \subset (L(p,q),\xi_{UT})$ be Legendrian and tranverse links. Let $L' = \pi^{-1} (L)$ and $T' = \pi^{-1} (T)$. Then we have the following:
\begin{align*}
tb_\mathbb{Q} (L)  = \frac{1}{p} tb(L')\quad\quad\quad
rot_\mathbb{Q} (L) = \frac{1}{p} rot(L')\quad\quad\quad
sl_\mathbb{Q} (T) = \frac{1}{p} sl(T')
\end{align*}
and 
\begin{align*}
rot^i_\mathbb{Q} (L) = \frac{1}{p} rot^i(L')\quad\quad\quad
sl^i_\mathbb{Q} (T) = \frac{1}{p} sl^i(T')
\end{align*}
where the lift $L'$ has been partitioned into sublinks $L_1'\cup\dots \cup L_l '$.
\end{lemma}

%The above Lemma holds for more general contact covers (see Lemma 1.1 of \cite{RLCG}), but we've no need for this. 

If $T$ is the positive transverse push-off of a null-homologous Legendrian link $L$, it is well known that $sl(T)=tb(L)-rot(L)$. By the above Lemma, this equality holds for links in universally tight lens spaces as well.

Recall that $\delta^{q/p}$ denotes the counterclockwise $2\pi q/p$ boundary twist on the disk, which is the monodromy of a rational open book supporting $(L(p,q),\xi_{UT})$. If $\beta\in B_n$ is an element of the braid group, then we may consider the corresponding braid $\beta \circ \delta^{q/p}$ about this rational open book. The closure of such a braid is naturally a transverse link, just as in the integral case. We will often not distinguish the braid from its closure.

Let $\Delta_n\in B_n$ denote the Garside element
\[
\Delta_n = (\sigma_1 \dots \sigma_{n-1})(\sigma_1,\dots,\sigma_{n-2})\dots(\sigma_1\sigma_2)(\sigma_1),
\]
which has square
\[
\Delta_n ^2 = (\sigma_1,\dots\sigma_{n-1})^n
\]
the full twist on $n$ strands. Recall that the self-linking number of a braid $\beta\in B_n$ is given by $w(\beta)-n$, where $w(\beta)$ is the writhe.

\begin{lemma}
\label{lem:QSL}
The rational self-linking number of $\beta\circ\delta^{q/p}$ in $(L(p,q),\xi_{UT})$ is given by 
\[
sl_{\mathbb{Q}}(\beta\circ\delta^{q/p})= w(\beta) +\frac{1}{p}(qn^2 -qn - n).
\]
\end{lemma}

\begin{proof}
Consider the contact universal cover $\pi :(S^3,\xi_{std})\to (L(p,q),\xi_{UT})$. The braid $\beta\circ\delta^{q/p}$ lifts to the braid $\beta^p\circ \Delta_n ^{2q}$. By Lemma \ref{lem:cover} we have that 
\[
sl_{\mathbb{Q}}(\beta\circ\delta^{q/p}) = \frac{1}{p}(sl(\beta^p\circ \Delta_n ^{2q}))= \frac{1}{p}(w(\beta^p\circ \Delta_n ^{2q}) - n) = w(\beta) +\frac{1}{p}(qn^2 -qn - n).
\]
\end{proof}

\subsection{Dual Grid diagrams}
\label{subsec:dual}
We describe how to \emph{dualize} a grid diagram encoding a link to obtain one encoding the link's mirror. Our conventions differ from those of \cite{lensgridleg}.

Suppose $\mathcal{G}$ is a grid diagram encoding $K\subset L(p,q)$.
Reflecting $\mathcal{G}$ about the horizontal line $y=1/2$ and exchanging the sets of $\bold{z}$ and $\bold{w}$ basepoints gives rise to a Heegaard diagram $H$ encoding the mirror $K\subset -L(p,q) = L(p,p-q)$. 
The $\boldsymbol{\beta}$-curves are now positively sloped, so $H$ is not a grid diagram.
Performing a shear homeomorphism to $T^2$, so that the $\boldsymbol{\beta}$-curves have slope $-\frac{p}{p-q}$ gives a grid diagram $\mathcal{G}_*$.

Suppose that the grid diagrams $\mathcal{G}$ and its dual $\mathcal{G}_*$ give rise to Legendrians $L_0$ and $L_1$, respectively. Topologically, $L_1$ is the mirror of $L_0$.  Let $g$ denote the index of $\mathcal{G}$.
\begin{proposition} (compare with Proposition 6.9 of \cite{lensgridleg})
\label{prop:index}
\[
tb_{\mathbb{Q}}(L_0)+tb_{\mathbb{Q}}(L_1) = -g
\]
\end{proposition}
\begin{proof}
Let $P$ denote some rectilinear projection of $L_0$ coming from $\mathcal{G}$, and let $w,l,m$ and $c$ be as in Proposition \ref{prop:classical}.
Let $P_*$ be the rectilinear projection of $L_1$ obtained by reflecting $P$ about $y=\frac{1}{2}$ and reversing the orientation of the link; $\mathcal{G}_*$ gives rise to $P_*$.
Let $w_*,l_*,m_*$ and $c_*$ be the corresponding numbers for $P_*$.
By the construction of $P_*$ it is clear that
\[
w=-w_*, l=-l_*, \text{ and }m=m_*
\]
Moreover, each basepoint of $\mathcal{G}$ gives rise to a cusp of the toroidal front induced by either $P$ or $P_*$, i.e.
\[
c+c_* = 2g.
\]
Applying Proposition \ref{prop:classical}
\[
tb_{\mathbb{Q}}(L_0)+tb_{\mathbb{Q}}(L_1) = (w-\frac{c}{2}-\frac{ml}{p})+(w_*-\frac{c_*}{2}-\frac{m_*l_*}{p}) = -\frac{c+c_*}{2}=-g.
\]
\end{proof}

\section{knot Floer background}
\label{sec:HFK}
\subsection{Knot Floer Homology}
\label{subsec:HFK}
We provide a brief overview of knot Floer homology, working with $\mathbb{F}=\mathbb{Z}_2$ coefficients throughout the entire paper.

Let $L\subset Y$ be an rationally null-homologous, oriented link in a closed, oriented 3-manifold.
A multi-pointed Heegaard diagram for $(Y,L)$ is an ordered tuple $\mathcal{H} = (\Sigma,\boldsymbol{\alpha},\boldsymbol{\beta},\bold{z},\bold{w}\cup\bold{w}_F)$ where
\begin{itemize}
\item $\Sigma$ is a genus $g$ Riemann surface,
\item $\boldsymbol{\alpha} = \{\alpha_1,\dots,\alpha_{g+m+n-1}\}$ and $\boldsymbol{\beta} = \{\beta_1,\dots,\beta_{g+m+n-1}\}$ are sets of disjoint, simple closed curves on $\Sigma$ such that $\boldsymbol{\alpha}$ and $\boldsymbol{\beta}$ each span half dimensional subsets of $H_1 (\Sigma;\mathbb{Z})$, 
\item $\bold{z}$ and $\bold{w}$ are sets of $m$ \emph{linked} basepoints. Each component of $\Sigma \smallsetminus \{\alpha_1,\dots,\alpha_{g+m-1}\}$ and $\Sigma \smallsetminus \{\beta_1,\dots,\beta_{g+m-1}\}$ contains exactly one element of $\bold{z}$ and one of $\bold{w}$,
\item $\bold{w}_F$ is a set of $n$ \emph{free} basepoints. Every component of $\Sigma \smallsetminus \boldsymbol{\alpha}$ and $\Sigma\smallsetminus\boldsymbol{\beta}$ contains exactly one element of $\bold{w}\cup\bold{w}_F$.
\end{itemize}

$Y$ is specified by the Heegaard diagram $(\Sigma,\boldsymbol{\alpha},\boldsymbol{\beta})$. The link $L$ is obtained in the usual way as follows. Connect the points in $\bold{z}$ to those in $\bold{w}$ with $m$ oriented, disjoint, embedded arcs in $\Sigma \smallsetminus \boldsymbol{\alpha}$; form $\{\gamma^\alpha _1,\dots,\gamma^\alpha _{m}\}$ by pushing the interiors of the arcs into the $\boldsymbol{\alpha}$ handlebody. Likewise, connect the points in $\bold{w}$ to those in $\bold{z}$ with $m$ oriented, disjoint, embedded arcs in $\Sigma \smallsetminus \boldsymbol{\beta}$; form $\{\gamma^\beta _1,\dots,\gamma^\beta _{m}\}$ by pushing the interiors of the arcs into the $\boldsymbol{\beta}$ handlebody. The union
\[
L=\gamma^\alpha _1\cup\dots\cup\gamma^\alpha _{m}\cup\gamma^\beta _1\cup\dots\cup\gamma^\beta _{m}
\]
forms the link. %One obtains a longitude for the link $\lambda$ by projecting onto $\Sigma$.

To each $w\in \bold{w}\cup\bold{w}_F$ we associate a formal variable $U_w$. Consider the totally real tori $\mathbb{T}_{\boldsymbol{\alpha}} = \alpha_1 \times \dots \times \alpha_{g+m+n-1}$ and $\mathbb{T}_{\boldsymbol{\beta}}= \beta_1 \times \dots \times \beta_{g+m+n-1}$ in the symmetric product $Sym^{g+m+n-1}(\Sigma)$. $CFK^- (\mathcal{H})$, the knot Floer complex, is a free $\mathbb{F} [\{ U_w \} _{w\in \bold{w}\cup\bold{w}_F} ]$-module generated by the intersections of $\mathbb{T}_{\boldsymbol{\alpha}}$ with $\mathbb{T}_{\boldsymbol{\beta}}$.

Let $\bold{x},\bold{y}\in\mathbb{T}_{\boldsymbol{\alpha}}\cap \mathbb{T}_{\boldsymbol{\beta}}$, $\phi\in\pi_2(\bold{x},\bold{y})$ be a Whitney disk, and suppose there is a suitable path of almost complex structures on $Sym^{g+m+n-1}(\Sigma)$; we denote the moduli space of pseudo-holomorphic representatives of $\phi$ by $\mathcal{M} (\phi)$. The formal dimension of $\mathcal{M}(\phi)$ is given by the Maslov index $\mu (\phi)$. $\widehat{\mathcal{M}} (\phi)$ denotes the quotient of $\mathcal{M}(\phi)$ by the natural translation action of $\mathbb{R}$. 
For $p\in \Sigma$, we let 
 $n_p (\phi)$ denote the multiplicity of $D(\phi)$, the domain of $\phi$, at the point $p$. For a finite set of points $\bold{p} = \{p_1,\dots, p_k\}\subset \Sigma$, $n_\bold{p} (\phi)$ denotes the sum $n_{p_1} (\phi) +\dots +n_{p_{k}}(\phi)$.

\subsection{$Spin^C$ structures}
The correspondence between $Spin^C$-structures and homology classes of non-vanishing vector fields for 3-manifolds was first introduced by Turaev \cite{turaev}.

Ozsv\'{a}th and Szab\'{o} \cite{relspinc} generalized this construction to 3-manifolds having torus boundary components. If $L\subset Y$ is a link, a \emph{relative} $Spin^C$-\emph{structure} is a homology class of non-vanishing vector field $v$ on $Y\smallsetminus \nu(L)$ such that the vector field $v$ points outwards along the boundary of $Y\smallsetminus \nu(L)$; we denote the set of such relative $Spin^C$-structures by $Spin^C (Y,L)$. There is an affine correspondence between $Spin^C (Y,L)$ and classes of $H^2 (Y,L;\mathbb{Z})$ which is analogous to the correspondence between $Spin^C (Y)$ and $H^2 (Y;\mathbb{Z})$; in particular there is an action of relative cohomology classes on relative $Spin^C$-structures.

There is a filling map
\[
G_{Y,L}: Spin^C (Y,L) \to Spin^C (Y)
\]
defined as follows. Let $v_L$ be a vector field on $Y\smallsetminus \nu(L)$ representing $\mathfrak{s}_L \in Spin^C(Y,L)$. Identifying $\nu (L)$ with $L\times D^2$, it is easy to see that there is a unique vector field $v_{\nu(L)}$, up to homotopy, on $\nu(L)$ which points inward along the boundary, is everywhere transverse to the $D^2$ factor, and has $L$ as an oriented closed orbit. Let $v$ denote the vector field on $Y$ obtained by gluing $v_L$ to $v_{\nu(L)}$.
We define $G_{Y,L} (\mathfrak{s}_L)$ to be the homology class of $v$.

This filling map is equivariant with respect to the action of cohomology, meaning that if $\eta \in H^2 (Y,L;\mathbb{Z})$ and $i:Y\smallsetminus L \to Y$ is the inclusion map, then 
\[
G_{Y,L}(\mathfrak{s}_L +\eta) = G_{Y,L} (\mathfrak{s}_L) + i^* \eta.
\]

Let $\mathcal{H} = (\Sigma,\boldsymbol{\alpha},\boldsymbol{\beta},\bold{z},\bold{w}\cup\bold{w}_F)$ be a Heegaard diagram encoding $(Y,L)$. Ozsv\'{a}th and Szab\'{o} define a map 
\[
\mathfrak{s}_{z,w} : \mathbb{T}_{\boldsymbol{\alpha}}\cap \mathbb{T}_{\boldsymbol{\beta}}\to Spin^C(Y,L)
\]
by explicitly constructing a vector field representing $\mathfrak{s}_{z,w}(\bold{x})$, the construction is similar to that of the map
\[
\mathfrak{s}_{z} : \mathbb{T}_{\boldsymbol{\alpha}}\cap \mathbb{T}_{\boldsymbol{\beta}}\to Spin^C(Y)
\]
in their earlier work.

These maps behave nicely with respect to the filling map defined above, in particular for a generator $\bold{x}\in  \mathbb{T}_{\boldsymbol{\alpha}}\cap \mathbb{T}_{\boldsymbol{\beta}}$
\[
G_{Y,L} (\mathfrak{s}_{z,w} (\bold{x})) = \mathfrak{s}_{z} (\bold{x}).
\]

The complex $CFK^-(\mathcal{H})$ splits as a direct sum over both $Spin^C$ and relative $Spin^C$-structures. 
The relative homological grading is called the Maslov grading. It is specified by 
 \[
 M(\bold{x})-M(\bold{y}) = \mu(\phi)-2n_{\bold{w}\cup \bold{w}_F}(\phi)
 \]
 for $\bold{x},\bold{y}\in\mathbb{T}_{\boldsymbol{\alpha}}\cap\mathbb{T}_{\boldsymbol{\beta}}$ and any Whitney disk $\phi\in \pi_2 (\bold{x},\bold{y})$, and the fact that multiplication by each of the formal variables $U_w$ lowers Maslov grading by two. 
If we are working in a summand corresponding to a torsion $Spin^C$ structure $\mathfrak{s}\in Spin^C (Y)$, the relative Maslov grading can be enhanced to an absolute $\mathbb{Q}$ grading \cite{absgrading}.

\subsection{The Alexander grading}
The set of relative $Spin^C$-structures determines a filtration of the chain complex $CFK^-(\mathcal{H})$ called the Alexander filtration. If $L$ is null-homologous, the filtration levels can be identified with the integers via the Alexander grading (\cite{holknots},\cite{ras}). Ni \cite{coversalex} later generalized this construction to rationally null-homologous links.

\begin{definition}
Let $L= L_1\cup\dots\cup L_l \subset Y$ be a rationally null-homologous link represented by a multi-pointed Heegaard diagram $\mathcal{H} = (\Sigma,\boldsymbol{\alpha},\boldsymbol{\beta},\bold{z},\bold{w}\cup\bold{w}_F)$. Let $F$ be a rational Seifert surface for $L$. For a generator $\bold{x}\in\mathbb{T}_{\boldsymbol{\alpha}}\cap\mathbb{T}_{\boldsymbol{\beta}}$ the \emph{$i^{th}$ Alexander grading of $\bold{x}$ with respect to F} is given by
%\begin{align*}
%A_{L_i} ^F (\bold{x}) =\frac{1}{2[\mu_i]\cdot[F]} \Big(<c_1(\mathfrak{s}_{z,w}(\bold{x})) - (2n_i - \frac{1}{r_i})PD([\mu_i]),[F]>\Big)\\ =\frac{<c_1(\mathfrak{s}_{z,w}(\bold{x})),[F]>}{2[\mu_i]\cdot[F]} - (n_i - \frac{1}{2r_i}).
%\end{align*}
\begin{align*}
A_{L_i} ^F (\bold{x}) =\frac{1}{2[\mu_i]\cdot[F]} \Big(<c_1(\mathfrak{s}_{z,w}(\bold{x})) - (2n_i - 1)PD([\mu_i]),[F]>\Big)\\ =\frac{<c_1(\mathfrak{s}_{z,w}(\bold{x})),[F]>}{2[\mu_i]\cdot[F]} - (n_i - \frac{1}{2}).
\end{align*}
%where $\mu_i$ is an oriented meridian for $L_i$, $n_i$ is the number of basepoint pairs used to encode $L_i$, and $r_i$ is the order of $[L_i]$ in $H_1 (Y;\mathbb{Z})$.
where $\mu_i$ is an oriented meridian for $L_i$ and $n_i$ is the number of basepoint pairs used to encode $L_i$.
\end{definition}
%\begin{remark}
%The above definition agrees with definitions found the literature up to some constant shift (or scaling and a constant shift, it is always some linear function of the Chern class evaluation). Our convention is more natural when considering covers, see Lemma \ref{lem:Acovers}. With this convention the Alexander grading is no longer symmetric; if $A_{max}$ ($A_{min}$) denote the maximal (minimial) Alexander gradings among all non-zero classes in $\widehat{HFK}(Y,K)$ then we do not have that $A_{max} = - A_{min}$.
%\end{remark}

The Alexander grading only depends on the Seifert surface through its relative homology class. In this paper we will primarily be studying the case that $Y$ is a rational homology 3 sphere where the choice of $F$ is irrelevant, or we will be studying the Alexander grading induced by a binding of a rational open book, where the fiber will be the preferred rational Seifert surface, so we often suppress $F$ from the notation.

Multiplication by $U_w$, for any $w\in \bold{w}_{L_i}$, lowers $A_{L_i}$ by one, multiplication by the other formal variables does not change $A_{L_i}$. 
We denote the sum $A_{L_1} (\bold{x}) + \dots + A_{L_l}(\bold{x})$ by $A_L(\bold{x})$. The bigrading on the knot Floer homology of a link is comprised of the Maslov and collapsed Alexander gradings.

\begin{definition}
Let $K\subset Y$ be a rationally null-homologous knot. 

We define the \emph{complexity} of $K$ to be 
\[
||K|| = inf \Big{\{} \frac{-\chi(F)}{2[\mu]\cdot [F]}\Big{\}}
\]
where the infimum is taken over all rational Seifert surfaces $F$ for $K$ having no sphere components.
\end{definition}

Ni \cite{coversalex} has proven that knot Floer homology of links in rational homology spheres detects the Thurston norm of the link complement.
This result specializes to the following theorem for knots; this is a generalization of the analogous Theorem in the $S^3$ setting due to Ozsv\'{a}th and Szab\'{o} \cite{genusdetection}.

%Note that this extends the definition of the Alexander grading for null-homologous links. Moreover if $K$ is knot in a $\mathbb{Q}HS^3$ then the primary result of \cite{coversalex} specializes, in our conventions, to
\begin{theorem}
\label{thm:genusdetection}
Let $K$ be a knot in a rational homology sphere $Y$.
Let $A_{max}$ denote the maximal Alexander grading among all non-zero classes in $\widehat{HFK}(Y,K)$. Then
\[
||K|| = A_{max}- \frac{1}{2}
\]
\end{theorem}

The following notion was introduced in \cite{QOB} and is very useful for studying the relative Alexander grading.
\begin{definition}
Let $L_1\cup \dots\cup L_l = L \subset Y$ be an $l$ component link, and let 
\[
(\Sigma, \boldsymbol{\alpha},\boldsymbol{\beta},\bold{z}_{1}\cup\dots\cup\bold{z}_{l},\bold{w}_{1}\cup\dots\cup\bold{w}_{l})
\]
be a Heegaard diagram for $(Y,L)$ where the basepoints $\bold{z}_{i}$ and $\bold{w}_{i}$ encode the link component $L_i$. Suppose that $[L_i]$ has order $r$ in $H_1(Y)$. Let $\lambda_i \subset \Sigma$ be a longitude for $L_i$ constructed as above. Let $D_1,\dots,D_r$ denote the closures of components of $\Sigma \setminus (\lambda_i\cup\boldsymbol{\alpha}\cup\boldsymbol{\beta})$. A \emph{relative periodic domain} is a 2-chain $\mathcal{P} = \Sigma a_i D_i$, whose boundary satisfies
\[
\partial \mathcal{P} = r\lambda _i +\sum n_i \alpha _i + \sum m_i \beta_i.
\]
\end{definition}

A relative periodic domain $\mathcal{P}$ naturally corresponds to a homology class in $H_2 (Y\smallsetminus \nu (L_i),\partial (Y\smallsetminus \nu (L_i)))$.

\begin{lemma} (see Lemma 2.3 of \cite{QOB})
\label{lemma:relperiodic}
Let $L_i\subset L$ be as in the definition above. Let $\mathcal{P}$ be a relative periodic domain whose homology class agrees with that of some rational Seifert surface $F$ for $L_i$. For $\bold{x},\bold{y}\in \mathbb{T}_{\boldsymbol{\alpha}}\cap \mathbb{T}_{\boldsymbol{\beta}}$, we have
\[
A_{L_i}(\bold{x})-A_{L_i}(\bold{y}) =  \frac{1}{r} (n_\bold{x} (\mathcal{P})-n_\bold{y}(\mathcal{P}))
\]
where the Alexander grading above is defined using the surface $F$.
\end{lemma}

Ni has shown the relative Alexander grading to behave nicely under covers.
Let $\mathcal{H} = (\Sigma,\boldsymbol{\alpha},\boldsymbol{\beta},\bold{z},\bold{w})$ be a Heegaard diagram for $L_1\cup \dots\cup L_l = L \subset Y$. If $Y$ is a $\mathbb{Q}HS^3$, the universal cover $\pi: S^3\to Y$ is of some finite index $p$. We may take a $p$-fold cover of $\mathcal{H}$ to get a diagram $\widetilde{\mathcal{H}} = (\widetilde{\Sigma},\widetilde{\boldsymbol{\alpha}},\widetilde{\boldsymbol{\beta}},\widetilde{\bold{z}},\widetilde{\bold{w}})$ for $(S^3,\widetilde{L})$, where $\widetilde{L} = \pi ^{-1} (L)$. We let $\widetilde{L}_i = \pi^{-1}(L_i)$, note that $\widetilde{L}_i$ is a link having $p/r_i$ components, where $r_i$ is again the order of $[L_i]$ in $H_1 (Y;\mathbb{Z})$. 
A generator $\bold{x}\in \mathbb{T}_{\boldsymbol{\alpha}}\cap \mathbb{T}_{\boldsymbol{\beta}}$ lifts to a generator $\widetilde{\bold{x}}\in \mathbb{T}_{\boldsymbol{\alpha}}\cap \mathbb{T}_{\boldsymbol{\beta}}$. 

\begin{lemma}(see Lemma 4.2 of \cite{coversalex})
For $\bold{x},\bold{y}\in \mathbb{T}_{\boldsymbol{\alpha}}\cap \mathbb{T}_{\boldsymbol{\beta}}$,
\[
A_{L_i}(\bold{x})-A_{L_i}(\bold{y}) = \frac{1}{p} (A_{\widetilde{L_i}}(\widetilde{\bold{x}}) - A_{\widetilde{L_i}}(\widetilde{\bold{y}}))
\]
\end{lemma}

We will need to understand the behavior of the absolute grading under covers.

\begin{lemma}
\label{lem:Acovers}
\[
A_{L_i}(\bold{x})=\frac{1}{p} A_{\widetilde{L_i}}(\widetilde{\bold{x}}) +\frac{1}{2}\Big(1-\frac{1}{r_i}\Big)
\]
where $r_i$ denotes the order of $L_i$ in $H_1(Y;\mathbb{Z})$.
\end{lemma}
\begin{proof}
Suppose that $F$ is a rational Seifert surface for the link $L$. We may use $\widetilde{F} = \pi ^{-1} (F)$ to compute the Alexander grading with respect $\widetilde{L}$, even though it may not be a Seifert surface.

By construction of the relative $Spin^C$-structures it is clear that 
if $v$ is a vector field representing $\mathfrak{s}_{z,w} (\bold{x})$, then we may pull back $v$ to a vector field $\pi ^* v$ on $S^3 \smallsetminus \widetilde{L}$ representing $\mathfrak{s}_{\widetilde{z},\widetilde{w}} (\widetilde{\bold{x}})$.

It follows that $\pi ^* (c_1(\mathfrak{s}_{z,w} (\bold{x}))) = c_1 (\mathfrak{s}_{\widetilde{z},\widetilde{w}} (\widetilde{\bold{x}}))$ and 
\begin{align*}
<c_1 (\mathfrak{s}_{\widetilde{z},\widetilde{w}} (\widetilde{\bold{x}})), [\widetilde{F}]>\\ = <\pi ^* (c_1(\mathfrak{s}_{z,w} (\bold{x}))), [\pi ^{-1}(F)]>\\= p <c_1(\mathfrak{s}_{z,w} (\bold{x})),[F]>.
\end{align*}
Let $K_1\dots K_m$ be the components of $\widetilde{L}_i$, where $m= p/r_i$. Let $m_j$ denote a meridian for $K_j$, and $\widetilde{n}_j$ denote the number of basepoint pairs encoding $K_j$. Note that
\begin{align*}
p[\mu_i]\cdot[F] = [\pi^{-1}(\mu_i)]\cdot[\widetilde{F}] = r_i[m_1]\cdot[\widetilde{F}].
\end{align*}

We can now evaluate,
\begin{align*}
A_{\widetilde{L_i}}(\widetilde{\bold{x}}) = \sum_{j=1}^{p/r_i} A_{K_j}(\widetilde{\bold{x}}) \\
= \sum_{j=1}^{p/r_i} \Big(\frac{<c_1 (\mathfrak{s}_{\widetilde{z},\widetilde{w}} (\widetilde{\bold{x}})),[\widetilde{F}]>}{2[m_j]\cdot[\widetilde{F}]} - (\widetilde{n}_j - 1/2)\Big)\\
= \sum_{j=1}^{p/r_i} \Big(\frac{p <c_1(\mathfrak{s}_{z,w} (\bold{x})),[F]>}{2[m_j]\cdot[\widetilde{F}]}\Big) - (n_ip -\frac{p}{2r_i})\\
= p\Big( \frac{p}{r_i} \frac{<c_1(\mathfrak{s}_{z,w} (\bold{x})),[F]>}{2[m_1]\cdot[\widetilde{F}]} - (n_i - \frac{1}{2r_i})\Big)\\
= p\Big(\frac{<c_1(\mathfrak{s}_{z,w} (\bold{x})),[F]>}{2[\mu_i]\cdot[F]} - (n_i - \frac{1}{2r_i})\Big)\\
= p\Big(\frac{<c_1(\mathfrak{s}_{z,w} (\bold{x})),[F]>}{2[\mu_i]\cdot[F]} - (n_i - \frac{1}{2}) - \frac{1}{2}+\frac{1}{2r_i}\Big)\\
= p \Big(A_{L_i} (\bold{x}) - \frac{1}{2}(1-\frac{1}{r_i})\Big).
\end{align*}
\end{proof}

\subsection{Knot Floer complexes and stabilizations}

The differential $\partial ^- : CFK^- (\mathcal{H})\to CFK^- (\mathcal{H})$ is defined as follows on generators
\[
\partial ^- (\bold{x}) := \sum\limits_{\bold{y}\in \mathbb{T}_{\boldsymbol{\alpha}}\cap \mathbb{T}_{\boldsymbol{\beta}}}  \sum_{\substack{\phi\in\pi_2 (\bold{x},\bold{y})\\ \mu(\phi)=1\\ n_z(\phi) = 0\ \  \forall z\in\bold{z}}}  \# \widehat{\mathcal{M}}(\phi) \cdot \prod\limits_{w\in \bold{w}\cup \bold{w}_F} U_w ^{n_w (\phi)}\cdot \bold{y},
\]
and extends linearly to the entire complex. We define the minus version of knot Floer homology to be
\[
HFK^- (Y,L):= HFK^- (\mathcal{H}) = H_* ( CFK^- (\mathcal{H}),\partial^-).
\]

If $w$ and $w'$ are in the same $\bold{w}_{L_i}$ for some $i$, the formal variables $U_w$ and $U_{w'}$ act identically on $HFK^- (Y,L)$. Each of the formal variables corresponding to free basepoints also act identically on $HFK^- (Y,L)$. Letting $U_i$ denote the action of $U_w$ for $w\in \bold{w}_{L_i}$, and $w_f \in \bold{w}_F$ be some free basepoint, one can show that $HFK^- (Y,L)$ is an invariant of $L\subset Y$, which is well defined up to graded $\mathbb{F}[U_1,\dots,U_l,U_{w_f}]$-module isomorphism.

$\widehat{CFK}(\mathcal{H})$ is the chain complex obtained by setting $U_w=0 $ for exactly on $w$ in each $\bold{w}_{L_i}$. We let $\widehat{\partial}$ denote the induced differential on $\widehat{CFK}(\mathcal{H})$, and let
\[
p:CFK^- (\mathcal{H})\to \widehat{CFK}(\mathcal{H})
\]
denote the natural projection. The hat version of knot Floer homology, 
\[
\widehat{HFK}(Y,L):=\widehat{HFK}(\mathcal{H}):=H_* (\widehat{CFK}(\mathcal {H}),\widehat{\partial}),
\]
is an invariant of $(Y,L)$ up to graded $\mathbb{F}$-module isomorphism.

Setting each $U_w = 0$ for all $w\in \bold{w}_F$ one obtains another chain complex, $CFK^{-,\bold{w}_F} (\mathcal{H})$. This complex plays a key role in our reformulation of the transverse invariant in subsequent sections. The homology
\[
HFK^{-,n} (Y,L):= HFK^{-,n} (\mathcal{H}):= H_* (CFK^{-,\bold{w}_F} (\mathcal{H}),\partial ^-)
\]
is an invariant of $(Y,L)$ and the number, $n$, of free basepoints up to graded $\mathbb{F}[U_1,\dots ,U_l]$-module isomorphism. 

Any pair of multi-pointed Heegaard diagrams for $(Y,L)$ is related by a sequence of Heegaard moves in the complement of all basepoints. The moves are isotopy, handleslide, index 1/2 (de)stabilization, linked index 0/3 (de)stabilization, and free index 0/3 (de)stabilization. A pair of such diagrams with equal number of free basepoints may be related by a sequence of Heegaard moves not including free 0/3 (de)stabilization. 

Isotopies and handleslides induce chain maps, via pseudo-holomorphic triangle counts, which induce isomorphisms on homology. Index 1/2 (de)stabilization induces an isomorphism of chain complexes. We describe the maps associated to linked and free 0/3 (de)stabilizations and their relationship to certain basepoint actions on the complex.

Suppose that $D$ is a region of $\Sigma \smallsetminus \boldsymbol{\beta}$ containing some $z\in \bold{z}$ and $w\in\bold{w}$. Performing a linked index 0/3 stabilization consists of adding a basepoints $z'$ to $\bold{z}$ and $w'$ to $\bold{w}$, and curves $\alpha '$ to $\boldsymbol{\alpha}$ and $\beta '$ to $\boldsymbol{\beta}$ as depicted in Figure \ref{fig:linked03}. The two intersections of $\alpha ' $ and $\beta '$, denoted $x'$ and $y'$, must be in the region of $\Sigma \smallsetminus \boldsymbol{\alpha} \smallsetminus \boldsymbol{\beta}$ which contains $z$. 

\begin{figure}[h]
\def\svgwidth{300pt}
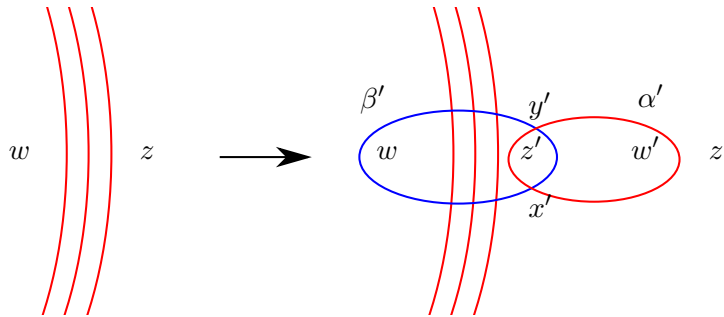
\caption{Before and after a linked 0/3 stabilization.}
\label{fig:linked03}
\end{figure}

Let $CFK^{-,n}(\mathcal{H}')_2$ be the subcomplex of $CFK^{-,n}(\mathcal{H}')$ generated by elements of the form $\bold{x}\cup\{y'\}$, and let $CFK^{-,n}(\mathcal{H}')_1$ denote the quotient complex generated by elements of the form $\bold{x}\cup\{x'\}$, where $\bold{x}\in \mathbb{T}_{\boldsymbol{\alpha}}\cap \mathbb{T}_{\boldsymbol{\beta}}$. Define $f:CFK^{-,n}(\mathcal{H}')_1\to CFK^{-,n}(\mathcal{H}')_2$ by 
\[
f(\bold{x}\cup \{x'\}) = (U_w + U_{w'})(\bold{x}\cup \{y'\}).
\]
$CFK^{-,n} (\mathcal{H}')$ is isomorphic to the mapping cone of $f$, and it follows that the map from $CFK^{-,n}(\mathcal{H})$ to $CFK^{-,n}(\mathcal{H})$ defined on generators by sending $\bold{x}$ to $\bold{x}\cup \{y'\}$ induces an isomorphism on homology. Linked index 0/3 destablization induces the inverse of this isomorphism.

Given any $z'\in \bold{z}$ we define a chain map $\Psi_{z'}: CFK^{-,n}(\mathcal{H})\to CFK^{-,n}(\mathcal{H})$ by counting holomorphic disks which pass exactly once through $z'$.
On generators the map is defined as follows:
\[
\Psi_{z'}(\bold{x}) := \sum\limits_{\bold{y}\in \mathbb{T}_{\boldsymbol{\alpha}}\cap \mathbb{T}_{\boldsymbol{\beta}}}  \sum_{\substack{\phi\in\pi_2 (\bold{x},\bold{y})\\ \mu(\phi)=1\\ n_{z'}(\phi) =1\\ n_z(\phi) = 0\ \  \forall z\in\bold{z} \smallsetminus \{z'\}}}  \# \widehat{\mathcal{M}}(\phi) \cdot \prod\limits_{w\in \bold{w}\cup \bold{w}_F} U_w ^{n_w (\phi)}\cdot \bold{y}.
\]
Studying degenerations of holomorphic disks shows that $\Psi_{z'}$ is a chain map; let $\psi_{z'}$ denote the induced map on homology. More degeneration arguments involving disks show that $\psi_{z'}^2 = 0$ and if $z'\ne z \in \bold{z}$ then $\psi_{z'}\psi_{z} = \psi_{z}\psi_{z'}$. Further standard degeneration arguments involving triangles show that $\psi_{z'}$ commutes with the isomorphisms associated to isotopies and handleslides. $\psi_{z'}$ commutes with the map associated to free 0/3 (de)stabilization, and linked 0/3 (de)stabilization so long as $z'$ is not the basepoint being added (or removed).

 Note that for the diagram $\mathcal{H}'$, obtained from $\mathcal{H}$ by linked 0/3 stabilization, we have that $CFK^{-,n} (\mathcal{H}')_1 = ker\  \Psi_{z'}$ and $CFK^{-,n} (\mathcal{H}')_2 = coker\ \Psi_{z'}$. Thus the summand $\bigcap_{z\in \bold{z}} coker(\psi_{z})$ is preserved by the isomorphism induced by any Heegaard move.

Free index 0/3 stabilization consists of adding a free basepoint $w'$ to $\bold{w}_F$, one curve $\alpha '$ to $\boldsymbol{\alpha}$ and one curve $\beta '$ to $\boldsymbol{\beta}$, in a region of $\Sigma\smallsetminus\boldsymbol{\alpha}\smallsetminus\boldsymbol{\beta}$ containing a point of $\bold{z}$, as depicted in Figure \ref{fig:free03}, to obtain a new diagram $\mathcal{H}'$. We say that $\alpha '$ and $\beta '$ form a \emph{small configuration} about $w'$.

\begin{figure}[h]
\def\svgwidth{200pt}
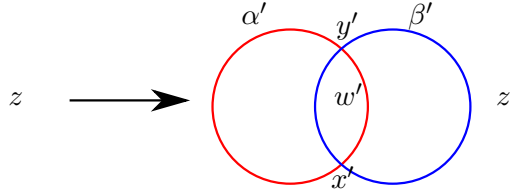
\caption{Before and after a free 0/3 stabilization.}
\label{fig:free03}
\end{figure}

Let $CFK^{-,n+1}(\mathcal{H}')_1$, $CFK^{-,n+1}(\mathcal{H}')_2$, be the subcomplex generated by elements of the form $\bold{x}\cup \{x'\}$, $\bold{x}\cup \{y'\}$, respectively, for $\bold{x}\in \mathbb{T}_{\boldsymbol{\alpha}}\cap \mathbb{T}_{\boldsymbol{\beta}}$.
$CFK^{-,n+1}(\mathcal{H}')$ splits as a direct sum of complexes,
\[
CFK^{-,n+1}(\mathcal{H}') = CFK^{-,n+1}(\mathcal{H}')_2 \oplus CFK^{-,n+1}(\mathcal{H}')_2.
\]
The inclusion $i:CFK^{-,n}(\mathcal{H})\to CFK^{-,n+1}(\mathcal{H}')$ which sends $\bold{x}$ to $\bold{x}\cup \{x'\}$, is an isomorphism from $CFK^{-,n}(\mathcal{H})$ to $CFK^{-,n+1}(\mathcal{H}')_1 [1]$, where the $[1]$ indicates that the Maslov grading has been increased by 1. The projection $j$, sending generators $\bold{x}\cup \{x'\}$ to $\bold{x}$, and all others to zero, restricts to the inverse of $i$ on $CFK^{-,n+1}(\mathcal{H}')_1[1]$.

Given any $w\in \bold{w}_F$ we define a chain map $\Psi_{w}: CFK^{-,n}(\mathcal{H})\to CFK^{-,n}(\mathcal{H})$ by counting holomorphic disks which pass exactly once through $w$.
On generators the map is defined as follows:
\[
\Psi_{w}(\bold{x}) := \sum\limits_{\bold{y}\in \mathbb{T}_{\boldsymbol{\alpha}}\cap \mathbb{T}_{\boldsymbol{\beta}}}  \sum_{\substack{\phi\in\pi_2 (\bold{x},\bold{y})\\ \mu(\phi)=1\\ n_{w}(\phi) =1\\ n_z(\phi) = 0\ \  \forall z\in\bold{z} }}  \# \widehat{\mathcal{M}}(\phi) \cdot \prod\limits_{w\in \bold{w}} U_w ^{n_w (\phi)}\cdot \bold{y}.
\]
$\Psi_{w}$ is a chain map; let $\psi_{w}$ denote the induced map on homology and refer to it as the free basepoint action associated to $w$. Standard degeneration arguments show that the basepoint actions associated to two distinct free basepoints commute, any free basepoint action squares to zero, and that $\psi_w$ will commute with maps induced by all Heegaard moves, including 0/3 free (de)stabilization so long as $w$ is not the free basepoint being added (or removed).

 Note that for the diagram $\mathcal{H}'$, obtained from $\mathcal{H}$ by free 0/3 stabilization, we have that $CFK^{-,n+1} (\mathcal{H}')_1 = coker\  \Psi_{w'}$ and $CFK^{-,n+1} (\mathcal{H}')_2 = ker\ \Psi_{w'}$. Thus the splitting in (1) gives rise to the splitting on homology,
\[
HFK^{-,n+1} (\mathcal{H}') = coker\  \psi_{w'} \oplus ker\  \psi_{w'}.
\]

The inclusion $i_*$ induces, and the projection $j_*$ restricts to, isomorphisms which are inverses of each other:
\[
\begin{split}
i_*:HFK^{-,n}(\mathcal{H})\to coker\ \psi_{w'}[1]&\\
j_*: coker\ \psi_{w'}[1]\to HFK^{-,n}(\mathcal{H}).
\end{split}
\]

Suppose now that $\mathcal{H}'$ is obtained from $\mathcal{H}$ by $k$ free index 0/3 stabilizations. Let $w_1,\dots,w_k$ denote the free basepoints which are added in the stabilizations. Let $i^k$ and $j^k$ denote the obvious compositions of inclusion and projection maps
\[
\begin{split}
i^k:CFK^{-,n}(\mathcal{H})\to CFK^{-,n+k}(\mathcal{H}')&\\
j^k:CFK^{-,n+k}(\mathcal{H}')\to CFK^{-,n}(\mathcal{H}).
\end{split}
\]
The compositions $i^k_{*}$, $j^k_{*}$, induce and restrict to, respectively, isomorphisms which are inverses of each other: 
\[
\begin{split}
i^k_*:HFK^{-,n}(\mathcal{H})\to \Big( \bigcap_{i=1}^k coker\ \psi_{w_i}\Big)[k]&\\
j^k_*: \Big( \bigcap_{i=1}^k coker\ \psi_{w_i}\Big)[k]\to HFK^{-,n}(\mathcal{H}).
\end{split}
\]

\subsection{Combinatorial knot Floer homology}
\label{subsec:combmaps}

Grid diagrams have been used to give a combinatorial definition of link Floer homology, first for links in $S^3$ \cite{MOST}, and subsequently for links in lens spaces \cite{lensgridcomb}. Also see the text \cite{gridhom} for a comprehensive treatment of grid homology for links in $S^3$.

Given a grid diagram $\mathcal{G} = (T^2,\boldsymbol{\alpha},\boldsymbol{\beta},\bold{z},\bold{w})$, consider the diagram $G = (T^2, \boldsymbol{\beta},\boldsymbol{\alpha},\bold{w},\bold{z})$ for $(-L(p,q),K)$. In this section we refer to $G$ as a grid diagram for $K$. 

If $G$ has index $n$, the generators of $CFK^-(G)$ can be identified with $S_n \oplus \mathbb{Z}/p\mathbb{Z}$, where $S_n$ denotes the symmetric group on $n$ elements. The differential on $CFK^-(G)$ is defined by counting certain pseudo-holomorphic disks, c.f. Subsection \ref{subsec:HFK}, for a grid diagram all of the appropriate disks contributing to the differential have domains which are rectangles. Computing the homology $HFK^-(G)$ is a combinatorial task, this is the basic idea behind grid homology. 

\begin{definition}
Fix $\bold{x},\bold{y} \in \mathbb{T}_{\boldsymbol{\beta}}\cap\mathbb{T}_{\boldsymbol{\alpha}}$. A \emph{rectangle} from $\bold{x}$ to $\bold{y}$ is an embedded disk $r\subset T^2$ whose boundary consists of four arcs, each of which lies along some $\boldsymbol{\beta}$ or $\boldsymbol{\alpha}$ curve, satisfying the conditions:
\begin{itemize}
\item Four corners of $p$ are in $\bold{x}\cup\bold{y}$. Moreover $\bold{x}$ and $\bold{y}$ agree away from these four corners.
\item The portion of $\partial p$ along the $\boldsymbol{\alpha}$ curves is an oriented path from $\bold{y}$ to $\bold{x}$.
\end{itemize}
\end{definition}

The set of rectangles from $\bold{x}$ to $\bold{y}$ is denoted $Rect(\bold{x},\bold{y})$, and is either empty or consists of two rectangles. A rectangle $x\in Rect(\bold{x},\bold{y})$ is called empty if its interior is disjoint from $\bold{x}$ and $\bold{y}$. The space of empty rectangles from $\bold{x}$ to $\bold{y}$ is denoted $Rect^\circ (\bold{x},\bold{y})$.

The differential on $CFK^-(G)$ can be expressed as

\[
\partial ^- (\bold{x}) := \sum\limits_{\bold{y}\in \mathbb{T}_{\boldsymbol{\beta}}\cap \mathbb{T}_{\boldsymbol{\alpha}}}  \sum_{\substack{r\in Rect^\circ (\bold{x},\bold{y})\\ r\cap \bold{w} = \emptyset}} U_0 ^{z_0 (r)} U_1^{z_1(r)}\dots U_{n-1}^{z_{n-1}(r)} \cdot \bold{y},
\]
where $z_i(r)$ denotes the intersection number of $z_i$ with $r$. 

As explained in Subsection \ref{subsec:contactgrid}, a grid diagram naturally gives rise to Legendrian and transverse representatives of the link. In Section \ref{sec:invariants} we will use grid diagrams to define invariants of Legendrian and transverse links in universally tight lens spaces, naturally extending the invariants defined in \cite{grid}.

In \cite{MOST} not only is a combinatorial method of computing $HFK^{-}(S^3,K)$ given, a combinatorial proof of invariance is presented. There are quasi-isomorphisms associated to commutations and the various stabilizations. These quasi-isomorphisms admit natural extensions to grid homology for links in lens spaces. We will later use properties of these extensions to prove invariance of the GRID invariants defined in Section \ref{sec:invariants}.

%Note that if the basepoint pairs in two adjacent columns (or rows) of a grid diagram interleave, then commuting these columns (or rows) possibly changes the link type. This operation is called a cross-commutation, and has effect of performing a crossing change on the natural projection of the link to the Heegaard torus. In \cite{gridhom} there are combinatorial chain maps associated to cross-commutations. We will later use these maps to prove Proposition \ref{prop:nontorsion}.

We now turn to the definition of the chain map for a column commutation, the case of a row commutation is similar. Suppose that $G'$ is obtained by commuting two adjacent columns of $G$. It is useful to draw both $G$ and $G'$ on a $T^2$ simultaneously. Note that replacing $\beta _i \in \boldsymbol{\beta}$ by the curve $\gamma_i$ depicted in Figure \ref{fig:COMMcomb} gives the diagram $G'$. The curve $\gamma_{i}$ intersects $\beta_i$ in two points. Let $\theta\in\gamma_i\cap\beta_i$ denote the point at the top of the bigon in $T^{2}\smallsetminus \{\gamma_i\cup\beta_i\}$ whose left boundary consists of an arc along $\beta_i$.

We set $\boldsymbol{\gamma} = \boldsymbol{\beta}\smallsetminus \beta_i \cup \gamma_i$, so that $G' = (T^2, \boldsymbol{\gamma},\boldsymbol{\alpha},\bold{w},\bold{z})$.

\begin{figure}[h]
\def\svgwidth{200pt}
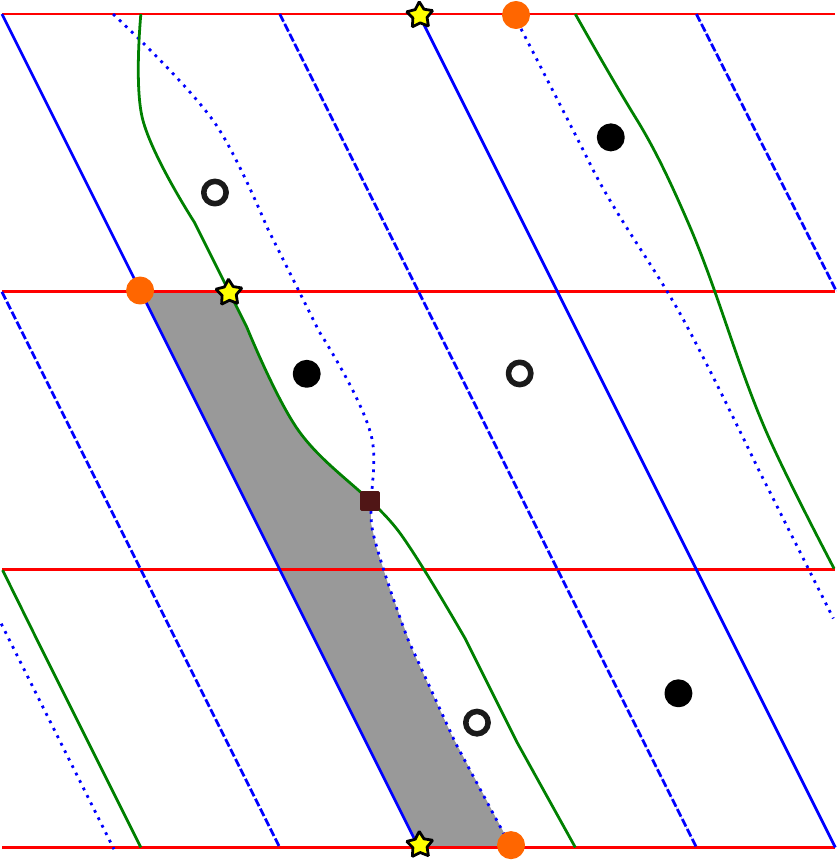
\caption{The solid and hollow dots depict $\bold{w}$ and $\bold{z}$ basepoints, respectively. The brown square depicts $\theta$. The green curve is $\gamma_i$. The domain of a pentagon is shaded.}
\label{fig:COMMcomb}
\end{figure}

\begin{definition}
Fix $\bold{x}\in \mathbb{T}_{\boldsymbol{\beta}}\cap\mathbb{T}_{\boldsymbol{\alpha}}$ and $\bold{y}\in \mathbb{T}_{\boldsymbol{\gamma}}\cap\mathbb{T}_{\boldsymbol{\alpha}}$. A \emph{pentagon} from $\bold{x}$ to $\bold{y}$ is an embedded disk $p\subset T^2$ whose boundary consists of five arcs, each of which lies along some $\boldsymbol{\beta},\boldsymbol{\gamma},$ and $\boldsymbol{\alpha}$ curve, satisfying the conditions:
\begin{itemize}
\item Four corners of $p$ are in $\bold{x}\cup\bold{y}$.
\item $p$ has multiplicity $1/4$ at each of its corners.
\item The portion of $\partial p$ along the $\boldsymbol{\alpha}$ curves is an oriented path from $\bold{y}$ to $\bold{x}$.
\end{itemize}
The set of pentagons from $\bold{x}$ to $\bold{y}$ is denoted $Pent(\bold{x},\bold{y})$.
\end{definition}

Note that $Pent(\bold{x},\bold{y})$ is empty unless $\bold{x}$ and $\bold{y}$ share $n-2$ components. Moreover the set of pentagons between two generators consists of at most one element. The fifth corner of any pentagon must be $\theta$.

A pentagon $p\in Pent(\bold{x},\bold{y})$ is said to be \emph{empty} if the interior of $p$ is disjoint from $\bold{x}$ and $\bold{y}$, the set of such pentagons is denoted $Pent^\circ (\bold{x},\bold{y})$.

We define a map $P: CFK^{-}(G)\to CFK^{-}(G')$ by 
\[
P(\bold{x}) = \sum\limits_{\bold{y}\in \mathbb{T}_{\boldsymbol{\gamma}}\cap \mathbb{T}_{\boldsymbol{\alpha}}}  \sum_{\substack{p\in Pent^\circ(\bold{x},\bold{y})\\ p\cap\bold{w}=\emptyset}}   U_0 ^{z_0 (p)} U_1^{z_1(p)}\dots U_{n-1}^{z_{n-1}(p)} \cdot \bold{y}.
\]
$P$ is a chain map, and induces an isomorphism on homology. The proof of these facts is a straightforward adaptation of the arguments appearing in \cite{MOST}.

We avoid defining the quasi-isomorphisms associated to destabilizations, instead we define their restrictions to certain subcomplexes. 

\begin{remark}
The chain maps, and all facts asserted about them, are discussed in detail in Chapter 5 of \cite{gridhom}. Slightly different conventions are used, and only grid diagrams for knots in $S^3$ are considered. Translating to our conventions and generality is just a matter of changing notation.
\end{remark}

Let $G'$ be obtained by performing a stabilization near a $\bold{w}$ basepoint of $G$. By renumbering the variables we think of $CFK^{-}(G)$ as an $\mathbb{F}[U_1,\dots,U_n]$-module, and $CFK^{-}(G')$ as an $\mathbb{F}[U_0,\dots,U_n]$-module. Let $CFK^{-}(G)[U_0]$ denote the bigraded complex $CFK^{-}(G)\otimes_{\mathbb{F}[U_1,\dots,U_n]} \mathbb{F}[U_0,\dots,U_n]$, where $x\otimes U_0 ^k$ has bigrading $(d-2k,s-k)$ for a homogenous $x\in CFK^- (G)$ having bigrading $(d,k)$. 

%Here, we are working with the convention that the pair of basepoints introduced in the stabilization are denoted $z_0$ and $w_0$, and both of these lie in the same row of $G'$.

There is a natural projection
\[
\pi : H_* (CFK^- (G)[U_0]) \simeq HFK^-(G)[U_0]\to \frac{HFK^-(G)[U_0]}{U_0+U_1}\simeq HFK^- (G).
\]

Let $\beta '$ and $\alpha '$ be the pair of curves introduced in the stabilization. There is a distinguished intersection point $\eta \in\beta' \cap\alpha '$. Let $I(G')$ denote the points of $\mathbb{T}_{\boldsymbol{\beta '}}\cap \mathbb{T}_{\boldsymbol{\alpha '}}$ having $\eta$ as a component, and let $N(G')$ denote the complement of $I(G')$. Let $\bold{I}$ and $\bold{N}$ denote the submodules of $CFK^- (G')$ generated by $I(G')$ and $N(G')$ over $\mathbb{F}[U_0,\dots,U_n]$, respectively.

There is a natural one-to-one correspondence between $I(G')$ and $\mathbb{T}_{\boldsymbol{\beta}}\cap\mathbb{T}_{\boldsymbol{\alpha}}$. This extends to give an isomorphism of $\mathbb{F}[U_0,\dots,U_n]$-modules
\[
e:\bold{I}\to CFK^- (G)[U_0].
\]
If we are considering a stabilization of type W:NE or W:SW then $e$ is a bigraded map, otherwise it is homogenous of degree $(1,1)$.

For stabilizations of type W:NW or W:SE, $\bold{N}$ is easily seen to be a subcomplex of $CFK^-(G')$. In these cases we will need a chain homotopy equivalence
\[
\mathcal{H}^I_{w_1}:\bold{N}\to\bold{I}
\]
defined by 
\[
\mathcal{H}^I_{w_1}(x) =  \sum\limits_{\bold{y}\in I(G')}  \sum_{\substack{r\in Rect^\circ (\bold{x},\bold{y})\\ r\cap\bold{w}=w_1}}   U_0 ^{z_0 (p)} U_1^{z_1(p)}\dots U_{n}^{z_{n}(p)} \cdot \bold{y}.
\]

\begin{proposition}(See Proposition 5.4.1 of \cite{gridhom})
\label{prop:stabcomb}
If $G'$ is obtained from $G$ by a stabilization, then there is an isomorphism of bigraded $\mathbb{F}[U]$-modules from $HFK^{-}(G')\to HFK^{-}(G)$. In particular:
\begin{itemize}
\item
If the stabilization is of type W:NW or W:SE, the restriction of the above isomorphism to cycles coming from the subcomplex $\bold{N}$ is the map $\pi \circ (e\circ\mathcal{H}^I_{w_1})_*$.
\item
If the stabilization is of type W:NE or W:SW, then $\bold{I}$ is a subcomplex of $CFK^-(G')$. The restriction of the above isomorphism to cycles coming from the subcomplex $\bold{I}$ is given by $\pi \circ e_*$.
\end{itemize}
\end{proposition}

\section{The GRID invariants for links in lens spaces}
\label{sec:invariants}

Given a grid diagram $\mathcal{G} = (T^2,\boldsymbol{\alpha},\boldsymbol{\beta},\bold{z},\bold{w})$, consider the diagram $G = (T^2, \boldsymbol{\beta},\boldsymbol{\alpha},\bold{w},\bold{z})$ for $(-L(p,q),K)$, and the generator $\bold{x}^+ \in CFK^- (G)$ (respectively $\bold{x}^-$) having components which are in the upper left (respectively lower right) corners of regions in $T^2 \smallsetminus(\boldsymbol{\beta}\cup\boldsymbol{\alpha})$ containing points of $\bold{w}$. In this section we refer to $G$ as a grid diagram for $K$.

\begin{lemma}
\label{lem:cycle}
The generators $\bold{x}^+,\bold{x}^- \in CFK^- (G)$ are cycles.
\end{lemma}
\begin{proof}
The differential for the complex $CFK^-(G)$ counts parallelograms (Proposition 2.1 of \cite{lensgridcomb}). Suppose that $P$ is a parallelogram contributing to the differential of $\bold{x}^+$, let $x$ denote the component of $\bold{x}^+$ in the top left corner of $P$. The parallelogram $P$ now clearly contains a $\bold{w}$ basepoint. The proof for $\bold{x}^-$ is similar.
\end{proof}

For the case of a grid diagram $G$ representing a link in the three sphere, the following was proven in \cite{grid}:
\begin{theorem}(combination of Theorems 1.1 and 7.1 from \cite{grid})
\label{thm:OST}
If $G$ is a grid diagram for a link $K=K_1\cup K_2\cup\dots\cup K_l \subset S^3$, let $L$ and $T$ be the corresponding oriented Legendrian and transverse representatives of $K$.
The homology class $[x^+]$ in $HFK^-(-S^3,L)$ is an invariant of $L$ and $T$ up to Legendrian and transverse isotopy, respectively.
The class $[x^-]$ is also an invariant of $L$ up to Legendrian isotopy. These invariants are supported in multi-gradings
\begin{align*}
M(\bold{x}^+) = tb(L) - rot(L) +1 = sl(T)+1\\ 
A_{L_i}(\bold{x}^+) = \frac{1}{2}\Big{(}tb(L_i) - rot^i(L) +1\Big{)}= \frac{1}{2}\Big{(}sl^i(T) +1\Big{)}\\
M(\bold{x}^-) = tb(L) + rot(L) +1\quad \quad \text {and }\quad\quad A_{L_i}(\bold{x}^-) = \frac{1}{2}\Big{(}tb(L_i) + rot^i(L) +1\Big{)}.
\end{align*}
\end{theorem}

\begin{remark}
The conventions used in \cite{grid} differ from ours, the above has been translated to our conventions. In \cite{equiv} their version of $[\bold{x}^+]$ has been shown to agree with the BRAID invariant (also defined in \cite{equiv}). Using the same proof (in the case of $S^3$), the invariant we will construct can also be shown to agree with the BRAID invariant, and in particular the invariant defined in \cite{grid}.
\end{remark}

We establish Theorem \ref{thm:grid} in a sequence of Lemmas and Propositions.

Let $G$ be a grid diagram encoding a Legendrian link $L\subset (L(p,q),\xi_{UT})$.
We begin by showing that $[\bold{x}^+],[\bold{x}^-]\in HFK^- (G)$ are invariants of the oriented Legendrian isotopy class of $L$. In light of Theorem \ref{thm:leg} it suffices to prove invariance under the elementary Legendrian grid moves, which are commutations along with (de)stabilizations of types W:NE and W:SW.

\begin{lemma}
The classes $[\bold{x}^+]$ and $[\bold{x}^-] \in HFK^- (G)$ are invariant under commutations. In particular, if $G'$ is obtained from $G$ by a commutation move then the quasi-isomorphism
\[
P: CFK^-(G)\to CFK^-(G')
\]
sends $\bold{x}^+(G)$ and $\bold{x}^- (G)$ to $\bold{x}^+(G')$ and $\bold{x}^- (G')$, respectively.
\end{lemma}
\begin{proof}

Let $G$ and $G'$ differ be grid diagrams differing by a commutation. Recall that in Subsection \ref{subsec:combmaps} we defined a quasi-isomorphism
\[
P: CFK^-(G)\to CFK^-(G')
\]
which counts empty pentagons.
There is an obvious empty pentagon $p\in Pent^\circ (\bold{x}^+(G),\bold{x}^+(G'))$, see Figure \ref{fig:COMMcomb}. It is easy to see that $p$ is the unique empty pentagon connecting $\bold{x}^+(G)$ to any point of $\mathbb{T}_{\boldsymbol{\gamma}}\cap \mathbb{T}_{\boldsymbol{\alpha}}$. If $p'$ is a pentagon other than $p$ having upper right corner at a component of $\bold{x}^+(G)$, then $p'$ contains a parallelogram of $T^2\smallsetminus \{\boldsymbol{\beta}\cup\boldsymbol{\alpha}\}$ 
having a $\bold{w}$ basepoint, so $p'$ is not empty.
We have established that $P(\bold{x}^+(G))=\bold{x}^+(G')$. The proof that $P(\bold{x}^-(G))=\bold{x}^- (G')$ is similar, the picture is obtained by rotating the diagram of Figure \ref{fig:COMMcomb} by 180 degrees.
\end{proof}

\begin{proposition}
Suppose that $G'$ is obtained from $G$ by applying a destabilization of type W:NE or W:SW. There is an isomorphism 
\[
HFK^-(G')\to HFK^-(G)
\]
sending $[\bold{x}^+(G')]$ and $[\bold{x}^-(G')]$ to $[\bold{x}^+(G)]$ and $[\bold{x}^-(G)]$, respectively.
\end{proposition}
\begin{proof}
Suppose that $G'$ is obtained form $G$ by applying a destabilization of type W:NE. Note that the generator $\bold{x}^+(G')$ lies in the subcomplex $\bold{I}$ of $CFK^- (G')$. The second half of Proposition \ref{prop:stabcomb} tells us that the isomorphism
\[
HFK^-(G')\to HFK^-(G)
\]
induced by destabilization of type W:NE maps $[\bold{x}^+ (G')]$ to $\pi \circ e_* ([\bold{x}^+(G')]) = \pi ([\bold{x}^+(G)]\otimes 1)=[\bold{x}^+(G)]$. The case of a (de)stabilization of type W:SW is similar. 

The proof of invariance for the class $[\bold{x}^-]$ is similar.
\end{proof}

We have shown that if a grid diagram $G$ encodes a Legendrian link $L\subset (L(p,q),\xi_{UT})$ then the classes $[\bold{x}^\pm (G)]\in HFK^- (G)$ are invariants of the oriented Legendrian link $L$ up to Legendrian isotopy, we  denote them by $\lambda^\pm (L)$.

We establish the behavior of the invariants $\lambda^\pm (L)$ under Legendrian stabilizations. 

\begin{proposition}
Let $L^-$ (respectively $L^+$) denote a negative (respectively positive) Legendrian stabilization along some component of $L_i\subset L$ of $L$,  in $(L(p,q),\xi_{UT})$. We have that 
\begin{align*}
\lambda^+ (L^-) = \lambda^+(L) \quad \quad \quad  \quad \lambda^- (L^-) = U\cdot \lambda^-(L)\\
\lambda^+ (L^+) = U\cdot \lambda^+(L) \quad \quad \quad  \quad \lambda^- (L^+) = \lambda^-(L).
\end{align*}
\end{proposition}

\begin{proof}
Let $G$ be a grid diagram encoding the Legendrian link $L$. Let $G^-$ ($G^+$) be a grid diagram obtained from $G$ by performing some W:SE (W:NW) stabilization; recall that stabilizations of this type correspond to negative (positive) Legendrian stabilizations.

Recall the notation established in the discussion preceding Proposition \ref{prop:stabcomb}. Note that both $x^\pm (G^-)$ (respectively $x^\pm (G^+)$) lie in the subcomplex $\bold{N}$ of $CFK^-(G^-)$ (respectively $CFK^-(G^+)$).
We will show that the compositions (there are really two such maps, one from the subcomplex of $CFK^-(G^-)$, and one from the subcomplex $CFK^-(G^+)$)
\[
e\circ \mathcal{H}^I_{w_1}:\bold{N}\to CFK^-(G)[U_0]
\]
map the the distinguished generators as follows:
\begin{align*}
\bold{x}^+(G^-)\to x^+(G)\otimes 1\quad \quad\quad \bold{x}^+(G^+)\to \bold{x}^+(G)\otimes U_0\\
\bold{x}^-(G^-)\to x^-(G)\otimes U_0\quad \quad\quad \bold{x}^-(G^-)\to \bold{x}^-(G)\otimes 1.
\end{align*}
The first half of Proposition \ref{prop:stabcomb} will then give the desired result.

\begin{figure}[h]
\def\svgwidth{200pt}
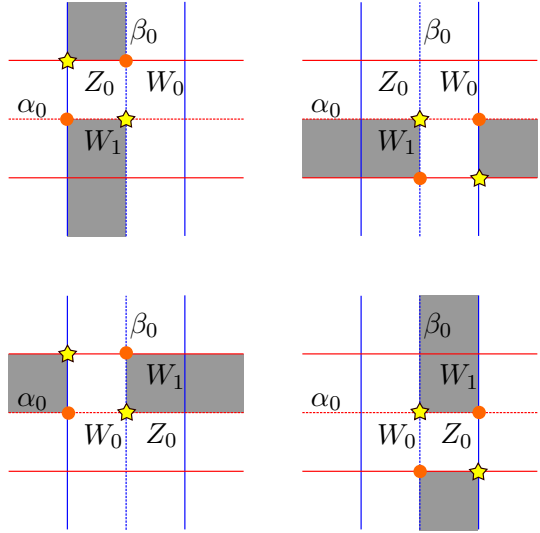
\caption{The domains of rectangles contributing to $\mathcal{H}^I_{w_1}(\bold{x}^\pm(G^+))$ or $\mathcal{H}^I_{w_1}(\bold{x}^\pm(G^-))$ are shaded above. In each case, the orange dots are components of $\bold{x}^\pm (G^\pm)$.}
\label{fig:legstabinv}
\end{figure}

The rectangles illustrated in Figure \ref{fig:legstabinv} are the only ones emanating from $\bold{x}^\pm (G^\pm)$ disjoint from all points of $\bold{w}\smallsetminus w_1$, and hence the only rectangles contributing to $\mathcal{H}^I_{w_1}(\bold{x}^\pm(G^+))$ or $\mathcal{H}^I_{w_1}(\bold{x}^\pm(G^-))$. In each case it is straightforward to post-compose with the map $e$ and check that the composition behaves as desired.

\end{proof}

\begin{proposition}
\label{prop:spincagrees}
The GRID invariants have $Spin^C$ structures agreeing with those of the contact plane fields. In particular
\[
\mathfrak{s}_\bold{w}(\theta) = \mathfrak{s}_\bold{w} (\lambda^+) = \mathfrak{s}_{\xi_{UT}}\quad \quad\text{ and }\quad\quad \mathfrak{s}_\bold{w}(\lambda^-) = \mathfrak{s}_{\overline{\xi_{UT}}}.
\]
\end{proposition}
\begin{proof}
Note that moving the $\bold{z}$ basepoints has no effect on $Spin^C$ structure $\mathfrak{s}_\bold{w}(\bold{x}^+)$. In light of the previous two propositions, stabilizations of any kind have no effect on the $Spin^C$ structure as well. Using these two types of moves (stabilizations and moving $\bold{z}$ basepoints arbitrarily), we may go from an arbitrary grid diagram $G$ to an index one diagram for the binding $B$ of the standard rational open book supporting $\xi_{UT}$. 

For this index one diagram, the generator $\bold{x}^+$ is easily seen to have maximal Alexander grading. (This is trivial to check using Lemma \ref{lemma:relperiodic}). By Theorem 1.1 of \cite{QOB}, it follows that $[\bold{x}^+] \in \widehat{HF}(T^2,\boldsymbol{\beta},\boldsymbol{\alpha},\bold{w})$ represents the contact invariant $c(\xi_{UT})$, in particular $\mathfrak{s}_\bold{w}(\bold{x}^+) = \mathfrak{s}_{\xi_{UT}}$.

Rotating the grid diagram by 180 degrees, noting that this corresponds to co-orientation reversal of the contact structures, and carrying out the same argument gives the analogous result for $\bold{x}^-$.
\end{proof}

Corollary \ref{corollary:grid} follows immediately. Recall that if $T$ is the positive transverse push-off of a Legendrian $L$, then we denote by $\theta (T) = \lambda^+ (L)$ the transverse invariant.

We turn to the computation of the bigradings of $\lambda^\pm (L)$.
Recall that $\Delta_n\in B_n$ denotes the Garside element, so that $\Delta_n ^2$ is the full twist on $n$ strands.

\begin{proposition}
\label{prop:gradingcomparison}
Let $L$ and $T$ denotes Legendrian and transverse links, respectively. The invariants $\lambda^+(L), \lambda^-(L)$ and $\theta(T)$ are supported in Maslov gradings
\begin{align*}
M(\lambda^+(L)) = tb_\mathbb{Q}(L) - rot_\mathbb{Q}(L) +\frac{1}{p} -d(p,q,q-1)\\
M(\lambda^-(L)) = tb_\mathbb{Q}(L) + rot_\mathbb{Q}(L) +\frac{1}{p}-d(p,q,q-1)\\
M(\theta(T)) = sl_{\mathbb{Q}}(T) +\frac{1}{p}-d(p,q,q-1)
\end{align*}
\end{proposition}
\begin{proof}

We prove the third equality.
Let $\beta\circ \delta^{q/p}$ be an n-braid representing $T$, and let $\beta_1$ denote the 1-braid. 
Let $G$ be a grid diagram for $\beta\circ \delta^{q/p}$. We may take a p-fold cover of $G$ to get a grid diagram for $\beta^p\circ (\Delta^2)^q$. Likewise, we may cover a grid diagram for the braid $\tau_1\circ \delta^{q/p}$ by a grid diagram for the 1-braid $\tau_1$. 
The gradings of the GRID invariants in $(S^3,\xi_{std})$ have been computed, see Theorem \ref{thm:OST}, we have
\begin{align*}
M(\theta(\beta^p\circ (\Delta^2)^q)) = sl(\beta^p\circ (\Delta^2)^q)+1=pw(\beta)+(qn-1)(n-1)\\
M(\theta(\tau_1))=sl(\tau_1)+1 = 0
\end{align*}
where $w(\beta)$ is the writhe of $\beta$.

Theorem 2.6 of \cite{coversmaslov} allows us to compute the grading difference:
\begin{align*}
M(\theta(\beta\circ \delta^{q/p}))-M(\theta(\beta_1\circ\delta^{q/p}))\\ = \frac{1}{p}(M(\theta(\beta^p\circ (\Delta^2)^q))-M(\theta(\beta_1)))\\ 
=w(\beta)+\frac{1}{p}(qn^2-qn-n+1)\\
= sl_\mathbb{Q}(\beta\circ\delta^{q/p})+\frac{1}{p}
\end{align*}
where the last equality uses Lemma \ref{lem:QSL}.

Note that $\tau_1\circ\delta^{q/p}$ can be encoded with an index one diagram. In this diagram, the Maslov grading $M(\theta(\tau_1\circ\delta^{q/p}))$ has been computed in \cite{dinvt} to be $-d(p,q,q-1)$ where the function $d$ is recursively defined by
\begin{align*}
d(1,0,0) =0\\
d(p,q,i) = \Big{(}\frac{pq-(2i+1-p-q)^2}{4pq}\Big{)}-d(q,r,j)
\end{align*}
where $r$ and $j$ are the reductions of $p$ and $i$ modulo $q$, respectively. The minus sign comes from our orientation conventions being opposite to that of \cite{dinvt}.

Note that the first equality follows from the third, for if $T$ is the positive transverse push-off of a Legendrian $L\subset (L(p,q),\xi_{UT})$ then we have that $sl_{\mathbb{Q}}(T) = tb_\mathbb{Q}(L) - rot_\mathbb{Q}(L)$. 

The second equality follows from the first by another application of \cite{coversmaslov} and \cite{grid}. Let $\tilde{L}$ denote the pre-image of $L$ under the contact universal cover.
\begin{align*}
M(\lambda^+(L))-M(\lambda^-(L))\\ = \frac{1}{p}(M(\lambda^+(\tilde{L}))-M(\lambda^-(\tilde{L}))\\
=\frac{1}{p}((tb)\tilde{L})-rot(\tilde{L})+1) - (tb(\tilde{L})+rot(\tilde{L})+1)\\
=\frac{1}{p}(-2rot\tilde{L}) = -2rot_\mathbb{Q}(L)
\end{align*}

\end{proof}

\begin{proposition}
\label{prop:agradingcomp}
Let $L=L_1\cup\dots\cup L_{l}$ and $T\cup\dots\cup T_{l}$ denote Legendrian and transverse links, respectively. The invariants $\lambda^+(L), \lambda^-(L)$ and $\theta(T)$ are supported in Alexander gradings
\begin{align*}
A_{L_i}(\lambda^+(L)) = \frac{1}{2}\Big{(}tb_\mathbb{Q}(L_i) - rot^i_\mathbb{Q}(L) +1\Big{)}\\
A_{L_i}(\lambda^-(L)) = \frac{1}{2}\Big{(}tb_\mathbb{Q}(L_i) + rot^i_\mathbb{Q}(L) +1\Big{)}\\
A_{T_i}(\theta(T)) = \frac{1}{2}\Big{(}sl^i_\mathbb{Q}(T)  +1\Big{)}.
\end{align*}
%where $r_i$ denotes the order of $[L_i]$ or $[T_i]$ in $H_1(L(p,q))$.
\end{proposition}
\begin{proof}
This is a straightforward combination of Lemmas \ref{lem:cover} and \ref{lem:Acovers} with Theorem \ref{thm:OST}.
\begin{comment}
Fix $\bold{x}^+\subset \mathbb{T}_{\boldsymbol{\beta}}\cap\mathbb{T}_{\boldsymbol{\alpha}}$ to be as before. We inherit notation from the proof of Lemma \ref{lem:rationalalexandergrading}.
Applying Lemma \ref{lem:rationalalexandergrading} we get
\[
A_{L_i} (\bold{x}^+)= \frac{1}{2}(M_\bold{z} (\bold{x}^+) - M_\bold{w} (\bold{x}^+) - (n_i - \frac{1}{r_i}))
\]
where $n_i$ is the number of basepoints pairs encoding $L_i$. By \cite{coversmaslov}, this equals
\[
\frac{1}{2}\Big{(} \frac{1}{p}(M_{\tilde{\bold{z}}} (\tilde{\bold{x}}^+) - M_{\tilde{\bold{w}}} (\tilde{\bold{x}}^+)) - (n_i - \frac{1}{r_i})\Big{)}.
\]
On the other hand, the Alexander grading of $\tilde{\bold{x}}^+$ is given in Theorem \ref{thm:OST},
\[
\frac{1}{2} (tb(\tilde{L}_i) - rot^i (\tilde{L}) + \frac{p}{r_i}) = A_{\tilde{L}_i}(\tilde{\bold{x}}^+) = \frac{1}{2} (M_{\tilde{\bold{z}}} (\tilde{\bold{x}}^+) - M_{\tilde{\bold{w}}} (\tilde{\bold{x}}^+)-(n_ip - \frac{p}{r_i}))
\]
Combining these equations with Lemma \ref{lem:cover} gives the result.
The computation of $A(\lambda^- (L))$ is similar.
\end{comment}
\end{proof}

This completes the proof of Theorem \ref{thm:grid}.

\begin{proposition}
\label{prop:nontorsion}
For any Legendrian link $L$ or transverse link $T$ in $(L(p,q),\xi_{UT})$ the homology classes $\lambda^+ (L)$ and $\theta(T)$ do not vanish; the classes are non U-torsion, i.e. for any $n\ge 1$ we have that $U^n \cdot \lambda^+(L)\ne 0$ and $U^n \cdot \theta (T)\ne 0$
\end{proposition}
\begin{proof}

Let $L\subset (L(p,q),\xi_{UT})$ be an arbitrary Legendrian encoded by a grid diagram $G$. Consider the complex
\[
C' (G) = CFK^- (G)/ \{U_i = 1\}_{i=0}^{n-1}.
\]
It suffices to show that the homology class $[x^+]$ is non-zero in $H_* (C'(G))$. Note that the $\bold{z}$ basepoints do not have any effect on the differential, so we may move them as we like. The isomorphisms associated to each (de)stabilization preserve the class $[x^+]$. 

It is easy to go from $G$ to some index one diagram $D$ using some sequence of relocations of the $\bold{z}$ basepoints and destabilizations. There is an induced isomorphism
\[
H_* (C'(G))\to H_* (C'(D))
\]
taking $[x^+(L)]$ to $[x^+(D)]$. Since $D$ is a index one diagram, the complex $C'(D)$ has no differential and $[x^+(D)]$ is non-zero. 
\end{proof}

\section{The BRAID invariant}
\label{sec:braidinvt}

In this section we review the definition of the BRAID invariant, defined in \cite{equiv}. The definition is reminiscent of the definition of the contact invariant given in \cite{HKM}. 

Let $(B,\pi)$ be an open book supporting $(Y,\xi)$. Let $(S,\phi)$ be the abstract open book corresponding to $(B,\pi)$, and let $g$ be the genus of a fiber. If $K$ is an index $k$ braid with respect to $(B,\pi)$, $K$ is specified by a lift $\widehat{\phi} \in MCG(S\smallsetminus \{p_1,\dots , p_k\}, \partial S)$ of $\phi$.

\begin{definition}
A \emph{basis of arcs} $\{a_i\}_1 ^{2g+k-1}\subset S\smallsetminus \{p_1,\dots, p_k\}$ is a collection of properly embedded disjoint arcs which cut $S\smallsetminus \{p_1,\dots, p_k\}$ into $k$ discs, each having precisely one of the $p_i$.
\end{definition}
Let $\{b_i\}_1^{2g+k-1}$ be another basis of arcs, where $b_i$ is obtained by slightly moving the endpoints of $a_i$ in the oriented direction of $\partial S$, and isotoping in $S\smallsetminus \{p_1,\dots, p_k\}$ so that $b_i$ intersects $a_i$ transversely in a single point.

Let $\Sigma$ denote the surface $S_{1/2}\cup -S_0$. For each $i$, let 
\begin{align*}
\alpha_i = a_i\times \{0,1/2\} \\
\beta_i = b_i\times \{1/2\}\cup \widehat{\phi}(b_i)\times \{0\}\\
z_i = p_i \times \{0\}\\
\text{and } w_i = p_i \times \{1/2\}
\end{align*}
Let $\boldsymbol{\alpha} = \{\alpha_1,\dots,\alpha_k\}$, $\boldsymbol{\beta} = \{\beta_1,\dots, \beta_k\}$, $\bold{z} = \{z_1,\dots,z_k\}$, and $\bold{w} = \{w_1,\dots,w_k\}$. Then $\mathcal{H} = (\Sigma,\boldsymbol{\beta},\boldsymbol{\alpha},\bold{w},\bold{z})$ is a multi-pointed Heegaard diagram encoding $(-Y,K)$.

Each $\alpha_i$ intersects $\beta_i$ in a single point in the region $S_{1/2}$ denoted $x_i$. Let $\bold{x} \in \mathbb{T}_{\boldsymbol{\beta}} \cap \mathbb{T}_{\boldsymbol{\alpha}}$ denote the generator having component $x_i$ on $\alpha _i$. The homology class $[\bold{x}] \in HFK^- (-Y,K)$ is an invariant of the transverse isotopy class of $K$ (Theorem 3.1 of \cite{equiv}). The invariant is denoted $t(K)$, and we refer to it as the BRAID invariant.

%The invariant $t(K)$ is equivalent to that of $\theta(K)$ defined in \cite{grid}, and also equivalent to $\mathcal{T}(K)$ defined in \cite{LOSS}. See \cite{equiv} for proofs of equivalence.

\subsection{The transverse invariant of a braid and its axis}
\label{subsec:binding}

Suppose that $(B,\pi)$ is an open book decomposition, with $B$ having $n$ components, supporting $(Y,\xi)$. Let $(S,\phi)$ be the abstract open book corresponding to $(B,\pi)$, where $S$ has genus $g$. As discussed in subsection \ref{subsec:contact} the binding $B$ is naturally a transverse link that may be braided about the open book via a transverse isotopy. Abusing notation, we denote the resulting $n$-braid $B$.

$B$ is specified by a lift $\widehat{\phi} \in MCG( S\smallsetminus \{p_1,\dots,p_n\})$ of $\phi$. Thinking of $\phi$ as fixing a collar neighborhood $\nu(\partial S)$ of the boundary, one obtains $\widehat{\phi}$ by composing $\phi$ with $n$ push maps supported in $\nu (\partial S)$. See Figure \ref{fig:one} for the push maps and a basis of arcs $\{a_i\}_1^{2g+n-1}\cup \{a_{2,i}\}_{2g+1}^{2g+n-1}$ for $S\smallsetminus \{p_1,\dots,p_n\}$.

\begin{figure}[h]

\def\svgwidth{250pt}
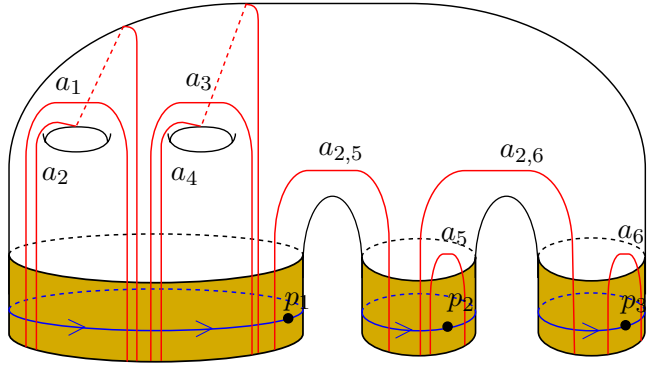
\caption{The basis of arcs $\{a_i\}_1^{2g+n-1}\cup \{a_{2,i}\}_{2g+1}^{2g+n-1}$ for the case $n=3$ and $g=2$ is depicted in red. The push maps, supported in the shaded neighborhood $\nu (\partial S)$, go in the orientation of $\partial S$ and are depicted in blue.}
\label{fig:one}
\end{figure}

The basis of arcs along with $\widehat{\phi}$ specify a Heegaard diagram $D= (\Sigma, \boldsymbol{\beta}, \boldsymbol{\alpha}, \bold{w}_B, \bold{z}_B) $, along with a generator $\bold{x}_D$, shown in Figure $\ref{fig:two}$ for $(-Y,B)$, as described in Section \ref{sec:braidinvt}. The homology class $[\bold{x}_D]\in HFK^-(D)$ is the braid invariant $\widehat{t}(B)$. The labelling of the basis arcs induces a labelling of the $\boldsymbol{\beta}$ and $\boldsymbol{\alpha}$ curves.

\begin{figure}[h]
\def\svgwidth{250pt}
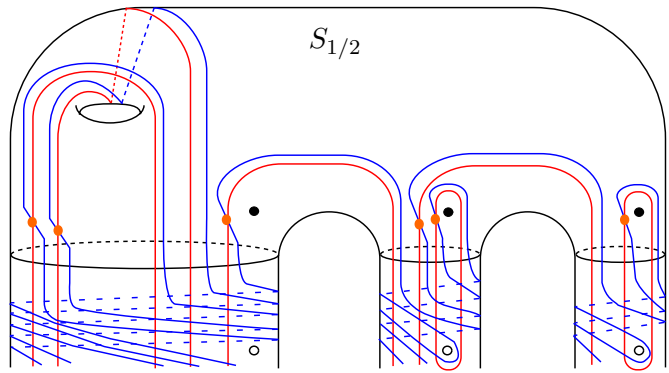
\caption{A portion of the Heegaard diagram $D$ for $(-Y,B)$ in the case $g=1$ and $n=3$. The $\bold{w}_B$ and $\bold{z}_B$ basepoints are depicted with solid and hollow dots, respectively. The homology class of the generator depicted by orange dots, in $\widehat{HFK}(D)$, is equal to the transverse invariant $\widehat{t}(B)$. The indexing of the basis in Figure \ref{fig:one} induces an indexing of the $\alpha$ and $\beta$ curves in $D$.}
\label{fig:two}
\end{figure}

%INSERT FIGURE FOR HEEGAARD DIAGRAM AND CYCLE WHOSE CLASS IS T(B)
By applying an isotopy to $D$, we obtain the Heegaard diagram $\mathcal{D}$ pictured in Figure \ref{fig:three}.

\begin{figure}[h]
\def\svgwidth{250pt}
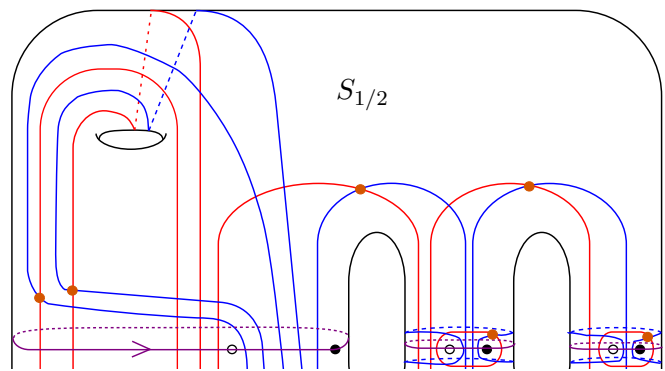
\caption{A portion of the Heegaard diagram $\mathcal{D}$  obtained from $D$ via and isotopy. The generator $\bold{x}_\mathcal{D}$, whose homology class is $\widehat{t}(B)$, is depicted by orange dots. The purple multi-curve depicts an oriented (as $\partial S$) longitude for $B$.}
\label{fig:three}
\end{figure}

Now let $K$ be a link braided about $B$ having braid index $k$. We may add $k$ pairs of basepoints $\bold{w}$, $\bold{z}$, and curves to $\mathcal{D}$ to obtain a new diagram $\mathcal{H} = (\Sigma, \boldsymbol{\beta},\boldsymbol{\alpha}, \bold{w}\cup \bold{w}_B,\bold{z}\cup \bold{z}_B)$, see Figure \ref{fig:nine}, which encodes $(-Y,K\cup B)$. We reindex the $\boldsymbol{\beta}$ and $\boldsymbol{\alpha}$ curves for our convenience. 

$\mathcal{H}$ is isotopic to the usual diagram appearing in the definition of the braid invariant when considering a particular basis of arcs for $S\smallsetminus \{p_1,\dots,p_{n+k}\}$. The generator $\bold{x}_{\mathcal{H}}$ pictured in the figure has homology class $\widehat{t}(K)\in \widehat{HFK}(\mathcal{H})$.

 \begin{figure}[h]
\def\svgwidth{250pt}
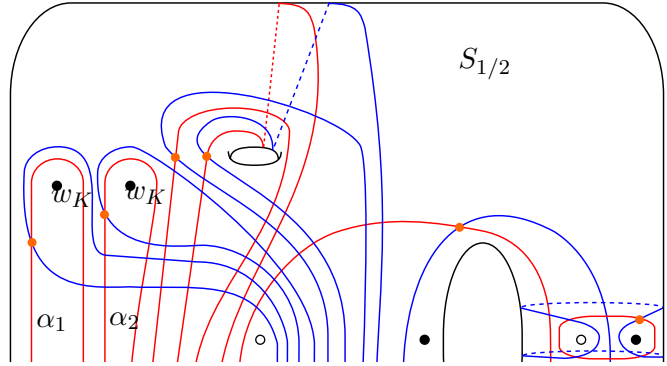
\caption{A portion of the diagram $\mathcal{H}$ in the case $k=2$, $g=1$, and $n=2$. The $k$ pairs of alpha/beta curves introduced to encode $K$ are indexed $1,\dots,k$ left to right. The rest of the curves are indexed as before (see Figures \ref{fig:one} and \ref{fig:two}) with a shift of $k$ in the first coordinate. The orange dots are components of $\bold{x}_\mathcal{H}$. The $\bold{w}_B\cup \bold{w}$ and $\bold{z}_B$ basepoints are depicted with solid and hollow dots, respectively.}
\label{fig:nine}
\end{figure}

In earlier work we use the diagrams $\mathcal{D}$ and $\mathcal{H}$ to prove Theorem \ref{thm:nonvanishing}, which plays an important role in the next section.

\begin{theorem} (Theorem 1.1 of \cite{braiddynamics})
\label{thm:nonvanishing}
Let $(B,\pi)$ be an open book supporting $(Y,\xi)$. If $K$ is braided about $B$, then $\widehat{t}(B\cup K)\in \widehat{HFK}(-Y,B\cup K)$ is nonzero.
\end{theorem}

\subsection{A Reformulation of the BRAID invariant $t$}
\label{subsec:reformulation}

Let $(B,\pi)$ be an open book decomposition supporting $(Y,\xi)$, where the binding $B$ has $n$ components. Let $(S,\phi)$ be the abstract open book corresponding to $(B,\pi)$, where $S$ has genus $g$. Let $K \subset Y$ be an index $k$ braid about $B$ having $m$ components.

In this section we reformulate the transverse invariant $t(K)$ in terms of the Alexander filtration induced by $-B$ on $CFK^- (-Y,K)$.

Let $\mathcal{H}=(\Sigma, \boldsymbol{\beta},\boldsymbol{\alpha}, \bold{w}\cup \bold{w}_B,\bold{z}\cup \bold{z}_B)$ be the diagram for $(-Y,K\cup B)$ of Figure \ref{fig:nine}. Let $\bold{w}_B = \bold{z}_{-B}$ and $\bold{z}_B = \bold{w}_{-B}$ as sets. Swapping these sets of basepoints corresponds to reversing the orientation of $B$, so $\mathcal{\widetilde{H}} = (\Sigma, \boldsymbol{\beta}, \boldsymbol{\alpha}, \bold{w}\cup \bold{w}_{-B}, \bold{z} \cup \bold{z}_{-B})$ is a diagram for $(-Y,K\cup -B)$.

$\mathcal{H}_0=  (\Sigma, \boldsymbol{\beta}, \boldsymbol{\alpha}, \bold{w} , \bold{z} \cup \bold{z}_{-B})$ is a diagram for $(-Y,K)$ with $n$ free basepoints. Let $\bold{x}_0 \in CFK^{-,n}(\mathcal{H}_0)$ be the generator corresponding to $\bold{x}_{\mathcal{H}}$ of Figure \ref{fig:nine}. $-B$ induces an Alexander filtration on $CFK^{-,n} (\mathcal{H}_0)$:

\[
\emptyset = \mathcal{F}_i^{-B}(\mathcal{H}_0) \subset \mathcal{F}_{i+1}^{-B}(\mathcal{H}_0)  \subset \dots \subset \mathcal{F}_l^{-B}(\mathcal{H}_0)  =CFK^{-,n} (\mathcal{H}_0)
\]

Let
\[
b= min \{j | H_*(\mathcal{F}_j^{-B}(\mathcal{H}_0) ) \ne 0\}.
\]

Let $\mathcal{D} = (\Sigma, \{\beta_{k+1},\dots,\beta_{2g+2n+k-2}\}, \{\alpha_{k+1},\dots,\alpha_{2g+2n+k-2}\}, \bold{w}_B,\bold{z}_B)$ be the diagram for $(-Y,B)$ from the previous section, where we have preemptively reindexed the curves and basepoints, so that each entry of the tuple for $\mathcal{D}$ is a subset of the analagous entry for $\mathcal{H}$.

Consider the Heegaard diagram $\widetilde{\mathcal{D}} = (\Sigma, \{\beta_{k+1},\dots,\beta_{2g+2n+k-2}\}, \{\alpha_{k+1},\dots,\alpha_{2g+2n+k-2}\}, \bold{z}_{-B})$ for $-Y$ with $n$ basepoints.

As above, $-B$ induces a filtration on $\widehat{CF}(\widetilde{\mathcal{D}})$. 
Let
\[
r= min \{j | H_*(\mathcal{F}_j^{-B}(\widetilde{\mathcal{D}}) ) \ne 0\}.
\]

Note that $r=-g-n+1$.

\begin{lemma}
\label{lem:alexbot}
A generator of $CFK^{-,n} (\mathcal{H}_0)$ lies in $\mathcal{F}_b ^{-B}(\mathcal{H}_0) $ if and only if each of its components is in the region $S_{1/2}$.
Likewise, a generator of $\widehat{CFK}(\widetilde{\mathcal{D}})$ lies in $\mathcal{F}_r ^{-B}(\widetilde{\mathcal{D}})$ if and only if each of its components is in the region $S_{1/2}$.
\end{lemma}
\begin{proof}
In both cases, The portion of the diagram $S_{1/2}$ is a relative periodic domain for $B$. The result following immediately by applying Lemma \ref{lemma:relperiodic}.

\end{proof}

\begin{lemma}
As complexes we have $\mathcal{F}_b ^{-B}(\mathcal{H}_0) \simeq  (V_1\otimes V_2\otimes \dots \otimes V_k) \otimes {\mathcal{F}_r}^{-B} (\widetilde{\mathcal{D}})$, where each $V_i$ is a free rank two $\mathbb{F}[U_1,\dots , U_m]$-module with basis $\{x_i,y_i\}$. Let $\sigma \in S_k$ denote the permutation of the points $\{p_1,\dots,p_k\}$ given by the monodromy $\widehat{\phi}$. The differential on $V_i$ is as follows
\[
\partial x_i=0 \ \ \ \ \ \ \ \ \ \ \ \ \ 
\partial y_i = (U_i + U_{\sigma (i)})x_i
\]
\end{lemma}

\begin{proof}
For $\bold{y}\in \mathbb{T}_{\boldsymbol{\beta}}\cap\mathbb{T}_{\boldsymbol{\alpha}}$ let $(\bold{y})_i$, $(\bold{y})^i$, denote the component of $\bold{y}$ on $\alpha_i$, $\beta_i$, respectively. Note that in the region $S_{1/2}$, for $1\le i \le k$, the curve $\alpha _i$ intersects only the curves $\beta_j$ with $j\le i$.
If generator $\bold{y}$ lies in $\mathcal{F}_b ^{-B}(\mathcal{H}_0)$, it must be the case that $(\bold{y})_i = (\bold{y})^i$ for each $1\le i \le k$.

Let $x_i$, $y_i$, denote the point of $\alpha_i\cap \beta_i$ in the region $S_{1/2}$ of higher, respectively lower, Maslov grading.
For each $i\le k$ there are no disks leaving $x_i$ which contribute to the differential. There are exactly two disks from $y_i$ to $x_i$. One disks passes through the point $z_i$, the other through $z_{\sigma (i)}$.

\end{proof}

\begin{proposition}
\label{prop:rank}
If $Y$ is a $\mathbb{Q}HS^3$, then $H_{top}(\mathcal{F}_b ^{-B}(\mathcal{H}_0)) \simeq \mathbb{F}[U_1,\dots , U_m]$ and is generated by the homology class of $\bold{x}_0 \in \mathbb{T}_{\boldsymbol{\beta}} \cap \mathbb{T}_{\boldsymbol{\alpha}}$.
\end{proposition}
\begin{proof}

$\widehat{t}(B)=[\bold{x}_\mathcal{D}] \in \widehat{HFK}(-Y,B,-r)$ is nonzero (Theorem \ref{thm:nonvanishing}).
By Proposition 2.2 of \cite{linkgenus} $\widehat{HFK}(-Y,B,-r)$ is rank one.

It follows that $H_* (\mathcal{F}_r ^{-B}(\widetilde{\mathcal{D}}))\simeq \widehat{HFK}(-Y,-B,r)$ is generated by $[\bold{x}_{\widetilde{\mathcal{D}}}]$, the homology class of a generator having components same as $\bold{x}_\mathcal{D}$ of Figure \ref{fig:three}.

In the previous lemma $\bold{x}_0$ is identified with $(x_1\otimes \dots \otimes x_k)\otimes \bold{x}_{\widetilde{\mathcal{D}}}$. Note that there are obvious domains (unions of index one disks) having positive Maslov index from the intersection point $\bold{x}_0$ to any other point of $V_1\otimes V_2\otimes \dots \otimes V_k$.

Thus $[\bold{x}_0]$ generates $H_{top}(\mathcal{F}_b ^{-B}(\mathcal{H}_0)) \simeq \mathbb{F}[U_1,\dots , U_m]$.

%The last part will follow from the next three lemmas. See them for the definitions of the maps and intersection points.
%The composition $F_{1,2}\circ F_{0,1}$ commutes with each of the maps $\{ \phi_z | z\in \bold{z}_{-B} \}$. Thus the summand $\bigcap_{z\in \bold{z}_{-B}} coker (\phi _z)$ of $H_{top}(\mathcal{F}_b ^{-B}(\mathcal{H}_0))$ is identified with the corresponding summand of $H_{top}(\mathcal{F}_b ^{-B}(\mathcal{H}_2))$. This composition sends $\bold{x}_0$ to $\bold{x}_2$, which lies in $\bigcap_{z\in \bold{z}_{-B}} coker (\phi _z)$.
\end{proof}

\begin{remark}
\label{remark:maslov}
The assumption that $Y$ is a $\mathbb{Q}HS^3$ is in place so that the absolute Maslov $\mathbb{Q}$ grading is defined. This technical assumption may be replaced with the assumption that $\mathfrak{s}_{\xi}$ is torsion. Alternatively, one may consider the basepoint action $\psi_{w}$ on $H_*(\mathcal{F}_b ^{-B}(\mathcal{H}_0))$ for each $w \in \bold{w}$. Then $[\bold{x}_0]$ generates $\bigcap_{w\in \bold{w}} coker(\psi_{w})\simeq \mathbb{F}[U_1,\dots , U_m]$.

\end{remark}

We now relate the class $[\bold{x}_0]$ to $t(K)$.\ 

Consider the triple diagram $(\Sigma, \boldsymbol{\beta}', \boldsymbol{\beta}, \boldsymbol{\alpha}, \bold{w}, \bold{z} \cup \bold{z}_{-B})$ shown in Figure \ref{fig:triangleone}. The set of curves $\boldsymbol{\beta}'$ is handleslide equivalent to $\boldsymbol{\beta}$ in the complement of all basepoints. 
For \begin{align*}
2\le i\le 2g+k\hspace{1cm}\text{ and}\\
2g+k+1\le j\le 2g+n+k-1,
\end{align*} $\beta_i '$, $\beta_{2,j}'$, is gotten by applying a small isotopy to $\beta_i$, $\beta_{2,j}$, respectively. 
Let $\mathcal{H}_1 = (\Sigma, \boldsymbol{\beta}', \boldsymbol{\alpha}, \bold{w} , \bold{z} \cup \bold{z}_{-B})$. Let $\boldsymbol{\Theta}$ denote the generator of $CFK^{-,n}(\Sigma,\boldsymbol{\beta}',\boldsymbol{\beta},  \bold{w} , \bold{z} \cup \bold{z}_{-B})$ in the top Maslov grading.\\

%DRAW TRIPLE DIAGRAM, CLASSES X0 AND X1

\begin{figure}[h]
\def\svgwidth{300pt}
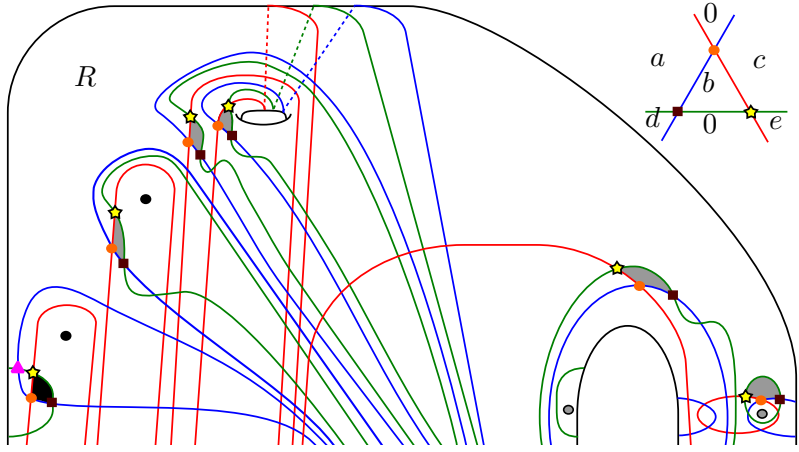
\caption{The $\boldsymbol{\alpha}$,$\boldsymbol{\beta}$ and $\boldsymbol{\beta '}$ curves are red, blue, and green, respectively. The generators $\bold{x}_0$, $\boldsymbol{\Theta}$ and $\bold{x}_1$ are represented by orange dots, brown squares, and yellow stars, respectively. The $\bold{w}$ and $\bold{z}_{-B}$ basepoints are depicted with solid and grey dots, respectively.}
\label{fig:triangleone}
\end{figure}

Consider the triple diagram $(\Sigma, \boldsymbol{\beta}', \boldsymbol{\alpha}, \boldsymbol{\alpha}',\bold{w} , \bold{z} \cup \bold{z}_{-B})$ shown in Figure \ref{fig:triangletwo}. The set of curves $\boldsymbol{\alpha}'$ is handleslide equivalent to $\boldsymbol{\alpha}$ in the complement of all basepoints. 
For 
\begin{align*}
2\le i\le 2g+k\hspace{1cm} \text{ and}\\
2g+k+1\le j\le 2g+n+k-1,
\end{align*} $\alpha_i '$, $\alpha_{2,j}'$, is gotten by applying a small isotopy to $\alpha_i$, $\alpha_{2,j}$, respectively. Let $\mathcal{H}_2 = (\Sigma, \boldsymbol{\beta}', \boldsymbol{\alpha}', \bold{w} , \bold{z} \cup \bold{z}_{-B})$. Abusing notation, let $\boldsymbol{\Theta}$ denote the generator of $CFK^{-,n}(\Sigma,\boldsymbol{\alpha},\boldsymbol{\alpha}',  \bold{w} , \bold{z} \cup \bold{z}_{-B})$ in the top Maslov grading.

%DRAW TRIPLE DIAGRAM, CLASSES X1 AND X2
\begin{figure}[h]
\def\svgwidth{300pt}
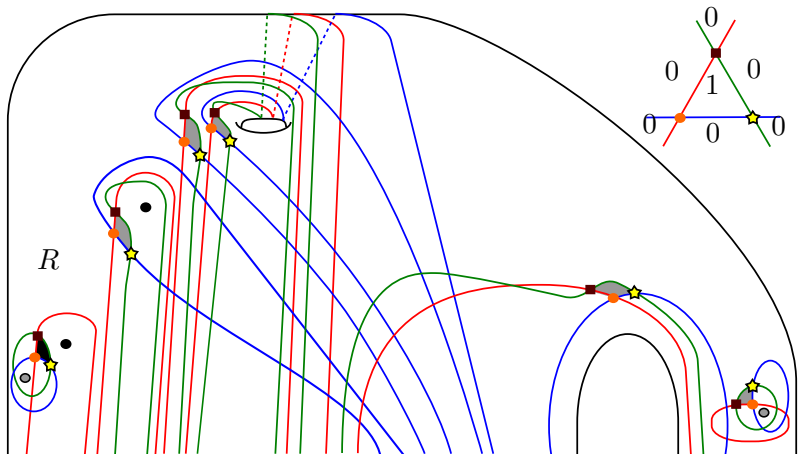
\caption{The $\boldsymbol{\alpha}$,$\boldsymbol{\beta} '$ and $\boldsymbol{\alpha} '$ curves are red, blue, and green, respectively. The generators $\bold{x}_1$, $\boldsymbol{\Theta}$ and $\bold{x}_2$ are represented by orange dots, brown squares, and yellow stars, respectively. To make the diagram simpler, we have performed an isotopy of $\boldsymbol{\alpha}$ and $\boldsymbol{\beta}'$ curves and moved the leftmost basepoint $z_{-B}$ in the complement of said curves (before drawing the $\boldsymbol{\alpha}'$ curves).}
\label{fig:triangletwo}
\end{figure}

\begin{proposition}
\label{prop:maps}
Let
\[
\begin{split}
F_{0,1} : HFK^{-,n}(\mathcal{H}_0) \xrightarrow {\simeq} HFK^{-,n}(\mathcal{H}_1) &\\
F_{1,2} : HFK^{-,n}(\mathcal{H}_1) \xrightarrow {\simeq} HFK^{-,n}(\mathcal{H}_2)
\end{split}
\]
denote the isomorphisms induced by the triple diagrams above.
The composition $F_{0,2} = F_{1,2}\circ F_{0,1}$ sends the class $[\bold{x}_0]$ to the class $[\bold{x}_2]$.
\end{proposition}
\begin{proof}
The isomorphisms $F_{0,1}$ and $F_{1,2}$ are induced by pseudo-holomorphic triangle counting maps $f_{0,1}$ and $f_{1,2}$ respectively. 
We first prove that $f_{0,1} (\bold{x}_0) = \bold{x}_1$.

Suppose that $u\in \pi_2 (\boldsymbol{\Theta},\bold{x}_0,\bold{y})$ is a Whitney triangle of Maslov index one which admits a pseudo-holomorphic representative. We claim that $\bold{y} = \bold{x}_1$ and that the domain $D(u)$ is a disjoint union of small triangles pictured in Figure \ref{fig:triangleone}. In this case $u$ has a unique pseudo-holomorphic representative. We prove this claim by analyzing the multiplicities of $D(u)$ near the generators $\bold{x}_0$ and $\boldsymbol{\Theta}$. First, we analyze the multiplicities near the small triangle shaded black in Figure \ref{fig:triangleone}. 

The diagram in the upper right of the figure shows the local multiplicities in the regions near the triangle. The region just outside the triangle adjacent to the $\boldsymbol{\beta}'$ curve contains a basepoint of $\bold{w}$, so the local multiplicity in this region is zero. The region opposite to the triangle at the corner having a component of $\bold{x}_0$ contains a basepoint $z_{-B}$, so the local multiplicity is zero there as well. Since $\boldsymbol{\Theta}$ and $\bold{x}_0$ are corners of $D(u)$ it follows that $b+d = a+1$ and $b=a+c+1$. Subtracting the second equation from the first we get that $d=-c$. Because $u$ admits a pseudo-holomorphic representative all multiplicities of $D(u)$ must be non-negative; it follows that $d=c=0$ and $b=a+1$. If the component of $\bold{x}_1$ on the vertex of this small triangle is not a corner of $D(u)$, it follows that $b+e =0$, which in turn implies $b=e=0$ and $a=-1$, a contradiction. 

Let $p\in \Sigma$ be the point denoted by a pink triangle in Figure \ref{fig:triangleone}. We have already shown that the multiplicities in three of the regions (all but $R$) which have a corner at $p$ are equal to zero. Since $p$ is not a corner of $D(u)$, it follows that the multiplicity in the region labeled $R$ is also equal to zero.

The multiplicities of regions near all the shaded small triangles now are identical to that of the multiplicities near the black triangle studied above. The claim follows.

To prove that $f_{1,2} (\bold{x}_1) = \bold{x}_2$ one uses a similar argument to show that the multiplicities of the domain of any triangle $v\in \pi_2 (\bold{x}_1,\Theta,\bold{y})$ admitting a pseudo-holomorphic representative are equal to $1$ in each shaded triangle of Figure \ref{fig:triangletwo} and zero elsewhere. As before, one must study multiplicities around the black triangle first to see that the multiplicity in the region labelled $R$ must be zero.

\end{proof}

%DRAW PICTURE FOR DIAGRAM $\mathcal{T}$ and generator $\bold{x}$.

Observe that in $\mathcal{H}_2$ we have small configurations about each point of $\bold{z}_{-B}$. 
By performing $n$ free index 0/3 destabilizations to remove all points of $\bold{z}_{-B}$ we obtain a new diagram $\mathcal{T}$ and generator $\bold{x}$, see Figure \ref{fig:diagramT}. Let $T$ denote the diagram used in the definition of the BRAID invariant when using the basis for $S\smallsetminus \{p_1,\dots, p_k\}$ pictured on the left half of Figure \ref{fig:basisanddiagram}. We modify $T$ to obtain $\mathcal{T}$ by applying some finger moves along the $\beta$-curves in the region $-S_0$, which in turn may be realized via an isotopy of $\widehat{\phi}$. Since $t(K)$ is invariant under isotopy of $\widehat{\phi}$ (\cite{equiv}) we have that $[\bold{x}]=t(K)\in HFK^{-}(\mathcal{T})$.

 \begin{figure}[h]
\def\svgwidth{200pt}
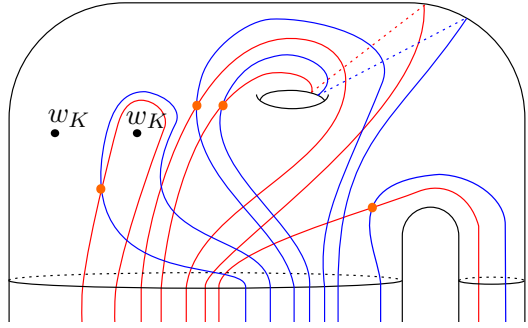
\caption{The diagram $\mathcal{T}$. The generator $\bold{x}$ is depicted with orange dots.}
\label{fig:diagramT}
\end{figure}

 \begin{figure}[h]
\def\svgwidth{250pt}
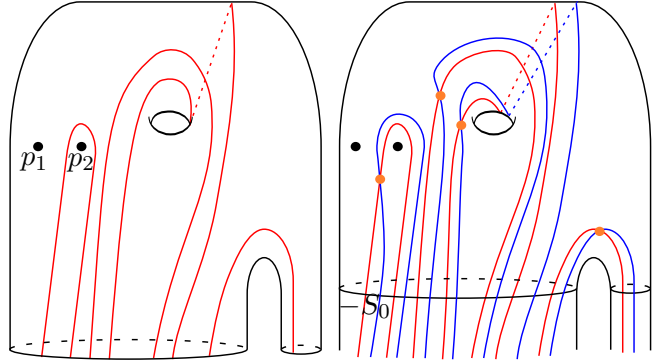
\caption{To the basis on the left we may associate the diagram on the right. Applying finger moves along the $\boldsymbol{\beta}$ curves in the region $-S_0$, which corresponds to isotopy of the monodromy, results in the diagram of Figure \ref{fig:diagramT}.}
\label{fig:basisanddiagram}
\end{figure}

The compositions of projection and inclusion maps
\[
\begin{split}
j^n : CFK^{-,n} (\mathcal{H}_2)\to CFK^- (\mathcal{T}) & \\  i^n : CFK^- (\mathcal{T})  \to  CFK^{-,n} (\mathcal{H}_2)
\end{split}
\] 
defined in subsection \ref{subsec:HFK}, send generators $\bold{x}_2$ to $\bold{x}$ and $\bold{x}$ to $\bold{x}_2$, respectively. In light of Proposition \ref{prop:maps} the following is evident:
\begin{proposition}
\label{prop:composition}
The compositions
\[
\begin{split}
(j^n)_*\circ F_{0,2} : HFK^{-,n} (\mathcal{H}_0)\to HFK^- (\mathcal{T}) & \\  
 F_{0,2}^{-1}\circ(i^n)_* : HFK^- (\mathcal{T})  \to  HFK^{-,n} (\mathcal{H}_0)
\end{split}
\] 
send $[\bold{x}_0]$ to $[\bold{x}] = t(K)$ and $[\bold{x}]=t(K)$ to $[\bold{x}_0]$, respectively. In particular, $[\bold{x}_0]$ lies in the summand $\bigcap_{z\in \bold{z}_{-B}} coker (\psi _z)$.
\end{proposition}

\section{The BRAID invariant $t$ and Rational Open Books}
\label{sec:rationalchar}

In this short section we show that the reformulation of $t(K)$ of the previous section generalizes to braids about rational open books having connected binding.
Let $(B,\pi)$ denote a rational open book decomposition for $Y$ with connected binding, such an open book supports a unique contact structure (see Theorem 1.7 \cite{CCSMCM}) $\xi$.
Let $K\subset Y$ denote a link braided about $B$ with $m$ components.
As in the integral case, $K$ is naturally a transverse link in $(Y,\xi)$.\\

We may choose a diagram \[
(\Sigma,\boldsymbol{\beta_0}',\boldsymbol{\alpha}_0,w_{-B},z_{-B}) \text{ for }(-Y,-B)
\] 
such that $\beta_0\subset \boldsymbol{\beta_0}'$ is a meridian for $-B$, and such that $\beta_0$ intersects only $\alpha_0$ among all $\boldsymbol{\alpha}_0$ curves, and does so in a single point. In proving Proposition 3.1 of \cite{QOB}, Hedden and Plamenevskaya perform some finger moves to $\beta_0'$, getting a new curve $\beta_0$, this gives us a new Heegaard diagram
\[
\mathcal{B}=(\Sigma,\boldsymbol{\beta_0},\boldsymbol{\alpha}_0,w_{-B},z_{-B}) \text{ for }(-Y,-B)
\] 
such that replacing $w_{-B}$ with another point $w_{-B'}$ results in a diagram 
\[\mathcal{B}' = (\Sigma, \boldsymbol{\beta}_0,\boldsymbol{\alpha}_0,w_{-B'},z_{-B}) \text{ for } (-Y,-B'),
\] 
where $B'$ is some $(P,Pn+1)$ genuinely fibered cable of $B$. Using relative periodic domains for longitudes of $B$ and $B'$ to study the Alexander gradings, Hedden and Plamenevskaya identify the complexes $\widehat{CFK}(\mathcal{B},bot)$ and $\widehat{CFK}(\mathcal{B}',bot')$ with each other. 

By starting with a diagram \[
(\Sigma,\boldsymbol{\beta}',\boldsymbol{\alpha},\bold{w} \cup w_{-B}, \bold{z} \cup z_{-B}) \text{ for }(-Y,K\cup -B)
\] 
such that $\beta_0\subset \boldsymbol{\beta}'$ is a meridian for $-B$, and such that $\beta_0$ intersects only $\alpha_0$ among all $\boldsymbol{\alpha}$ curves, and does so in a single point, we may perform the same finger moves to $\beta_0$ as in the proof of Proposition 3.1 of \cite{QOB}, getting a diagram
\[\widetilde{\mathcal{D}}=(\Sigma, \boldsymbol{\beta}, \boldsymbol{\alpha},\bold{w} \cup w_{-B}, \bold{z} \cup z_{-B})\text{ for }(-Y,K\cup -B)\]
such that replacing $w_{-B}$ with another point $w_{-B'}$ results in a diagram 
\[\widetilde{\mathcal{D}}' =(\Sigma,\boldsymbol{\beta},  \boldsymbol{\alpha},\bold{w} \cup w_{-B}', \bold{z} \cup z_{-B})\text{ for }(-Y,K\cup -B').\]
Let $\mathcal{D}=(\Sigma,\boldsymbol{\beta},  \boldsymbol{\alpha},\bold{w} , \bold{z} \cup z_{-B})$. Both $-B$ and $-B'$ induce filtrations on $CFK^{-,1}(\mathcal{D})$. Copying the proof of Proposition 3.1 of \cite{QOB}, with minor changes in notation, shows that $\mathcal{F}_{bot}^{-B}(\mathcal{D}) = \mathcal{F}_{bot'}^{-B'}(\mathcal{D})$ as complexes.

\begin{proposition}
\label{prop:characterization}
Let $B\subset Y$ be a rationally fibered knot, with $Y$ a $\mathbb{Q}HS^3$. Let $K$ be a link braided about $B$ with $m$ components. 

Let $\widetilde{\mathcal{G}} = (\Sigma, \boldsymbol{\beta}, \boldsymbol{\alpha}, \bold{w} \cup w_{-B}, \bold{z} \cup z_{-B})$ be a Heegaard diagram for $(-Y,K\cup -B)$ with a pair of basepoints encoding $-B$. $\mathcal{G}_0= (\Sigma, \boldsymbol{\beta}, \boldsymbol{\alpha}, \bold{w} , \bold{z}\cup z_{-B})$ is a Heegaard diagram for $(-Y,K)$ with $1$ free basepoint. $-B$ induces an Alexander filtration on the complex $CFK^{-,1}(\mathcal{G}_0)$
\[
\emptyset = \mathcal{F}_i^{-B}(\mathcal{G}_0) \subset \mathcal{F}_{i+1}^{-B}(\mathcal{G}_0)  \subset \dots \subset \mathcal{F}_j^{-B}(\mathcal{G}_0)  =CFK^{-,1} (\mathcal{G}_0)
\]

$H_{top} (\mathcal{F}_{bot}^{-B}(\mathcal{G}_0))$ is a rank one $\mathcal{F}[U_1,\dots,U_m]$-module.

\end{proposition}

\begin{proof}

%, let $i: \mathcal{F}_{bot}^{-B}(\mathcal{D}) \to \mathcal{F}_{bot'}^{-B'}(\mathcal{D})$ denote the identification.% In particular, there is an isomorphism
%\[
%F: H_{top}(\mathcal{F}_{bot}^{-B}(\mathcal{D})) \simeq H_{top}( \mathcal{F}_{bot'}^{-B'}(\mathcal{D}))
%\]
%which commutes with the map $\phi_{z_{-B}}$.

The diagram $\widetilde{\mathcal{G}_0}$ is related to $\widetilde{\mathcal{D}}$ by a sequence of handleslides and isotopies avoiding all basepoints, together with index 1/2 (de)stabilizations and linked index 0/3 (de)stabilizations involving only the basepoints in $\bold{w} \cup \bold{z}$. This sequence of moves give rise to chain maps, the composition of which %induces an isomorphism
\[
g:CFK^{-,1}(\mathcal{G}_0)\to CFK^{-,1}(\mathcal{D})%G: H_{top} (\mathcal{F}_{bot}^{-B}(\mathcal{G}))\simeq H_{top}(\mathcal{F}_{bot}^{-B}(\mathcal{D}))
\]
respects the Maslov grading and the filtration induced by $-B$.\\ 
\indent Similarly, the diagram $\widetilde{\mathcal{D}}'$ is related to $\widetilde{\mathcal{H}_0}$ of the previous section by a sequence of handleslides and isotopies avoiding all basepoints, together with index 1/2 (de)stabilizations and linked index 0/3 (de)stabilizations involving only the basepoints in $\bold{w} \cup \bold{z}$. This again gives rise to a chain map
\[
h: CFK^{-,1}(\mathcal{D}) \to CFK^{-,1}(\mathcal{H}_0)
\]
which also respects the Maslov grading and the filtration induced by $-B'$.\\ 
\indent 
The isomorphism
\[
(h\circ g)* :H_{top} (\mathcal{F}_{bot}^{-B}(\mathcal{G}_0))\to H_{top}(\mathcal{F}_{bot'} ^{-B'}(\mathcal{H}_0))
\]
combined with Proposition \ref{prop:rank}, shows that $H_{top} (\mathcal{F}_{bot}^{-B}(\mathcal{G}_0))$ is rank 1.
\end{proof}

The following is now evident:
\begin{lemma}
\label{lem:rationalmap}
Let
\[
Q : HFK^{-,1}(\mathcal{G}_0) \to HFK^{-,1}(\mathcal{H}_0)
\] 
denote the isomorphism induced on homology by $h\circ g$.
Suppose that $[\bold{x}^{G}_0]$ generates $H_{top} (\mathcal{F}_{bot}^{-B}(\mathcal{G}_0))$, then $Q ([\bold{x}^{G}_0])=([\bold{x}_0])$.
\end{lemma}

Because the free index 0/3 (de)stabilization maps commute with all of the maps above, the reformulation of the BRAID invariant of the previous section extends to the case of braids about rational open books having connected binding.

\begin{remark}

The assumption that $Y$ is a $\mathbb{Q}HS^3$ can be weakened, see Remark \ref{remark:maslov}. 
The case of disconnected rational binding can be dealt with similarly, although it is a bit more work. We have no use for this in the present paper, so we do not pursue this here.

\end{remark}

%Let $\Phi: HFK^{-,1}(\mathcal{G})\to HFK^{-,1}(\mathcal{H}_0)\ $ be the isomorphism associated to the reversed composition above, the fact that it commutes with $\phi_{z_{-B}}$ will be of use to us later.

\section{A diagram for lens space braids}

\label{sec:diagram}

In this section we construct a Heegaard diagram $\mathcal{T}^G$ for $(-L(p,q),K)$ for a link $K$ braided about about the standard rational open book $(B,\pi)$ supporting $(L(p,q),\xi_{UT})$ described in Subsection \ref{subsec:contact}. We use the reformulation of the transverse invariant to identify a generator $\bold{x}^G \in CFK^{-}(\mathcal{T}^G)$ whose homology class is $t(K)$.

Recall that the rational open book $(B,\pi)$ has disk fibers, the monodromy is a counterclockwise $2\pi q/p$ boundary twist, and that this rational open book is obtained from the standard open book for $S^3$ having disk pages by performing $-p/q$ surgery on the unknot $U\subset S^3$. The binding $B$ is the core of the filling torus.

Let $K' \subset S^3$ denote the pre-image of $K$ under the Dehn-surgery described above. $K'$ is braided about $U$, suppose that $K'$ has braid index $k$. Consider the Heegaard diagram pictured on the left side of Figure \ref{fig:2diagrams} for $(-S^3,K'\cup U)$, this is just the diagram from Figure \ref{fig:nine} in the case of $g=0, n=1$ where $K'$ is a trivial 3-braid (in general the $\boldsymbol{\beta}$ curves may look different in $-S_0$, the bottom half of the diagram.).

%INSERT FIGURE FOR THIS FIRST DIAGRAM IN THE CASE THAT MONODROMY IS IDENTITY
\begin{figure}[h]
\def\svgwidth{400pt}
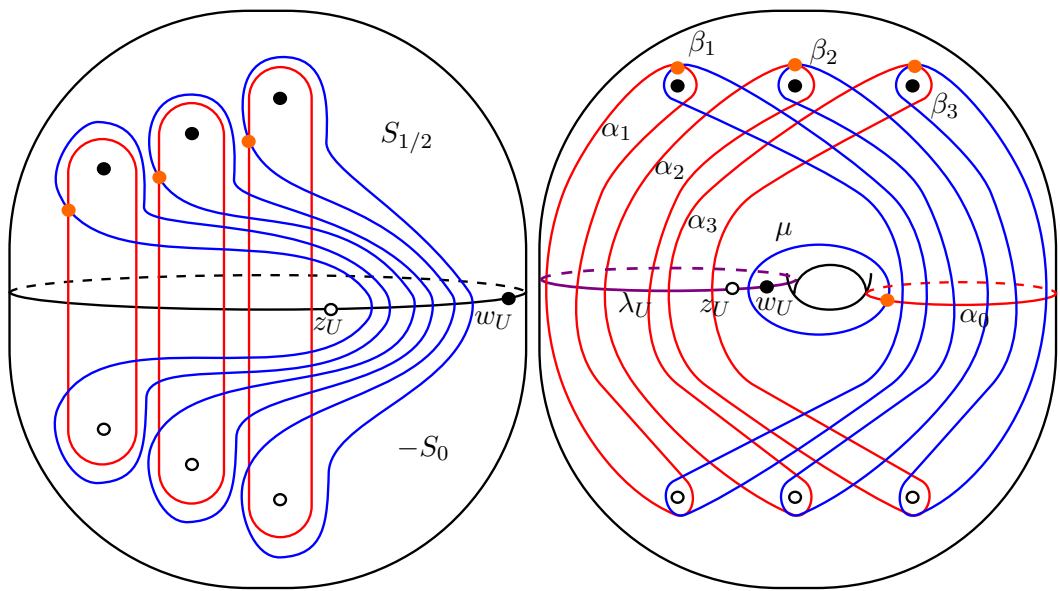
\caption{Two diagrams for $(-S^3,K\cup U)$. The second is obtained from the first by stabilization and handlesliding all the old beta curves over the new beta curve $\mu$. The $\bold{w}\cup w_U$ and $\bold{z}\cup z_U$ basepoints are depicted with solid and hollow dots, respectively.}
\label{fig:2diagrams}
\end{figure}

Stabilizing the diagram and performing a series of handle-slides we obtain the diagram pictured on the right side of Figure \ref{fig:2diagrams}. A longitude for $U$, denoted $\lambda _U$ is also pictured in the diagram. The curves $\alpha_0$ and $\lambda_U$ divide the Heegaard torus into two large regions, let $A$ denote the top region and $\overline{A}$ the bottom region, these correspond to the regions $S_{1/2}$ and $-S_0$, respectively. For a nontrivial braid, the $\boldsymbol{\beta}$ curves will look different in the region $\overline{A}$.

%COMMENT
\begin{comment}
\begin{remark}
\label{remark:pages}
$U$ is the core of the $\boldsymbol{\beta}$ solid torus, so the exterior of $U$ is the $\boldsymbol{\alpha}$ solid torus. The pages of the open book intersect the $\boldsymbol{\alpha}$ solid torus in disks and meet the Heegaard torus in curves parallel to $\alpha_0$, this is one way of seeing that $K$ is braided about $U$ from the diagram.
\end{remark}
\end{comment}
%COMMENT

%INSERT FIGURE FOR HEEGAARD DIAGRAM IN CASE THAT MONODROMY IS IDENTITY LABEL CURVES a0 b0 a1 b1 basepoints generator AND longitude for binding
\begin{figure}[h]
\def\svgwidth{400pt}
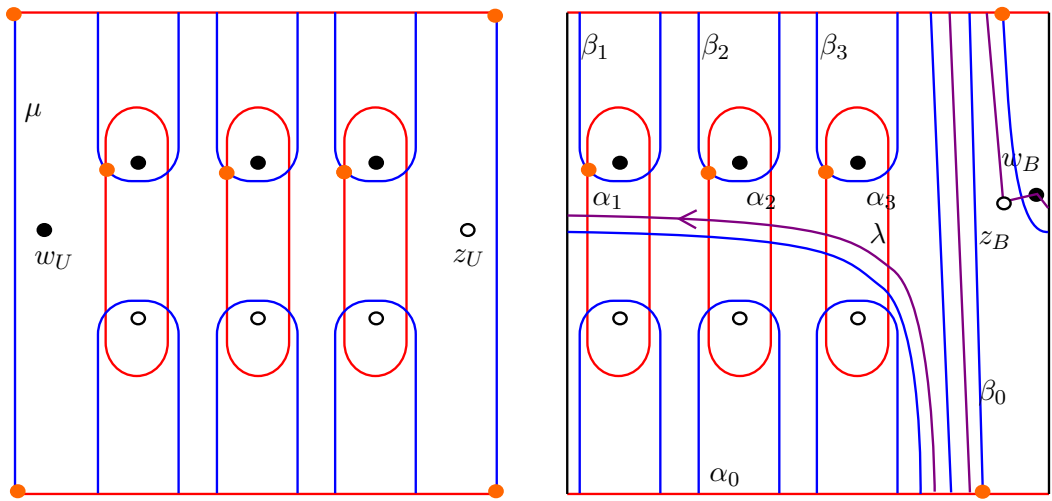
\caption{A diagram for a 3-braid in $L(3,1)$ is pictured on the right. A longitude $\lambda$ for $-B$ is pictured in purple.}
\label{fig:planardiagrams}
\end{figure}

%SAY SOMETHIGN ABOUT ADMISSIBILITY

%Let $\beta_0$, $\beta_1$, $\alpha_0$ and $\alpha_1$ denote the curves labelled in diagram CITE DIAGRAM. 
\begin{remark} 
The generator pictured on the right half of Figure \ref{fig:2diagrams} can be shown to represent $t(K'\cup U)$ using the reformulation of section \ref{subsec:reformulation}. Forgetting the basepoints $z_U$ and $w_U$, along with the pair of curves $\beta_1$ and $\alpha_1$ we get a diagram for $(-S^3,K')$. The generator representing $t(K')$ is easily identified as well, again using the reformulation. 
\end{remark}

We prefer to draw the diagram on a fundamental domain for the torus, see left side of Figure \ref{fig:planardiagrams}. The curve $\mu$ is a meridian for $U$. We may obtain a Heegaard diagram for $(-L(p,q),B\cup K)$ by replacing $\mu$ with another curve, $\beta_0$, and replacing basepoints $z_U$ and $w_U$ with $z_B$ and $w_B$, see right side of Figure \ref{fig:planardiagrams}.

The curves $\beta_0$, $\alpha_0$ and $\lambda_U$ cut $\Sigma$ into a large region $\Delta$ lying in $A$, a large region $\overline{\Delta}$ lying in $\overline A$, and a number of smaller regions. Depending on the braid, the $\boldsymbol{\beta}$ curves may look different in the region $\overline{\Delta}$.

%INSERT FIGURE FOR SURGERY. ALONG WITH CURVES b0 b1 a0 a1 basepoints labelled. and a longitude for binding

Setting $z_{-B}=w_{B}$ and $w_{-B}=z_{B}$ we reverse the orientation of $B$. Forgetting the basepoint $w_{-B}$ we obtain a diagram $\mathcal{G}_0 = (\Sigma, \boldsymbol{\beta},\boldsymbol{\alpha}, \bold{w},\bold{z} \cup z_{-B})$ for $(-Y,K)$ with a single free basepoint $z_{-B}$. Let $\bold{x}_0^G \in CFK^{-}(\mathcal{G}_0)$ denote the generator depicted by orange dots. The following is inspired by Proposition 3.4 of \cite{QOB}.

\begin{proposition}
\label{prop:gen}
$[\bold{x}^{G}_0]$ generates $H_{top} (\mathcal{F}_{bot}^{-B}(\mathcal{G}_0))$.

\end{proposition}
\begin{proof}

The curve $\beta_0$ is a meridian for $B$. We may draw a longitude $\lambda$ for $-B$ on $\Sigma$ that is supported in a neighborhood of $\beta_0$, and intersects $\beta_0$ transversely in a single point as pictured in Figure \ref{fig:planardiagrams}. The curve $\beta_0$ is homologous to $-p\mu +q\alpha_0$. $\lambda$ is homologous to $b\mu +a\alpha_0$ for $a$ and $b$ satisfying $pa+qb=-1$. Note that $b\beta_0 +p\lambda$ is homologous to $-\alpha_0$, so we may consider a relative periodic domain $\mathcal{P}$ whose homology class is negative that of the fiber (of the rational open book), having boundary $\alpha_0 +b\beta_0+p\lambda$. The multiplicity of $\mathcal{P}$ is 1 in the region $\overline{\Delta}$ and $0$ in the region $\Delta$.

Lemma \ref{lemma:relperiodic} tells us that a generator $\bold{y} \in \mathbb{T}_{\boldsymbol{\beta}} \cap \mathbb{T}_{\boldsymbol{\alpha}}$ lies in $\mathcal{F}_{bot}^{-B}(\mathcal{G}_0)$ if and only if $n_\bold{y} (\mathcal{P})$ is minimized.

For a generator $\bold{y}$, we let $(\bold{y})_i$, $(\bold{y})^i$ denote the component of $\bold{y}$ on $\alpha_i$, $\beta_i$, respectively.

We first claim that any generator $\bold{y} \in \mathbb{T}_{\boldsymbol{\beta}} \cap \mathbb{T}_{\boldsymbol{\alpha}}$ minimizing $n_\bold{y} (\mathcal{P})$ must satisfy $(\bold{y})_0 = (\bold{y})^0$. Seeking a contradiction, suppose that $\bold{y}$ is a generator minimizing $n_\bold{y} (\mathcal{P})$ such that $(\bold{y})^0 = (\bold{y})_i$ for some $i>0$.

The four regions surrounding $(\bold{y})_0$ have multiplicities $m,m,m+b$ and $m+b$, this is because $\partial \mathcal{P}$ contains $\beta_0$ with multiplicity $b$ and does not contain $\alpha_i$.
There is an arc $\kappa \subset \beta_0$ from $(\bold{y})_0$ to a point $y_0\in\alpha_0\cap\beta_0$ which does not intersect $\alpha_0$ nor $\lambda$ in its interior. Because all of the alpha curves $\kappa$ intersects in its interior have multiplicity zero in $\partial \mathcal{P}$, it is clear that the multiplicities of $\mathcal{P}$ in the four regions surrounding $y_0$ are $m-1,m, m+b-1$ and $m+b$.

It must be the case that $(\bold{y})_0 = (\bold{y})^{j_1}$ for some $j_1 >0$. If $j_1 \ne i$ (in general this can happen, as the beta curves can be twisted in the region $\overline{\Delta}$), it follows that $(\bold{y})_{j_1} = (\bold{y})^{j_2}$ for some $j_2 \ne j_1$. In this way we construct a sequence $\{j_1, j_2,\dots j_t\}$ having length at most $k$ such that for $1\le s< t$, $(\bold{y})_{j_s} = (\bold{y})^{j_{s+1}}$ and $(\bold{y})_{j_t} = (\bold{y})^{i}$. Note that all of these intersections are in $\overline{\Delta}$, and thus have multiplicities equal to 1.

Let $\bold{y}'$ be obtained from $\bold{y}$ by replacing $(\bold{y})_0$ with $y_0$, $(\bold{y})_i$ with $(\bold{x}^{G}_0)_i$, and $(\bold{y})_{j_s}$ with $(\bold{x}^{G}_0)_{j_s}$ for each $1\le s \le t$. We conclude that $n_{\bold{y}} (\mathcal{P}) - n_{\bold{y}'} (\mathcal{P}) = 1+t$, proving the claim.

In the following lemma we will show that $(\bold{x}^{G}_0)^0$ contributes minimally to $n_\bold{y} (\mathcal{P})$ among all points of $\beta_0\cap\alpha_0$. Assuming the lemma for now, we proceed with the proof.

For $n_\bold{y} (\mathcal{P})$ to be minimized, the other components of $\bold{y}$ must be in the region $\Delta$ (otherwise a component is in $\overline{\Delta}$, contributing strictly higher multiplicity). Let $x_i$ and $y_i$ denote the intersections between $\beta_i$ and $\alpha_i$ in the region $\Delta$ having higher and lower Maslov grading, respectively. It is now clear that $\mathcal{F}_{bot}^{-B}(\mathcal{G}_2) \simeq V_1 \otimes V_2 \otimes \dots \otimes V_k$ where each $V_i$ is a free rank two $\mathcal{F}[U_1,\dots U_m]$-module, generated by $x_i$ and $y_i$, such that $\partial x_i = 0$. Since $(\bold{x}^{G}_0)^i=x_i$, the claim is proven.

\end{proof}

\begin{lemma}
$(\bold{x}^{G}_0)^0$ contributes minimally to $n_\bold{y} (\mathcal{P})$ among all points of $\beta_0\cap\alpha_0$.
\end{lemma}
\begin{proof}

Let $\mathcal{D} = (\Sigma, \beta_0,\alpha_0,z_{-B})$. As usual, $-B$ induces a filtration on $\widehat{CFK}(\mathcal{D})$. We aim to prove that $(\bold{x}^{G}_0)^0$ has minimal filtration level among all generators. 

Proposition 3.4 of \cite{QOB} tells us that there is a unique generator $y$ having minimal filtration level, however they do not specify which generator. %Theorem 1 of \cite{QOB} tells us that the homology class of $y$ is the contact invariant of the contact structure $\xi$ supported by the rational open book having binding $B$.

Note that $pa+qb =-1 \implies gcd(b,p) = 1$. Thus we may find some $(r,s)$ cable, with $r>0$, $-\widetilde{B}$ of $-B$ which is homologous to $- \mu$. 

$\widetilde{B}$ is also the binding of some rational open book for $L(p,q)$, moreover Theorem 1.8 of \cite{CCSMCM} tell us that the contact structure supported by this new rational open book is contactomorphic to $\xi$.

$-\widetilde{B}$ also induces a filtration on $\widehat{CFK}(\mathcal{D})$; we claim that $(\bold{x}^{G}_0)^0$ has minimal filtration level among all generators. It will then follow from Theorem 1 of \cite{QOB} that $[(\bold{x}^{G}_0)^0] = [c(\xi)]= [y]$, hence $(\bold{x}^{G}_0)^0 = y$.

Consider the longitude $\widetilde{\lambda}$ for $-\widetilde{B}$ pictured in Figure \ref{fig:multiplicity}. There is a relative periodic domain $\widetilde{P}$, having homology class that of a negative fiber, with boundary $q\alpha_0 -\beta_0 +p\widetilde{\lambda}$. Analyzing the multiplicities of this domain is trivial, it is clear that $n_{(\bold{x}^{G}_0)^0} (\widetilde{P})$ is minimal.

\begin{tiny}
\begin{figure}[h]
\def\svgwidth{150pt}
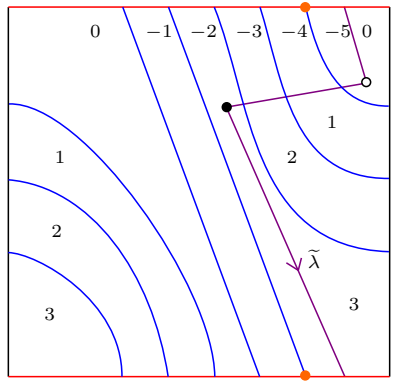
\caption{The diagram $\mathcal{D}$ for $L(5,3)$ is pictured. A longitude $\widetilde{\lambda}$ for $-\widetilde{B}$ is pictured in purple. The multiplicities of $\widetilde{P}$ in each region are shown. The basepoints $w_{-\widetilde{B}}$ and $z_{-B}$ are depicted with solid and hollow dots, respectively.}

\label{fig:multiplicity}
\end{figure}
\end{tiny}
\end{proof}

\begin{figure}[h]
\def\svgwidth{400pt}
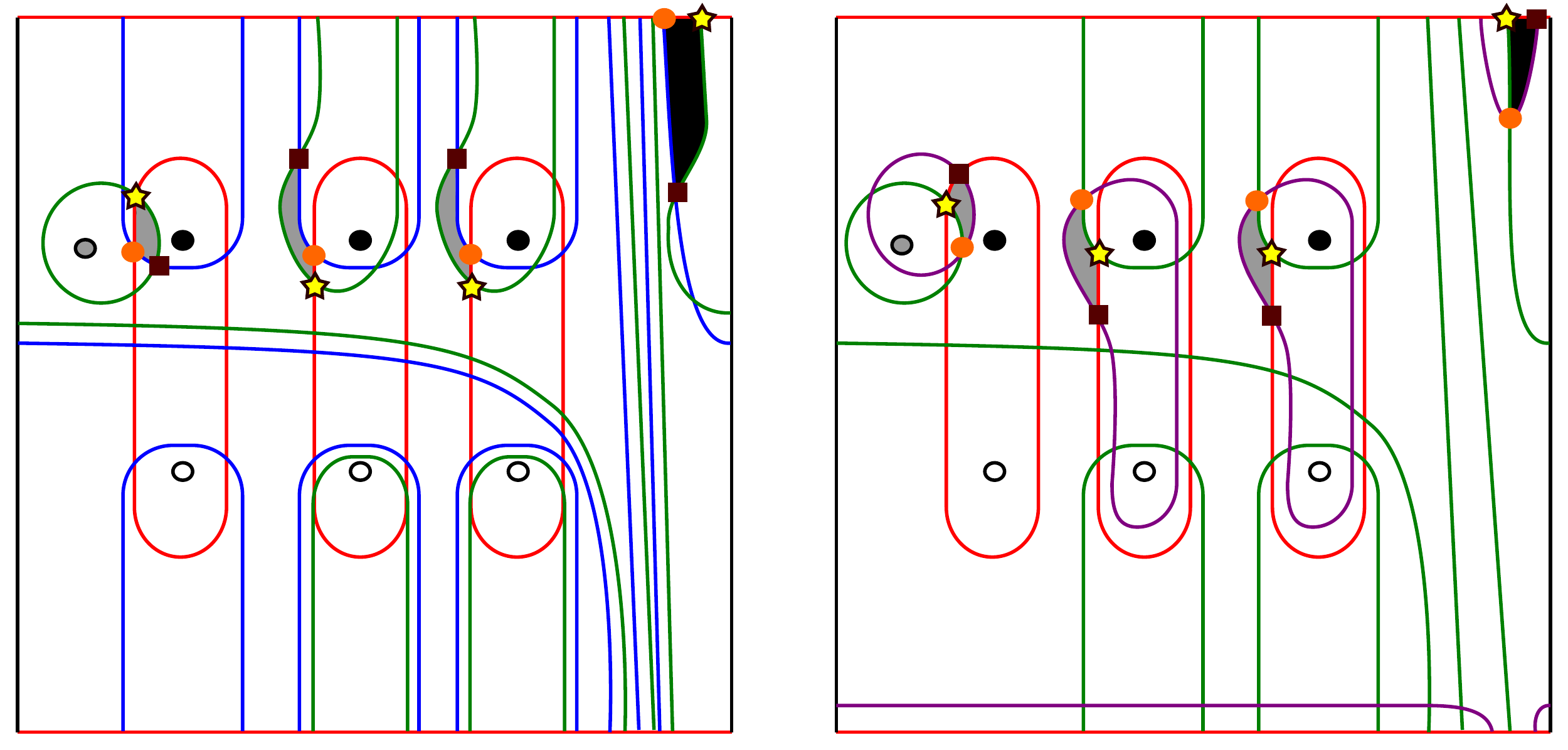
\caption{The triple diagrams $(\Sigma, \boldsymbol{\beta}', \boldsymbol{\beta}, \boldsymbol{\alpha}, \bold{w} , \bold{z} \cup z_{-B})$ and $(\Sigma, \boldsymbol{\beta}', \boldsymbol{\alpha}, \boldsymbol{\alpha}', \bold{w} , \bold{z} \cup z_{-B})$ are shown on the left and right respectively. The $\boldsymbol{\alpha}, \boldsymbol{\alpha}', \boldsymbol{\beta}$ and $\boldsymbol{\beta}'$ curves are drawn red, purple, blue and green respectively. The generators $\bold{x}^{G}_0$ and $\bold{x}^{G}_2$ are depicted with orange dots. The generators $\boldsymbol{\Theta}$ are depicted with brown squares. The generator $\bold{x}^{G}_{1}$ is depicted with yellow stars. The $\bold{w}$, $\bold{z}$, and $\bold{z}_{-B}$ basepoints are depicted with solid, hollow, and grey dots, respectively.}
\label{fig:somecounts}
\end{figure}

Consider the triple diagrams $(\Sigma, \boldsymbol{\beta}', \boldsymbol{\beta}, \boldsymbol{\alpha}, \bold{w} , \bold{z} \cup z_{-B})$ and $(\Sigma, \boldsymbol{\beta}', \boldsymbol{\alpha}, \boldsymbol{\alpha}', \bold{w} , \bold{z} \cup z_{-B})$ shown in Figure \ref{fig:somecounts}.
The sets of curves $\boldsymbol{\beta}, \boldsymbol{\alpha}$, are handleslide equivalent to $\boldsymbol{\beta}',\boldsymbol{\alpha}'$, respectively, in the complement of all basepoints.

For $i\ne 1$ the curve $\beta_i '$ is obtained from $\beta_i$ by a small isotopy. $\beta_1 '$ is obtained by sliding $\beta_1$ over other $\beta$ curves. The curves $\boldsymbol{\alpha}'$ are obtained from $\boldsymbol{\alpha}$ via handleslides in the same way.

Let $\mathcal{G}_1 = (\Sigma, \boldsymbol{\beta}', \boldsymbol{\alpha}, \bold{w} , \bold{z} \cup z_{-B})$ and $\mathcal{G}_2 = (\Sigma, \boldsymbol{\beta}', \boldsymbol{\alpha}', \bold{w} , \bold{z} \cup z_{-B})$. Abusing notation, let $\boldsymbol{\Theta}$ denote both the top graded generator of $CFK^{-,1}(\Sigma,\boldsymbol{\beta}',\boldsymbol{\beta}, \bold{w},\bold{z}\cup z_{-B})$ and $CFK^{-,1}(\Sigma,\boldsymbol{\alpha},\boldsymbol{\alpha}', \bold{w},\bold{z}\cup z_{-B})$. Let $\bold{x}^{G}_1\in CFK^{-,1}(\mathcal{G}_1)$ denote the intersection point depicted with stars in Figure \ref{fig:somecounts}, and let $\bold{x}^{G}_2\in CFK^{-,1}(\mathcal{G}_2)$ denote the intersection point depicted with orange dots on the right half of the figure.

\begin{proposition}
\label{prop:relate}
Let
\[
\begin{split}
G_{0,1} : HFK^{-,1}(\mathcal{G}_0) \xrightarrow {\simeq} HFK^{-,1}(\mathcal{G}_1) &\\
G_{1,2} : HFK^{-,1}(\mathcal{G}_1) \xrightarrow {\simeq} HFK^{-,1}(\mathcal{G}_2)
\end{split}
\]
denote the isomorphisms induced by the triple diagrams above.
The composition $G_{0,2}=G_{1,2}\circ G_{0,1}$ sends the class $[\bold{x}^{G}_0]$ to the class $[\bold{x}^{G}_2]$.
\end{proposition}

\begin{proof}
The isomorphisms $G_{0,1}$ and $G_{1,2}$ are induced by pseudo-holomorphic triangle counting maps $g_{0,1}$ and $g_{1,2}$ respectively. 
We outline the proof that $g_{0,1} (\bold{x}^{G}_0) = \bold{x}^{G}_1$, proving that $g_{1,2} (\bold{x}^{G}_1) = \bold{x}^{G}_2$ requires an similar argument.

Let $u\in \pi_2 (\boldsymbol{\Theta},\bold{x}^{G}_0,\bold{y})$ be a Whitney triangle having corner at some generator $\bold{y}\in \mathbb{T}_{\boldsymbol{\beta}'}\cap\mathbb{T}_{\boldsymbol{\alpha}}$ which misses the basepoints $\bold{w}$. We claim that $u$ has domain equal to the union of small gray and black triangles pictured in Figure \ref{fig:somecounts}, in which case it has a unique holomorphic representative. We count the multiplicities of the domain of $u$. Using the method presented in the proof of Proposition \ref{prop:maps}, it is immediate that for all $i>0$ we have that $(\bold{y})_i = (\bold{x}^{G}_1)_i$ and that the domain of $u$ contains the small gray triangles.  Because the triple diagram corresponds to the identity cobordism, the induced triangle counting map should preserve $Spin^C$ structure. The only such generator $\bold{y}$ having the correct $Spin^C$ structure is $\bold{x}^{G}_1$, and the only Whitney triangle $u\in \pi_2 (\boldsymbol{\Theta},\bold{x}^{G}_0,\bold{x}^{G}_1)$ having no negative multiplicities is the one desired.

\end{proof}

Note that the diagram $\mathcal{G}_2$ has a small configuration about the point $z_{-B}$. Let $\mathcal{T}^G$ be the diagram obtained by performing the corresponding free 0/3-index destablization. The maps 
\[
\begin{split}
j : CFK^{-,1} (\mathcal{G}_2)\to CFK^- (\mathcal{T}^G) & \\  i : CFK^- (\mathcal{T}^G)  \to  CFK^{-,1} (\mathcal{G}_2)
\end{split}
\] 
are defined in subsection \ref{subsec:HFK}. Let $\bold{x}^G$ denote $j(\bold{x}^G_2)$. 

\begin{theorem}
\label{thm:diagram}
The generator $\bold{x}^G$ has homology class $t(K)$, i.e. $[\bold{x}^G] = t(K)\in HFK^- (-L(p,q),K)$.
\end{theorem}
\begin{proof}
We are now in position to apply the reformulation of section \ref{sec:rationalchar}. Let $Q$ be the map defined in Lemma \ref{lem:rationalmap}. Let $\bold{x}, F_{0,2}, \mathcal{H}_2$ and $\mathcal{T}$ be as in Propositions \ref{prop:maps} and \ref{prop:composition}. Combining that lemma and those propositions with Propositions \ref{prop:characterization} and \ref{prop:gen} we have that the composition
\[
F_{0,2}\circ Q\circ G_{0,2}^{-1}: HFK^{-} (\mathcal{G}_2)\to HFK^{-} (\mathcal{H}_2)
\]
is an isomorphism mapping $[\bold{x}^G_2]$ to $[\bold{x}_2]$.

Moreover, since the maps above commute with the free 0/3 (de)stablization maps, the composition
\[
(j)_*\circ F_{0,2}\circ Q\circ G_{0,2}^{-1}\circ (i)_* : HFK^{-} (\mathcal{T}^G)\to HFK^{-} (\mathcal{T})
\]
is an isomorphism mapping $[\bold{x}^G]$ to $[\bold{x}] = t(K)$.
\end{proof}

We will refer to $\mathcal{T}^\mathcal{G}$ as the \textbf{standard braid diagram} for $K$.

\begin{lemma}
\label{lem:trivial}
Let $\tau_n\in B_n$ denote the trivial braid having index $n$. The Maslov gradings of the GRID and BRAID invariants agree for $\tau_n \circ \delta^{q/p}$. i.e.
\[
M(\theta(\tau_n \circ \delta^{q/p}))=M(t(\tau_n \circ \delta^{q/p}))
\]
\end{lemma}
\begin{proof}
The $\boldsymbol{\beta}$ curves of the standard braid diagram for $\tau_n\circ\delta^{q/p}$ are particularly simple, and we can handleslide to a grid diagram having index $n$. 
Consider the triple diagrams $(\Sigma,\boldsymbol{\beta},\boldsymbol{\alpha},\boldsymbol{\alpha}',\bold{w},\bold{z})$ and $(\Sigma,\boldsymbol{\beta}',\boldsymbol{\beta},\boldsymbol{\alpha'},\bold{w},\bold{z})$ pictured in Figure \ref{fig:maslovtriple}. Here, we have initially isotoped the diagram $\mathcal{T}^G = (\Sigma,\boldsymbol{\beta},\boldsymbol{\alpha}, \bold{w},\bold{z})$ so that the final diagram appears more grid like.

\begin{figure}[h]
\def\svgwidth{400pt}
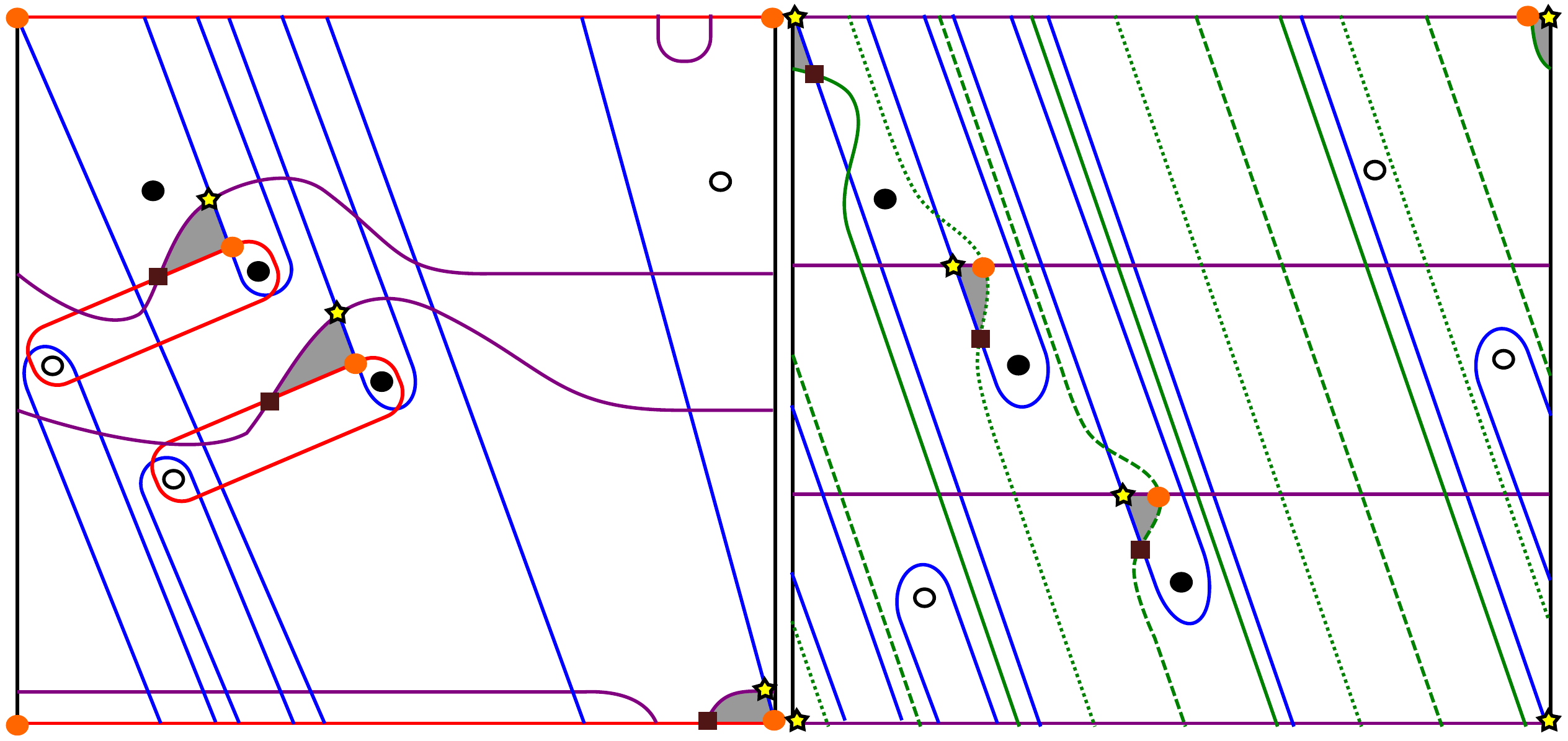
\caption{The triple diagrams $(\Sigma,\boldsymbol{\beta},\boldsymbol{\alpha},\boldsymbol{\alpha}',\bold{w},\bold{z})$ and $(\Sigma,\boldsymbol{\beta}',\boldsymbol{\beta},\boldsymbol{\alpha'},\bold{w},\bold{z})$ pictured left and right. The $\boldsymbol{\beta}',\boldsymbol{\beta},\boldsymbol{\alpha}$ and $\boldsymbol{\alpha}'$ are drawn green, blue, red and purple, respectively. The $\bold{w}$ and $\bold{z}$ basepoints are depicted with solid and hollow dots, respectively.. The generators $\bold{x}^G$ and $\bold{x}$ are depicted with orange dots. The generators $\boldsymbol{\Theta}$ and $\bold{x}'$ are depicted with brown squares and yellow stars, respectively.}
\label{fig:maslovtriple}
\end{figure}

Let $G = (\Sigma,\boldsymbol{\beta}',\boldsymbol{\alpha'},\bold{w},\bold{z})$ and $G' = (\Sigma,\boldsymbol{\beta},\boldsymbol{\alpha'},\bold{w},\bold{z})$. Let $\bold{x}$ and $\bold{x}'$ denote the generators pictured in the figure.
Let $\boldsymbol{\Theta}$ denote both the top graded generator in $CFK^- (\Sigma,\boldsymbol{\beta}',\boldsymbol{\beta},\bold{w},\bold{z})$ and $CFK^- (\Sigma,\boldsymbol{\alpha},\boldsymbol{\alpha}',\bold{w},\bold{z})$.

There is a Maslov index zero Whitney triangle $u \in \pi _2 (\bold{x}^G,\boldsymbol{\Theta},\bold{x}')$ and $u' \in \pi_2 (\boldsymbol{\Theta},\bold{x}',\bold{x})$, whose domain is shaded in Figure \ref{fig:maslovtriple}. It follows that 
\[
M(t(\beta_n\circ\delta^{q/p}))=M(\bold{x}^G)=M(\bold{x})=M(\theta(\beta_n\circ\delta^{q/p}))
\]
\end{proof}

\begin{proposition}
\label{prop:needaname}
For a transverse braid $K\subset (L(p,q),\xi_{UT})$, the Maslov gradings of the GRID and BRAID invariants agree. i.e.
\[
M(\theta(K))=M(t(K)).
\]
\end{proposition}
\begin{proof}

Let $\beta\in B_n$ be an arbitrary index $n$ braid. One easily computes (as in the proof of Proposition \ref{prop:gradingcomparison}) that $M(\theta(\beta\circ\delta^{q/p}))-M(\theta(\tau_n\circ\delta^{q/p}))=w(\beta)$, the writhe. We will show that Maslov grading of the BRAID invariant satisfies the same equation, the result will then follow from Lemma \ref{lem:trivial}.

We compare the Maslov gradings of $t(\beta\circ\delta^{q/p})$ and $t(\tau_n\circ\delta^{q/p})$.

Let $(T^2,\boldsymbol{\gamma},\boldsymbol{\alpha},\bold{w},\bold{z})$ and $(T^2,\boldsymbol{\beta},\boldsymbol{\alpha},\bold{w},\bold{z})$ denote the standard braid diagrams for $\beta\circ\delta^{q/p}$ and $\tau_n\circ\delta^{q/p}$, respectively. We draw all three sets of curves on a single torus (as in Figure \ref{fig:gradingshift}) getting the triple diagram $(T^2,\boldsymbol{\gamma},\boldsymbol{\beta},\boldsymbol{\alpha},\bold{w},\bold{z})$. Note that the diagram $(T^2,\boldsymbol{\gamma},\boldsymbol{\beta},\bold{w},\bold{z})$ can be identified with a diagram used to define $t(\beta)\in HFK^-(-S^3,\beta)$ connect sum the standard Heegaard diagram for $S^1\times S^2$. Let $\Theta\in HF^-(S^1\times S^2)$ denote the generator in top Maslov grading.

Let $u$ denote the Whitney triangle having Maslov index $1-n$ whose domain is shaded in Figure \ref{fig:gradingshift}. 
This Whitney triangle has corners at generators having homology classes $t(\beta)\otimes \Theta, t(\tau_n\circ\delta^{q/p})$ and $t(\beta\circ\delta^{q/p})$.
Using $u$ to compare Maslov gradings as in \cite{absgrading}, we see that
\[
M(t(\beta))+M(t(\tau_n\circ\delta^{q/p}))-M(t(\beta\circ\delta^{q/p}))=1-n.
\]
$M(t(\beta))$ has been computed in \cite{equiv} to equal $sl(\beta)+1 = w(\beta)-n+1$. Combining these equations gives
\[
M(t(\beta\circ\delta^{q/p}))-M(t(\tau_n\circ\delta^{q/p}))=w(\beta).
\]

\begin{figure}[h]
\def\svgwidth{200pt}
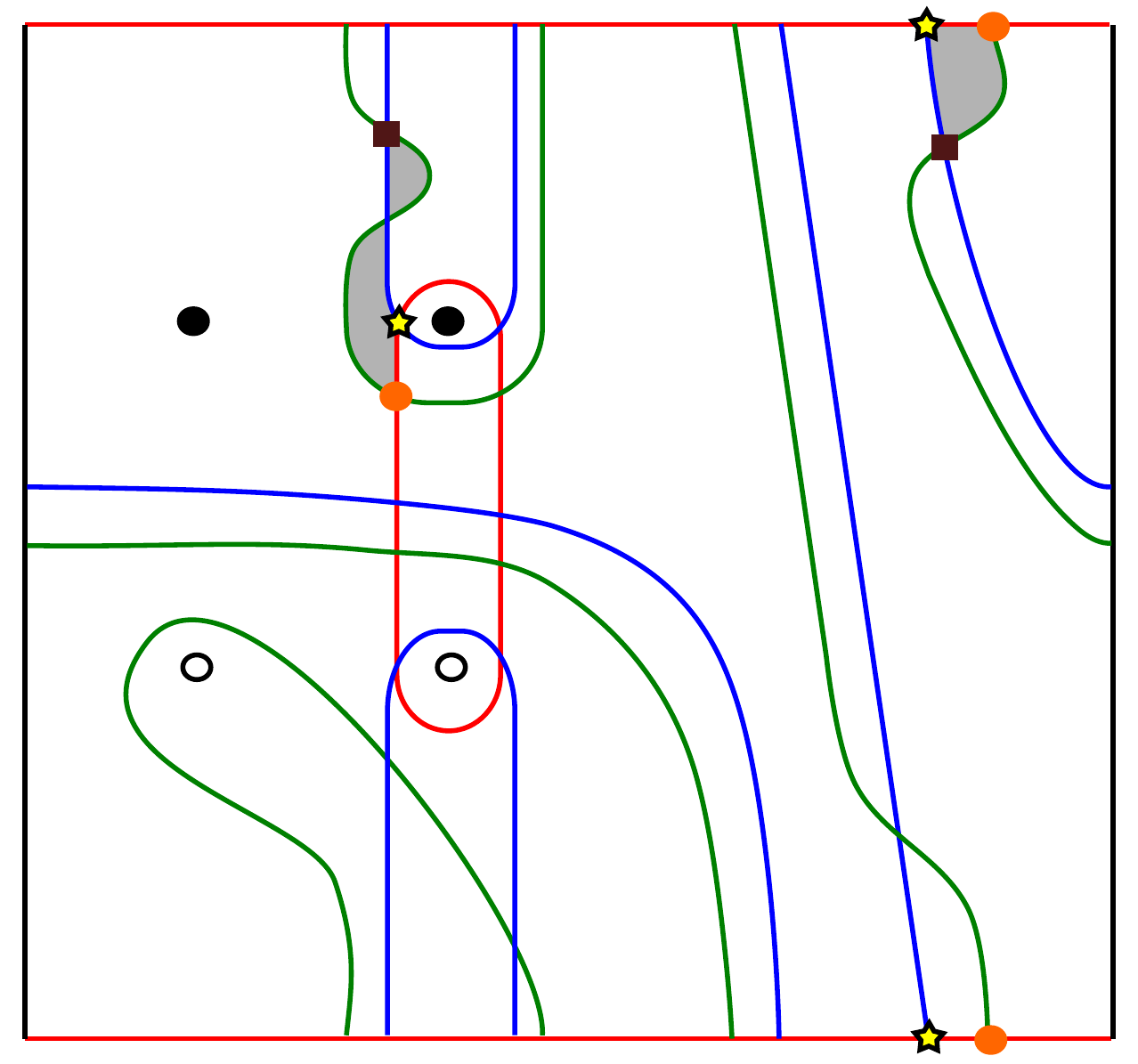
\caption{The homology classes of generators depicted by brown squares, yellow stars and orange dots are $t(\beta)\otimes \Theta,t(\tau_n\circ\delta^{p/q})$ and $t(\beta\circ\delta^{p/q})$, respectively. In this example $(p,q)=(2,1)$ and $\beta = \sigma_1$.}
\label{fig:gradingshift}
\end{figure}

\end{proof}

\section{Another reformulation of the BRAID invariant for lens space braids}
\label{sec:altchara}

Let $(B,\pi)$ denote the rational open book decomposition supporting $(L(p,q),\xi_{UT})$ studied in the previous section, and $K$ an index $k$ braid about $(B,\pi)$. Let $U\subset L(p,q)$ denote the Seifert cable of $B$; this is the cable specified by how a fiber $D$ of $(B,\pi)$ meets the boundary of a solid torus neighborhood of $B$. The braid $K$ intersects $D$ in $k$ points. 

In this section we reformulate the invariant $t(K)\in HFK^{-}(-L(p,q),K)$ in terms of the Alexander filtration induced by $-U$ on the knot Floer chain complex. We will use this alternate reformulation in a subsequent section to prove that the GRID invariant for transverse links in $(L(p,q),\xi_{UT})$ is equivalent to $t(K)$. 

\begin{remark}
In the following reformulation we use two pairs of basepoints to encode $-U$. One may easily formulate and prove a formulation using one pair of basepoints, but two pairs is better suited towards proving GRID = BRAID.
\end{remark}

By adding a basepoint $w_{-U}$ to the diagram $\mathcal{G}_0$ of the previous section, and relabelling $z_{-B}$ as $z_{-U}$ we obtain a Heegaard diagram for $(L(p,q),K\cup -U)$. Adding an extra pair of and curves and basepoints for $-U$ we obtain the diagram $\mathcal{D}$ pictured in Figure \ref{fig:altchara}. Denote the new curves $\alpha^s$ and $\beta^s$. Forgetting the basepoints $\bold{w}_{-U}$ we obtain a diagram \[\mathcal{D}_0 = (\Sigma,\boldsymbol{\beta},\boldsymbol{\alpha}, \bold{w},\bold{z}\cup\bold{z}_{-U})\text{ for }(-L(p,q),K)\] with two free basepoints. Let $\bold{x}_0^D$ denote the generator pictured in the figure.

Let $(\mathcal{F}_{bot}^{-U} (\mathcal{D}_0), \mathfrak{s}_{\xi})$ denote the summand of $\mathcal{F}_{bot}^{-U} (\mathcal{D}_0)$ whose generators $\bold{x}$ satisfy $\mathfrak{s}_{\bold{w}}(\bold{x})=\mathfrak{s}_{\xi}$.
\begin{figure}[h]
\def\svgwidth{200pt}
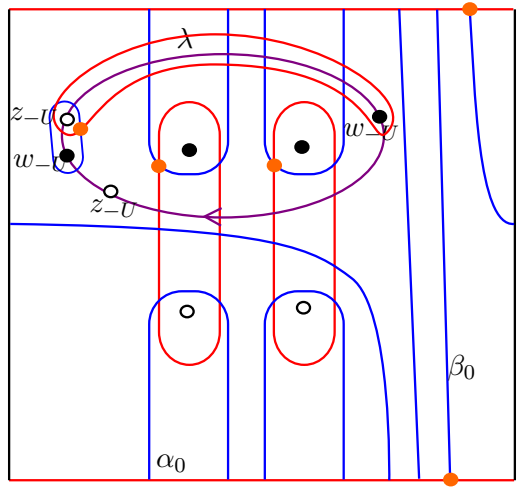
\caption{A Heegaard diagram for $(-L(3,1),K\cup -U)$. A longitude $\lambda$ for $-U$ is pictured in purple. $\mathcal{D}_0 = (\Sigma,\boldsymbol{\beta},\boldsymbol{\alpha},\bold{w},\bold{z}\cup \bold{z}_{-U})$. The components of $\bold{x}_0^D$ are orange dots.}
\label{fig:altchara}
\end{figure}

\begin{lemma}
\label{lem:needaname}
$\mathcal{F}_{bot}^{-U} (\mathcal{D}_0) \simeq V_1 \otimes \dots \otimes V_k \oplus \mathbb{F}_p$ and $(\mathcal{F}_{bot}^{-U} (\mathcal{D}_0), \mathfrak{s}_{\xi}) \simeq V_1 \otimes \dots \otimes V_k$, where each $V_i$ is a free rank two $\mathcal{F}[U_1,\dots U_m]$-module, generated by $x_i$ and $y_i$, such that $\partial x_i = 0$.
\end{lemma}
\begin{proof}
Let $\mathcal{P}$ denote the obvious disk bounded by a longitude for $U$ on the Heegaard diagram in Figure \ref{fig:altchara}. Orientation reversal corresponds to inverting the Alexander grading, up to an overall shift. In order to minimize the Alexander grading induced by $-U$, we maximize the grading induced by $U$. Lemma \ref{lemma:relperiodic} tells us that a generator $\bold{y}$ lies in $\mathcal{F}_{bot}^{-U} (\mathcal{G}_0)$ if $n_{\bold{y}} (\mathcal{P})$ is maximized. If $n_{\bold{y}} (\mathcal{P})$ is maximal, it is immediate that $(\bold{y})_i = (\bold{y})^i$ for each $i$, and that the component of $\bold{y}$ on $\alpha^s\cap\beta^s$ is fixed. For each $i>0$ there are two possible values of $(\bold{y})_i$, $x_i$ and $y_i$, let $x_i$ denote the intersection point contributing greater Maslov grading. The component of $\bold{y}$ on $\alpha_0$ is determined by the $Spin^C$ structure. There are $p$ possible values for $(\bold{y})_0$, corresponding to the different $Spin^C$ structures on $-L(p,q)$. 
\end{proof}

It follows that $H_{top}(\mathcal{F}_{bot}^{-U} (\mathcal{D}_0), \mathfrak{s}_{\xi})$ is generated by $[\bold{x}_0^D]$. Next, we perform an isotopy of $\alpha^s$ to get a diagram $\mathcal{D}_1$, followed by a free 0/3-index destabilization to obtain the diagram $\mathcal{G}_0$ of the previous section, and then relate $[\bold{x}_0^D]$ to $[\bold{x}_0^G]$.

Consider the triple diagram $(\Sigma,\boldsymbol{\beta},\boldsymbol{\alpha},\boldsymbol{\alpha}',\bold{w},\bold{z}\cup \bold{z}_{-U})$ shown in Figure \ref{fig:triple}. The set $\boldsymbol{\alpha}'$ is obtained by isotoping $\alpha^s$ to intersect only $\beta^s$, and $\alpha_i '$ is obtained by applying a small isotopy to $\alpha_i$. Let $\mathcal{D}_1 = (\Sigma,\boldsymbol{\beta},\boldsymbol{\alpha}',\bold{w},\bold{z}\cup \bold{z}_{-U})$, and
\[
D_{0,1}:HFK^{-,2}(\mathcal{D}_0) \xrightarrow {\simeq}HFK^{-,2}(\mathcal{D}_1)
\]
denote the isomorphism induced by the triple diagram. 
Let $\bold{x}_1^D$ denote the generator of $CFK^{-,2}(\mathcal{D}_1)$ whose components are pictured in Figure \ref{fig:triple}.

\begin{proposition}
\label{prop:name2}
$D_{0,1} ([\bold{x}_0^D]) = [\bold{x}_1^D]$.
\end{proposition}
\begin{proof}
The proof is essentially that of Proposition \ref{prop:relate}.
Let $\boldsymbol{\Theta}$ denote the top graded generator of $CFK^{-,2}(\Sigma,\boldsymbol{\alpha},\boldsymbol{\alpha}', \bold{w},\bold{z_K}\cup \bold{z}_{-B})$.
Let $u\in \pi_2 (\bold{x}_0^D, \boldsymbol{\Theta},\bold{y})$ by a Whitney triangle having corner at some generator $\bold{y}\in \mathbb{T}_{\boldsymbol{\beta}}\cap\mathbb{T}_{\boldsymbol{\alpha}'}$. As before, by studying the possible multiplicities of $u$, one can show that $u$ has domain equal to the union of small gray and black triangles pictured in Figure \ref{fig:triple}. It is immediate that for each $i>0$, $(\bold{y})_i = (\bold{x}^D_1)_i$. The $Spin^C$ structure fixes $(\bold{y})_0 = (\bold{x}^D_1)_0$.

\end{proof}

\begin{figure}[h]
\def\svgwidth{200pt}
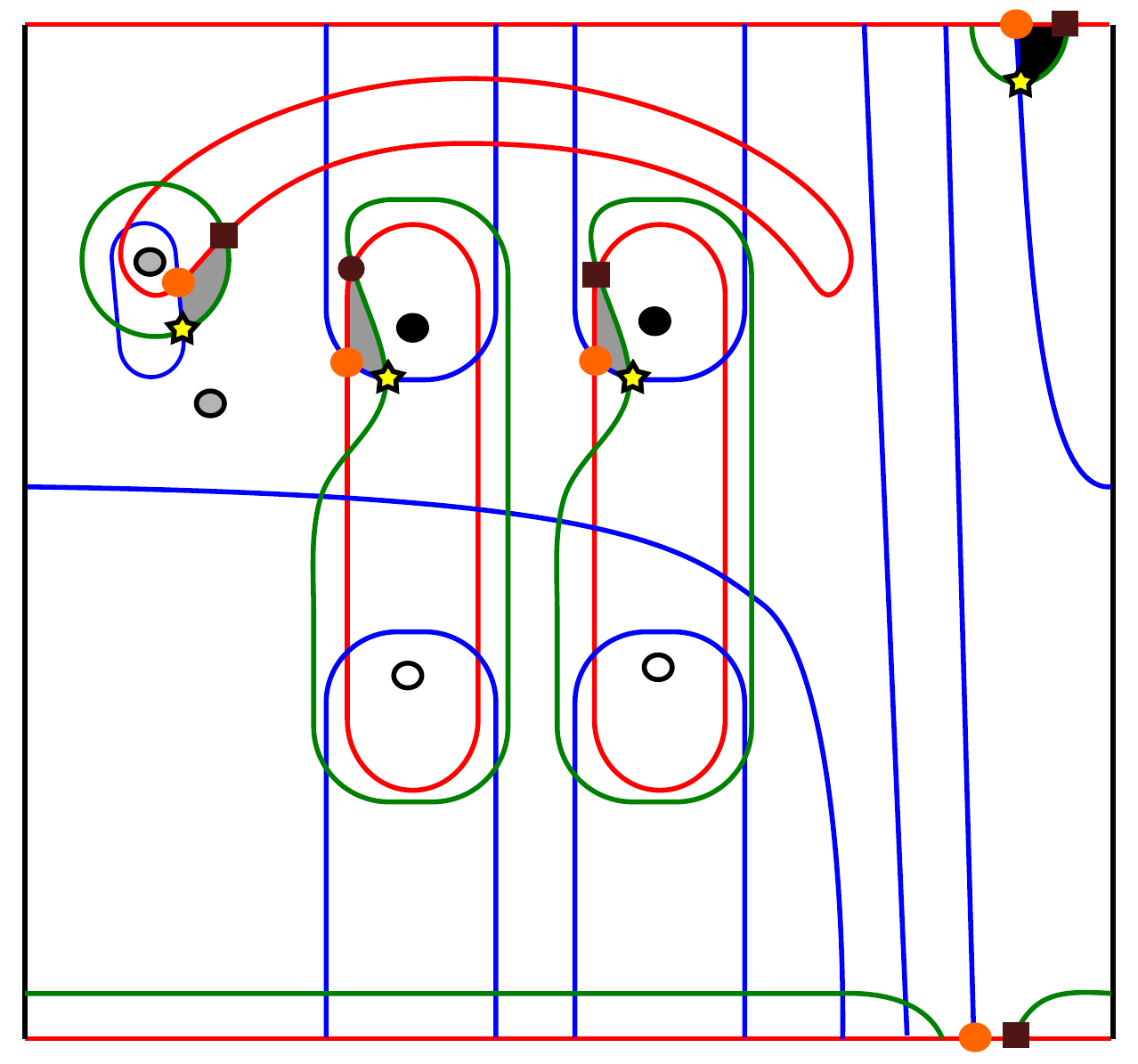
\caption{The triple diagram $(\Sigma,\boldsymbol{\beta},\boldsymbol{\alpha},\boldsymbol{\alpha}',\bold{w},\bold{z}\cup \bold{z}_{-U})$ encoding a 3-braid in $-L(3,1)$. The $\boldsymbol{\beta},\boldsymbol{\alpha}$ and $\boldsymbol{\alpha}'$ curves are blue, red, and green, respectively. The components of $\bold{x}_0^D, \bold{x}_1^D$, and $\boldsymbol{\Theta}$ are orange dots, stars, and brown squares, respectively.}
\label{fig:triple}
\end{figure}

The diagram $\mathcal{D}_1$ has a small configuration about one of the $\bold{z}_{-U}$ basepoints. Performing the index 0/3 free destabilization we see that $j (\bold{x}^D_1) = \bold{x}^G _0$ (where this generator is defined in the previous section). Proposition \ref{prop:relate} and Theorem \ref{thm:diagram} relate $[\bold{x}^G_0]$ to the BRAID invariant $t(K)$.

\section{A Reformulation of the GRID invariant $\theta$}
\label{sec:gridchara}
In this section we show that the GRID invariant $\theta$ can be reformulated in terms of the filtration on the knot Floer complex of a braid induced by the Seifert cable of the braid axis. This reformulation is the same as that of Section \ref{sec:altchara} for the BRAID invariant $t$, and we will use this to show that the two invariants are equivalent.

Let $U$ denote the Seifert cable of the binding $(B,\pi)$ of the standard rational open book for $(L(p,q),\xi_{UT})$.
Let $K\subset (L(p,q),\xi_{UT})$ be the transverse link encoded by a grid diagram $G$. Fixing a fundamental domain for the Heegaard torus, $G$ gives rise to a rectilinear braided projection of $K$ onto the fundamental domain missing the left and right boundaries of the fundamental domain, see Figure \ref{fig:rectproj}. This rectilinear projection may be altered and enhanced to one for $K\cup -U$. We may encode this projection with a grid diagram $ (T^2,\boldsymbol{\beta},\boldsymbol{\alpha},\bold{w}\cup\bold{w}_{-U},\bold{z}\cup\bold{z}_{-U})$ encoding $K\cup -U$ having index at most
\[
n+2k+2
\]
where $n$ is the index of $G$ and $k$ is the braid index of $K$. 
Consider the diagram $\mathcal{S}_0 =(T^2,\boldsymbol{\beta},\boldsymbol{\alpha},\bold{w},\bold{z}\cup\bold{z}_{-U})$ for $K$ with two free basepoints $\bold{z}_{-U} = \{z_0,z_1\}$. Let $\bold{x}^{S}_0\in \mathbb{T}_{\boldsymbol{\beta}}\cap\mathbb{T}_{\boldsymbol{\alpha}}$ denote the generator having components in the upper left corners of parallelograms containing points of $\bold{w}\cup\bold{z}_{-U}$.

\begin{figure}[h]
\def\svgwidth{200pt}
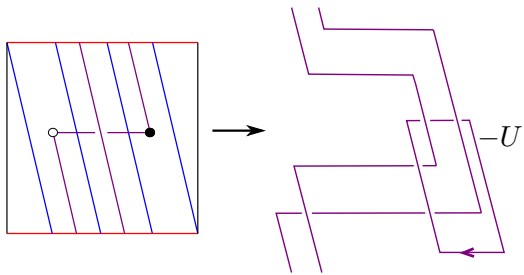
\caption{A rectilinear projection for $K=\sigma_1 ^{-1} \circ \delta ^{1/4}$ coming from an index one diagram is pictured on the left. On the right we have a rectilinear projection of $K\cup -U$.}
\label{fig:rectproj}
\end{figure}

\begin{lemma}
\label{lem:thetagen}
The class $[\bold{x}^{S}_0]$ generates $H_{top} (\mathcal{F}^{-U}_{bot}(\mathcal{S}_0), \mathfrak{s}_{\xi})$.
\end{lemma}
\begin{proof}
Proposition \ref{prop:spincagrees} tells us that $\mathfrak{s}_{\bold{w}}(\bold{x}_0^S) = \mathfrak{s}_{\xi}$.

Using Lemma \ref{lemma:relperiodic} it is easy to see that the generator $\bold{x}^S_0$ is in the bottom-most filtration level, as there is an obvious disk relative periodic $D$ domain for $U$, and the generator $\bold{x}^S_0$ has maximal multiplicity $n_{\bold{x}^S_0} (D)$.
The class is non-zero by Proposition \ref{prop:nontorsion}. 

The triangle counts in the following propositions relate the class $[\bold{x}^S_0]$ to the invariant $\theta(K)$, in particular, we will see that the Maslov grading of the generator $\bold{x}^S_0$ is $sl_\mathbb{Q} (L) + \frac{1}{p} - d(p,q,q-1) -2$.
It follows by the discussion following Lemma \ref{lem:needaname}, Proposition \ref{prop:needaname} and the results of the previous section that $top = sl_\mathbb{Q} (L) + \frac{1}{p} - d(p,q,q-1) -2$.  

By the discussion following Lemma \ref{lem:needaname}, $H_{top} (\mathcal{F}^{-U}_{bot}(\mathcal{S}_0), \mathfrak{s}_{\xi})$ is rank one.
\end{proof}

We wish to relate the class $[\bold{x}^{S}_0]$ to $\theta (K)$. We can perform two sequences of handleslides followed by two free index-0/3 destabilizations to go from $\mathcal{S}_0$ to a grid diagram encoding $K$.

We have labelled $\bold{z}_{-U} = \{z_0,z_1\}$. Suppose that $z_0$ lies in the $i_0^{th}$ column and $j_0^{th}$ row of $\mathcal{S}_0$, and $z_1$ lies in the $i_1^{th}$ column and $j_2^{th}$ row. Consider the triple diagram $(T^2,\boldsymbol{\beta},\boldsymbol{\alpha},\boldsymbol{\alpha}',\bold{w},\bold{z}\cup\bold{z}_{-U})$ pictured in Figure \ref{fig:finaltriple}. For $r\ne j_0+1$ or $j_1+1$ the curve $\boldsymbol{\alpha}' _r$ is a small perturbation of the curve $\boldsymbol{\alpha}_r$. For $r= j_0+1$ or $j_1+1$ the curve $\boldsymbol{\alpha}' _r$ is obtained by handlesliding $\boldsymbol{\alpha}_r$ over $\boldsymbol{\alpha}_{r-1}$.

Also consider the triple diagram $(T^2,\boldsymbol{\beta}',\boldsymbol{\beta},\boldsymbol{\alpha}',\bold{w},\bold{z}\cup\bold{z}_{-U})$. For $r\ne i_0$ or $i_1$, the curve $\boldsymbol{\beta}'_r$ is a small perturbation of the curve $\boldsymbol{\beta}_r$. For $r= i_0$ or $i_1$, the curve $\boldsymbol{\beta}'_r$ is obtained by handlesliding $\boldsymbol{\beta}_r$ over $\boldsymbol{\beta}_{r+1}$.

\begin{figure}[h]
\def\svgwidth{400pt}
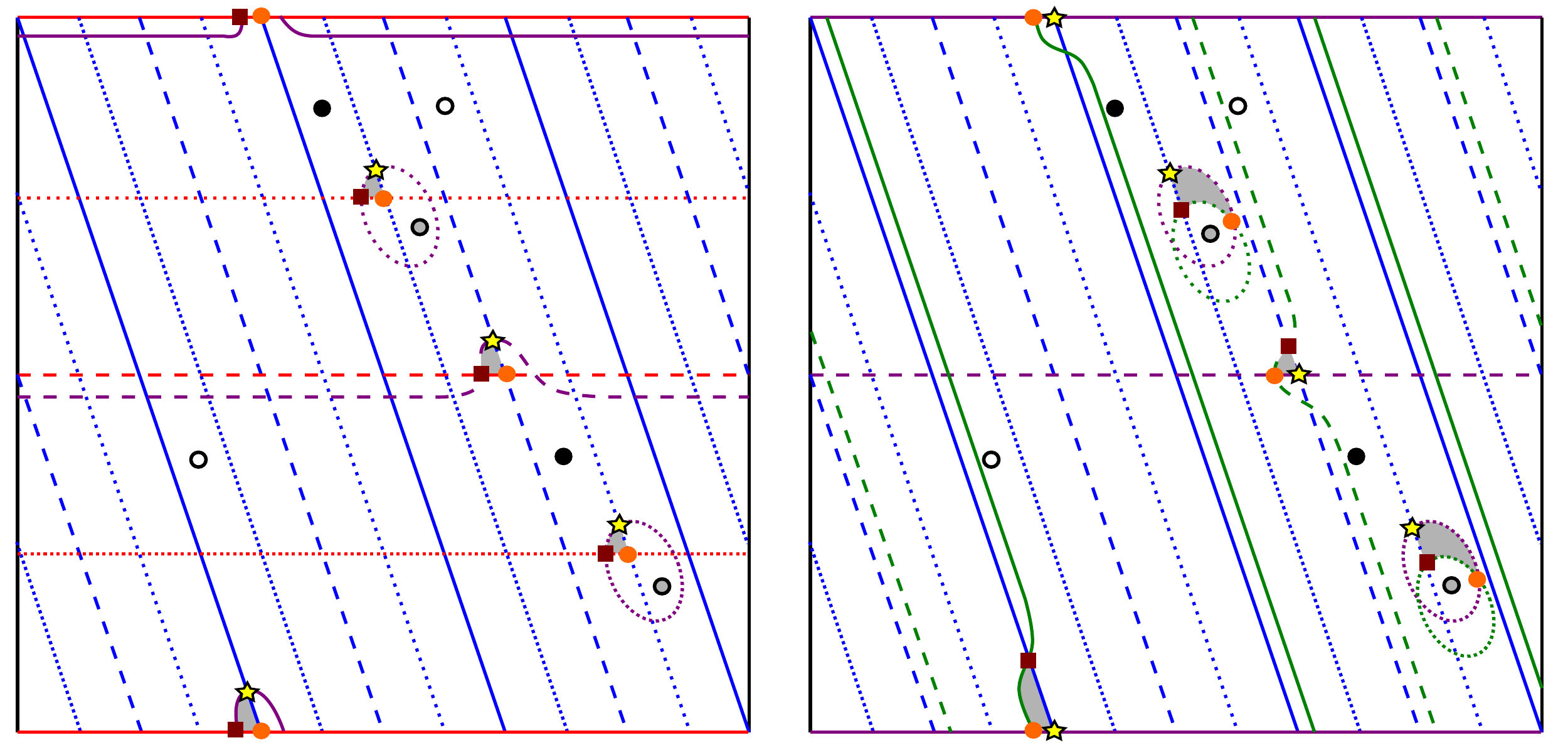
\caption{The case of a one braid $\tau_1\circ \delta^{1/3}$. The $\boldsymbol{\beta},\boldsymbol{\beta}',\boldsymbol{\alpha}$ and $\boldsymbol{\alpha}'$ curves are drawn blue, green, red and purple, respectively. The $\bold{w},\bold{z}$ and $\bold{z}_{-U}$ are solid, hollow and grey dots, respectively. The generators $\bold{x}_0^S$ and $\bold{x}_2^S$ are depicted with orange dots. The generator $\bold{x}_1$ is depicted with yellow stars. The brown squares depict $\boldsymbol{\theta}$.}
\label{fig:finaltriple}
\end{figure}

We let 
\begin{align*}
\mathcal{S}_1 =(T^2,\boldsymbol{\beta},\boldsymbol{\alpha}',\bold{w},\bold{z}\cup\bold{z}_{-U})\\
\mathcal{S}_2 =(T^2,\boldsymbol{\beta}',\boldsymbol{\alpha}',\bold{w},\bold{z}\cup\bold{z}_{-U}),
\end{align*}
and the generators $\bold{x}^S _1 \in CFK^{-,2} (\mathcal{S}_1)$ and $\bold{x}^S_2 \in CFK^{-,2} (\mathcal{S}_2)$ be as pictured in Figure \ref{fig:finaltriple}. 

\begin{proposition}
\label{prop:propS}
Let \begin{align*}
S_{0,1}:HFK^{-,2} (\mathcal{S}_0)\to HFK^{-,2} (\mathcal{S}_1)\\
S_{1,2}:HFK^{-,2} (\mathcal{S}_1)\to HFK^{-,2} (\mathcal{S}_2)
\end{align*}
denote the isomorphisms induced by the triple diagrams above. The composition $S_{0,2}=S_{1,2}\circ S_{0,1}$ sends the class $[\bold{x}^S _0]$ to $[\bold{x}^S_2]$.
\end{proposition}
\begin{proof}
The isomorphisms $S_{0,1}$ and $S_{0,2}$ are induced by pseudo-holomorphic triangle counts $s_{0,1}$ and $s_{1,2}$, respectively. 
We will show that $s_{0,1}(\bold{x}^S_0) = \bold{x}^S_1$, the proof that $s_{1,2}(\bold{x}^S_1) = \bold{x}^S_2$ is similar.

We argue that the triple diagram $(T^2,\boldsymbol{\beta},\boldsymbol{\alpha},\boldsymbol{\alpha}',\bold{w},\bold{z}\cup\bold{z}_{-U})$ is weakly admissible. Let $n$ denote the number of $\boldsymbol{\beta}$ curves.
Any doubly periodic domain of $(T^2,\boldsymbol{\alpha},\boldsymbol{\alpha}',\bold{w},\bold{z}\cup\bold{z}_{-U})$ missing all basepoints is a linear combination of periodic domains $\mathcal{P}_0,\dots,\mathcal{P}_{n-1}$, where 
\begin{align*}
\partial \mathcal{P}_r = \alpha _r \cup \alpha_r '         \quad\quad \text{ for } r \ne j_0 +1, j_1+1\\
\partial \mathcal{P}_r = \alpha_{r}\cup\alpha_{r-1}'\cup \alpha_r '         \quad\quad \text { for } r = j_0+1,j_1+1.
\end{align*}
Each of these has positive and negative coefficients, and it is easy to see that any linear combination also has this property. This establishes weak admissibility of $(T^2,\boldsymbol{\alpha},\boldsymbol{\alpha}',\bold{w},\bold{z}\cup\bold{z}_{-U})$. Any triply periodic domain $\mathcal{P}$ of $(T^2,\boldsymbol{\beta},\boldsymbol{\alpha},\boldsymbol{\alpha}',\bold{w},\bold{z}\cup\bold{z}_{-U})$ missing all basepoints will either be a doubly periodic domain of $(T^2,\boldsymbol{\alpha},\boldsymbol{\alpha}',\bold{w},\bold{z}\cup\bold{z}_{-U})$, in which case it has both positive and negative coefficients, or it will have some $\boldsymbol{\beta}$ curve in its boundary. This $\boldsymbol{\beta}$ curve must intersect a curve in $\boldsymbol{\alpha}\cup\boldsymbol{\alpha}'$, and near this intersection point multiplicities of both signs will appear.

Let $\boldsymbol{\theta}$ denote the top graded generator of $CFK^- (T^2,\boldsymbol{\alpha},\boldsymbol{\alpha}',\bold{w},\bold{z}\cup\bold{z}_{-U})$. We argue that the Whitney triangle $u_0 \in \pi_2 (\bold{x}_0^S,\boldsymbol{\theta},\bold{x}_1^S)$ whose domain $D(u_0)$ is shaded in Figure \ref{fig:finaltriple} is the unique triangle contributing to $s_{0,1}(\bold{x}^S_0)$. Let $u_0\ne u\in \pi_2(\bold{x}_0^S,\boldsymbol{\theta},\bold{y})$ be a Whitney triangle for some $\bold{y}\in\mathbb{T}_{\boldsymbol{\beta}}\cap\mathbb{T}_{\boldsymbol{\alpha}'}$. The domain $D(u)-D(u_0)$ has boundary consisting of arcs along the $\boldsymbol{\beta}$ and $\boldsymbol{\alpha}'$ curves and some total number of $\boldsymbol{\alpha}$ curves. It follows that for some doubly periodic domain $\mathcal{P}'$ of $(T^2,\boldsymbol{\alpha},\boldsymbol{\alpha}',\bold{w},\bold{z}\cup\bold{z}_{-U})$ the domain
\[
D=D(u)-D(u_0)-\mathcal{P}'
\]
has boundary consisting only of arcs along the $\boldsymbol{\beta}$ and $\boldsymbol{\alpha}'$ curves. $D\in\pi_2 (\bold{x}^S_1,\bold{y})$ is a Whitney disk. Any such disk can easily be seen to have some negative multiplicities. For any doubly periodic domain $\mathcal{P}$ of $(T^2,\boldsymbol{\alpha},\boldsymbol{\alpha}',\bold{w},\bold{z}\cup\bold{z}_{-U})$ the domain 
$D(u_0)+\mathcal{P}$ does not fully cover any region of $T^2\smallsetminus \{\boldsymbol{\beta}\cup\boldsymbol{\alpha}'\}$, in particular $D(u_0)+\mathcal{P}'$ does not cover the region in which $D$ has a negative multiplicity.  

It follows that $D(u)$ must have a negative multiplicity in the same region, and that $u$ can not admit a holomorphic representative. In summary, $u_0$ is the unique Whitney triangle having corners at $\bold{x}^S_0$ and $\boldsymbol{\theta}$ which admits a holomorphic representative. We conclude that $s_{0,1}(\bold{x}^S_0) = \bold{x}^S_1$. 
\end{proof}

We see that the diagram $\mathcal{S}_2$ has small configurations about both of the $\bold{z}_{-U}$ basepoints. Let $\mathcal{S}$ be the diagram obtained by performing both free index 0/3 destabilizations.

The compositions of projection and inclusion maps
\[
\begin{split}
j^2 : CFK^{-,2} (\mathcal{S}_2)\to CFK^- (\mathcal{S}) & \\  i^n : CFK^- (\mathcal{S})  \to  CFK^{-,2} (\mathcal{S}_2)
\end{split}
\] 
defined in subsection \ref{subsec:HFK}, send generators $\bold{x}^S_2$ to $\bold{x}^+$ and $\bold{x}^+$ to $\bold{x}_2 ^S$, respectively.

\subsection{GRID=BRAID}
\label{subsec:equiv}
In this subsection we prove Theorem \ref{thm:equivalence}. 

\begin{proof}
We inherit the notations of Propositions \ref{prop:relate}, \ref{prop:name2}, \ref{prop:propS} and Theorem \ref{thm:diagram}.

If we include the $\bold{w}_{-U}$ basepoints in the diagrams $\mathcal{D}_0$ and $\mathcal{S}_0$, both are diagrams for the link $K\cup -U \subset -L(p,q)$. It follows that $\mathcal{D}_0$ may be obtained from $\mathcal{S}_0$ by a sequence of isotopies and handleslides avoiding all basepoints, together with index 1/2 (de)stabilizations and linked index 0/3 (de)stabilizations not involving the basepoints $\bold{w}_{-U}\cup \bold{z}_{-U}$. Associated to this sequence of moves is a chain map which induces an isomorphism on homology
\[
F: HFK^{-,2}(\mathcal{S}_0)\to HFK^{-,2}(\mathcal{D}_0).
\]
Since the chain map respects the filtrations on both complexes induced by $-U$, it follows from the proof of Lemma \ref{lem:needaname} and Lemma \ref{lem:thetagen} that $F([\bold{x}^S_0])=[\bold{x}^D_0]$, since both of these generate $H_{top}(\mathcal{F}^{-U}_{bot}, \mathfrak{s}_{\xi})$.

By Propositions \ref{prop:relate}, \ref{prop:name2}, \ref{prop:propS} and Theorem \ref{thm:diagram}, the composition
\[
HFK^- (\mathcal{S})\xrightarrow[]{S_{0,2}^{-1}\circ(i^2)_*} HFK^{-,2}(S_0)\xrightarrow[]{F } HFK^{-,2}(D_0)\xrightarrow[]{(j)_* \circ D_{0,1} } HFK^{-,1} (\mathcal{G}_0)\xrightarrow[]{(j)_*\circ G_{0,2} } HFK^- (\mathcal{T}^G) 
\]
is a graded isomorphism of $\mathbb{F}[U_1,\dots,U_m]$-modules mapping the class $\theta(K) = [\bold{x}^S]$ to $[\bold{x}^G] = t(K)$.
\end{proof}

\section{Proof of Theorem \ref{thm:gridmirror}}
\label{sec:gridmirror}
%\cite{AG2004}
As mentioned in the introduction, the forward implication is immediate. Here we prove the reverse implication.
Suppose that $K\subset L(p,q)$ admits a surgery to the 3-sphere. Let $K'\subset L(p,q)$ denote the simple knot in the same homology class as $K$, i.e. $[K']=[K]\in H_1(L(p,q))$.
By the proof of Theorem 1.3 of \cite{realization}, there is an (Alexander, $Spin^C$)-graded isomorphism 
\[
\widehat{HFK}(-L(p,q),K)\simeq \widehat{HFK}(-L(p,q),K').
\]

Let $\mathcal{G}$ be a grid diagram for $K$, and let $\mathcal{G}_*$ denote the dual diagram as constructed in Subsection \ref{subsec:dual}. This pair of diagrams gives rise to a pair of Legendrians $L_0$ and $L_1$, where the $L_1$ is the topologically the mirror of $L_0$. 
Likewise, let $S_0$ and $S_1$ be the Legendrians induced by the index one diagram $\mathcal{H}$ for $K'$ and its dual.

We wish to show that if the quadruple of invariants $\widehat{\lambda}^{+}(L_0),\widehat{\lambda}^{-}(L_0),  \widehat{\lambda}^{+}(L_1),\widehat{\lambda}^{-}(L_1)$ do not vanish, then the diagram $\mathcal{G}$ has index one.
Clearly, $\widehat{\lambda}^{+}(S_0),\widehat{\lambda}^{-}(S_0)\ne 0$, because the complex used to define the pair of invariants has trivial differential.
Suppose that $\widehat{\lambda}^{+}(L_0),\widehat{\lambda}^{-}(L_0)\ne 0$.
By Proposition \ref{prop:spincagrees} and the fact that $K'$ is Floer-simple, 
the $Spin^C$-graded isomorphism of $\widehat{HFK}$ groups maps $\widehat{\lambda}^{+}(L_0)$ to $\widehat{\lambda}^{+}(S_0)$, as these generate
the summand \[\widehat{HFK}(-L(p,q),K,\mathfrak{s}_{\xi_{UT}})\simeq \widehat{HFK}(-L(p,q),K',\mathfrak{s}_{\xi_{UT}})\simeq \mathbb{F}.\] Likewise, $\widehat{\lambda}^{-}(L_0)$ is mapped to $\widehat{\lambda}^{-}(S_0)$, as both generate the summand with grading $\mathfrak{s}_{\overline{\xi_{UT}}}$.

 In particular, the Alexander gradings of the invariants must agree:
\begin{align*}
\frac{1}{2}\Big{(}tb_\mathbb{Q}(L_0) - rot_\mathbb{Q}(L_0) +1 \Big{)} = A(\widehat{\lambda}^{+}(L_0)) =  A(\widehat{\lambda}^{+}(S_0)) = \frac{1}{2}\Big{(}tb_\mathbb{Q}(S_0) - rot_\mathbb{Q}(S_0) +1 \Big{)}\\
\frac{1}{2}\Big{(}tb_\mathbb{Q}(L_0) + rot_\mathbb{Q}(L_0) +1 \Big{)} = A(\widehat{\lambda}^{-}(L_0)) =  A(\widehat{\lambda}^{-}(S_0)) = \frac{1}{2}\Big{(}tb_\mathbb{Q}(S_0) + rot_\mathbb{Q}(S_0) +1 \Big{)}.
\end{align*}
Adding the two equations, we see that $tb_\mathbb{Q}(L_0) = tb_\mathbb{Q}(S_0)$.

Assuming that $\widehat{\lambda}^{+}(L_1),\widehat{\lambda}^{-}(L_1)\ne 0$ and applying the above argument, one concludes that $tb_\mathbb{Q}(L_1) = tb_\mathbb{Q}(S_1)$. Let $g$ denote the index of $\mathcal{G}$. Applying Proposition \ref{prop:index} twice gives the desired result:
\[
-g = tb_{\mathbb{Q}}(L_0)+tb_{\mathbb{Q}}(L_1) = tb_{\mathbb{Q}}(S_0)+tb_{\mathbb{Q}}(S_1) = -1.
\]

\nocite{IKess,IKquasi,automatic,VVbind,StipV,LOSS, equiv, Pla, QOB, torsionob, comultcontact, comultgrid, sigmapositive, geomorder, lensgridleg, lensgridcomb, HKM, contactsurgery,CCSMCM, contactclass, contact1, transverseapprox, openbooks, legtransverse, torsionob, trivialbraid, coversmaslov, gridhom, MOST, dinvt, braiddynamics, absgrading, coversalex, turaev, relspinc, ras, disksgenus, holknots,genusdetection, linkgenus}

\newpage
%\bibliograph

%\style{alpha}
%\bibliography{References}
\thispagestyle{empty}
{\small
\markboth{References}{References}
\bibliographystyle{myalpha}
\bibliography{mybib}{}
}

\end{document}

%% file: template.tex
\newlength{\myhmargin}  
\usepackage[textheight=574pt, textwidth=445pt, marginparwidth=50pt, centering]{geometry}

\usepackage{amsmath,amssymb,amsthm,amsfonts,amscd,flafter,
graphicx,verbatim,pinlabel,mathrsfs,caption}
\usepackage[all]{xy}
\usepackage{epstopdf}
\usepackage{verbatim}
\usepackage[colorlinks=false]{hyperref}
\usepackage[all]{hypcap}
\usepackage{xcolor}
\usepackage{url}
\AtBeginDocument{\addtocontents{toc}{\protect\setlength{\parskip}{0pt}}}

\newcommand{\longcomment}[2]{#2}

\DeclareFontFamily{U}{mathx}{\hyphenchar\font45}
\DeclareFontShape{U}{mathx}{m}{n}{
      <5> <6> <7> <8> <9> <10>
      <10.95> <12> <14.4> <17.28> <20.74> <24.88>
      mathx10
      }{}
\DeclareSymbolFont{mathx}{U}{mathx}{m}{n}
\DeclareFontSubstitution{U}{mathx}{m}{n}
\DeclareMathAccent{\widecheck}{0}{mathx}{"71}
\newcommand{\HMto}{\widecheck{\mathit{HM}}}

\longcomment{
    \RequirePackage{rotating}                   % Case (2)
    \def\HMto{%
       \setbox0=\hbox{$\widehat{\mathit{HM}}$}
       \setbox1=\hbox{$\mathit{HM}$}
       \dimen0=1.1\ht0
       \advance\dimen0 by 1.17\ht1
       \smash{\mskip2mu\raise\dimen0\rlap{%
          \begin{turn}{180}
              {$\widehat{\phantom{\mathit{HM}}}$}
           \end{turn}} \mskip-2mu    
                \mathit{HM}
    }{\vphantom{\widehat{\mathit{HM}}}}{}}
}
    
    \newcommand*\oline[1]{%
  \vbox{%
    \hrule height 0.35pt%                  % Line above with certain width
    \kern0.1ex%                          % Distance between line and content
    \hbox{%
      \kern-0.0em%                        % Distance between content and left side of box, negative values for lines shorter than content
      \ifmmode#1\else\ensuremath{#1}\fi%  % The content, typeset in dependence of mode
      \kern-0.1em%                        % Distance between content and left side of box, negative values for lines shorter than content
    }% end of hbox
  }% end of vbox
}

\newtheorem{theorem}{Theorem}[section]
\newtheorem{lemma}[theorem]{Lemma}

\newtheorem{conjecture}[theorem]{Conjecture}
\newtheorem{corollary}[theorem]{Corollary}
\newtheorem{proposition}[theorem]{Proposition}

\theoremstyle{definition}
\newtheorem{definition}[theorem]{Definition}

\newtheorem{remark}[theorem]{Remark}

\makeatletter
\newtheorem*{rep@thm}{\rep@title}
\newcommand{\newreptheorem}[2]{%
\newenvironment{rep#1}[1][0,0]{%
\def\rep@title{#2##1}%
\begin{rep@thm}}%
{\end{rep@thm}}}
\makeatother
\newreptheorem{theorem}{}

%% file: WNWstab.pdf_tex
%% Creator: Inkscape inkscape 0.91, www.inkscape.org
%% Accompanies image file 'WNWstab.pdf' (pdf, eps, ps)
\begingroup%
  \makeatletter%
  \providecommand\color[2][]{%
    \errmessage{(Inkscape) Color is used for the text in Inkscape, but the package 'color.sty' is not loaded}%
    \renewcommand\color[2][]{}%
  }%
  \providecommand\transparent[1]{%
    \errmessage{(Inkscape) Transparency is used (non-zero) for the text in Inkscape, but the package 'transparent.sty' is not loaded}%
    \renewcommand\transparent[1]{}%
  }%
  \providecommand\rotatebox[2]{#2}%
  \ifx\svgwidth\undefined%
    \setlength{\unitlength}{720bp}%
    \ifx\svgscale\undefined%
      \relax%
    \else%
      \setlength{\unitlength}{\unitlength * \real{\svgscale}}%
    \fi%
  \else%
    \setlength{\unitlength}{\svgwidth}%
  \fi%
  \global\let\svgwidth\undefined%
  \global\let\svgscale\undefined%
  \makeatother%
  \begin{picture}(1,0.66666667)%
    \put(0,0){\includegraphics[width=\unitlength,page=1]{WNWstab.pdf}}%
    \put(0.15171973,0.33600563){\color[rgb]{0,0,0}\makebox(0,0)[lb]{\smash{$W$}}}%
    \put(0,0){\includegraphics[width=\unitlength,page=2]{WNWstab.pdf}}%
    \put(0.70484592,0.27616676){\color[rgb]{0,0,0}\makebox(0,0)[lb]{\smash{$W'$}}}%
    \put(0.81455058,0.38663857){\color[rgb]{0,0,0}\makebox(0,0)[lb]{\smash{$W$}}}%
    \put(0.81301622,0.27616676){\color[rgb]{0,0,0}\makebox(0,0)[lb]{\smash{$Z'$}}}%
    \put(0,0){\includegraphics[width=\unitlength,page=3]{WNWstab.pdf}}%
    \put(0.57153815,0.34599152){\color[rgb]{0,0,0}\makebox(0,0)[lb]{\smash{$\alpha '$}}}%
    \put(0.78481011,0.48484847){\color[rgb]{0,0,0}\makebox(0,0)[lb]{\smash{$\beta '$}}}%
  \end{picture}%
\endgroup%

%% file: commutation.pdf_tex
%% Creator: Inkscape inkscape 0.91, www.inkscape.org
%% Accompanies image file 'commutation.pdf' (pdf, eps, ps)
\begingroup%
  \makeatletter%
  \providecommand\color[2][]{%
    \errmessage{(Inkscape) Color is used for the text in Inkscape, but the package 'color.sty' is not loaded}%
    \renewcommand\color[2][]{}%
  }%
  \providecommand\transparent[1]{%
    \errmessage{(Inkscape) Transparency is used (non-zero) for the text in Inkscape, but the package 'transparent.sty' is not loaded}%
    \renewcommand\transparent[1]{}%
  }%
  \providecommand\rotatebox[2]{#2}%
  \ifx\svgwidth\undefined%
    \setlength{\unitlength}{560bp}%
    \ifx\svgscale\undefined%
      \relax%
    \else%
      \setlength{\unitlength}{\unitlength * \real{\svgscale}}%
    \fi%
  \else%
    \setlength{\unitlength}{\svgwidth}%
  \fi%
  \global\let\svgwidth\undefined%
  \global\let\svgscale\undefined%
  \makeatother%
  \begin{picture}(1,0.57142857)%
    \put(0,0){\includegraphics[width=\unitlength,page=1]{commutation.pdf}}%
    \put(0.04984891,0.4285714){\color[rgb]{0,0,0}\makebox(0,0)[lb]{\smash{$z$}}}%
    \put(0.11325414,0.29285711){\color[rgb]{0,0,0}\makebox(0,0)[lb]{\smash{$w$}}}%
    \put(0.26413463,0.14285711){\color[rgb]{0,0,0}\makebox(0,0)[lb]{\smash{$z$}}}%
    \put(0.3346827,0.14285711){\color[rgb]{0,0,0}\makebox(0,0)[lb]{\smash{$w$}}}%
    \put(0.27857143,0.29999997){\color[rgb]{0,0,0}\makebox(0,0)[lb]{\smash{}}}%
    \put(0.25975826,0.29285711){\color[rgb]{0,0,0}\makebox(0,0)[lb]{\smash{$z$}}}%
    \put(0.33629259,0.4285714){\color[rgb]{0,0,0}\makebox(0,0)[lb]{\smash{$w$}}}%
    \put(0.90611128,0.14285711){\color[rgb]{0,0,0}\makebox(0,0)[lb]{\smash{$w$}}}%
    \put(0.83118687,0.29285711){\color[rgb]{0,0,0}\makebox(0,0)[lb]{\smash{$z$}}}%
    \put(0.69270604,0.4285714){\color[rgb]{0,0,0}\makebox(0,0)[lb]{\smash{$z$}}}%
    \put(0.75611128,0.29285711){\color[rgb]{0,0,0}\makebox(0,0)[lb]{\smash{$w$}}}%
    \put(0.76413461,0.14285711){\color[rgb]{0,0,0}\makebox(0,0)[lb]{\smash{$z$}}}%
    \put(0.83629255,0.4285714){\color[rgb]{0,0,0}\makebox(0,0)[lb]{\smash{$w$}}}%
  \end{picture}%
\endgroup%

%% file: negstab.pdf_tex
%% Creator: Inkscape inkscape 0.91, www.inkscape.org
%% Accompanies image file 'negstab.pdf' (pdf, eps, ps)
\begingroup%
  \makeatletter%
  \providecommand\color[2][]{%
    \errmessage{(Inkscape) Color is used for the text in Inkscape, but the package 'color.sty' is not loaded}%
    \renewcommand\color[2][]{}%
  }%
  \providecommand\transparent[1]{%
    \errmessage{(Inkscape) Transparency is used (non-zero) for the text in Inkscape, but the package 'transparent.sty' is not loaded}%
    \renewcommand\transparent[1]{}%
  }%
  \providecommand\rotatebox[2]{#2}%
  \ifx\svgwidth\undefined%
    \setlength{\unitlength}{240bp}%
    \ifx\svgscale\undefined%
      \relax%
    \else%
      \setlength{\unitlength}{\unitlength * \real{\svgscale}}%
    \fi%
  \else%
    \setlength{\unitlength}{\svgwidth}%
  \fi%
  \global\let\svgwidth\undefined%
  \global\let\svgscale\undefined%
  \makeatother%
  \begin{picture}(1,0.66666667)%
    \put(0,0){\includegraphics[width=\unitlength,page=1]{negstab.pdf}}%
    \put(0.10356731,0.44879163){\color[rgb]{0,0,0}\makebox(0,0)[lb]{\smash{$w$}}}%
    \put(0,0){\includegraphics[width=\unitlength,page=2]{negstab.pdf}}%
  \end{picture}%
\endgroup%

%% file: gridcover.pdf_tex
%% Creator: Inkscape inkscape 0.91, www.inkscape.org
%% Accompanies image file 'gridcover.pdf' (pdf, eps, ps)
\begingroup%
  \makeatletter%
  \providecommand\color[2][]{%
    \errmessage{(Inkscape) Color is used for the text in Inkscape, but the package 'color.sty' is not loaded}%
    \renewcommand\color[2][]{}%
  }%
  \providecommand\transparent[1]{%
    \errmessage{(Inkscape) Transparency is used (non-zero) for the text in Inkscape, but the package 'transparent.sty' is not loaded}%
    \renewcommand\transparent[1]{}%
  }%
  \providecommand\rotatebox[2]{#2}%
  \ifx\svgwidth\undefined%
    \setlength{\unitlength}{456bp}%
    \ifx\svgscale\undefined%
      \relax%
    \else%
      \setlength{\unitlength}{\unitlength * \real{\svgscale}}%
    \fi%
  \else%
    \setlength{\unitlength}{\svgwidth}%
  \fi%
  \global\let\svgwidth\undefined%
  \global\let\svgscale\undefined%
  \makeatother%
  \begin{picture}(1,0.8245614)%
    \put(0,0){\includegraphics[width=\unitlength,page=1]{gridcover.pdf}}%
  \end{picture}%
\endgroup%

%% file: linked03.pdf_tex
%% Creator: Inkscape inkscape 0.91, www.inkscape.org
%% Accompanies image file 'linked03.pdf' (pdf, eps, ps)
\begingroup%
  \makeatletter%
  \providecommand\color[2][]{%
    \errmessage{(Inkscape) Color is used for the text in Inkscape, but the package 'color.sty' is not loaded}%
    \renewcommand\color[2][]{}%
  }%
  \providecommand\transparent[1]{%
    \errmessage{(Inkscape) Transparency is used (non-zero) for the text in Inkscape, but the package 'transparent.sty' is not loaded}%
    \renewcommand\transparent[1]{}%
  }%
  \providecommand\rotatebox[2]{#2}%
  \ifx\svgwidth\undefined%
    \setlength{\unitlength}{620bp}%
    \ifx\svgscale\undefined%
      \relax%
    \else%
      \setlength{\unitlength}{\unitlength * \real{\svgscale}}%
    \fi%
  \else%
    \setlength{\unitlength}{\svgwidth}%
  \fi%
  \global\let\svgwidth\undefined%
  \global\let\svgscale\undefined%
  \makeatother%
  \begin{picture}(1,0.38709677)%
    \put(0,0){\includegraphics[width=\unitlength,page=1]{linked03.pdf}}%
    \put(0.06348838,0.19645458){\color[rgb]{0,0,0}\makebox(0,0)[lb]{\smash{$w$}}}%
    \put(0.23239142,0.34379557){\color[rgb]{0,0,0}\makebox(0,0)[lb]{\smash{}}}%
    \put(0.22879773,0.19645458){\color[rgb]{0,0,0}\makebox(0,0)[lb]{\smash{$z$}}}%
    \put(0.52707331,0.19645458){\color[rgb]{0,0,0}\makebox(0,0)[lb]{\smash{$w$}}}%
    \put(0.70795528,0.19645458){\color[rgb]{0,0,0}\makebox(0,0)[lb]{\smash{$z'$}}}%
    \put(0.84810886,0.19645458){\color[rgb]{0,0,0}\makebox(0,0)[lb]{\smash{$w'$}}}%
    \put(0.94633616,0.19645458){\color[rgb]{0,0,0}\makebox(0,0)[lb]{\smash{$z$}}}%
    \put(0.5067091,0.26114082){\color[rgb]{0,0,0}\makebox(0,0)[lb]{\smash{$\beta '$}}}%
    \put(0.85649414,0.26233868){\color[rgb]{0,0,0}\makebox(0,0)[lb]{\smash{$\alpha '$}}}%
    \put(0.71873638,0.25275553){\color[rgb]{0,0,0}\makebox(0,0)[lb]{\smash{$y'$}}}%
    \put(0.71873638,0.12338297){\color[rgb]{0,0,0}\makebox(0,0)[lb]{\smash{$x'$}}}%
  \end{picture}%
\endgroup%

%% file: free03.pdf_tex
%% Creator: Inkscape inkscape 0.91, www.inkscape.org
%% Accompanies image file 'free03.pdf' (pdf, eps, ps)
\begingroup%
  \makeatletter%
  \providecommand\color[2][]{%
    \errmessage{(Inkscape) Color is used for the text in Inkscape, but the package 'color.sty' is not loaded}%
    \renewcommand\color[2][]{}%
  }%
  \providecommand\transparent[1]{%
    \errmessage{(Inkscape) Transparency is used (non-zero) for the text in Inkscape, but the package 'transparent.sty' is not loaded}%
    \renewcommand\transparent[1]{}%
  }%
  \providecommand\rotatebox[2]{#2}%
  \ifx\svgwidth\undefined%
    \setlength{\unitlength}{376bp}%
    \ifx\svgscale\undefined%
      \relax%
    \else%
      \setlength{\unitlength}{\unitlength * \real{\svgscale}}%
    \fi%
  \else%
    \setlength{\unitlength}{\svgwidth}%
  \fi%
  \global\let\svgwidth\undefined%
  \global\let\svgscale\undefined%
  \makeatother%
  \begin{picture}(1,0.42553191)%
    \put(0.02056656,0.20370686){\color[rgb]{0,0,0}\makebox(0,0)[lb]{\smash{$z$}}}%
    \put(0,0){\includegraphics[width=\unitlength,page=1]{free03.pdf}}%
    \put(0.63658395,0.20370686){\color[rgb]{0,0,0}\makebox(0,0)[lb]{\smash{$w'$}}}%
    \put(0.94312361,0.20370686){\color[rgb]{0,0,0}\makebox(0,0)[lb]{\smash{$z$}}}%
    \put(0.64050137,0.33787915){\color[rgb]{0,0,0}\makebox(0,0)[lb]{\smash{$y'$}}}%
    \put(0.63070783,0.04602981){\color[rgb]{0,0,0}\makebox(0,0)[lb]{\smash{$x'$}}}%
    \put(0.46127852,0.35648695){\color[rgb]{0,0,0}\makebox(0,0)[lb]{\smash{$\alpha '$}}}%
    \put(0.77663237,0.35648695){\color[rgb]{0,0,0}\makebox(0,0)[lb]{\smash{$\beta '$}}}%
  \end{picture}%
\endgroup%

%% file: COMMcomb.pdf_tex
%% Creator: Inkscape inkscape 0.91, www.inkscape.org
%% Accompanies image file 'COMMcomb.pdf' (pdf, eps, ps)
\begingroup%
  \makeatletter%
  \providecommand\color[2][]{%
    \errmessage{(Inkscape) Color is used for the text in Inkscape, but the package 'color.sty' is not loaded}%
    \renewcommand\color[2][]{}%
  }%
  \providecommand\transparent[1]{%
    \errmessage{(Inkscape) Transparency is used (non-zero) for the text in Inkscape, but the package 'transparent.sty' is not loaded}%
    \renewcommand\transparent[1]{}%
  }%
  \providecommand\rotatebox[2]{#2}%
  \ifx\svgwidth\undefined%
    \setlength{\unitlength}{241.27321777bp}%
    \ifx\svgscale\undefined%
      \relax%
    \else%
      \setlength{\unitlength}{\unitlength * \real{\svgscale}}%
    \fi%
  \else%
    \setlength{\unitlength}{\svgwidth}%
  \fi%
  \global\let\svgwidth\undefined%
  \global\let\svgscale\undefined%
  \makeatother%
  \begin{picture}(1,1.02549038)%
    \put(0,0){\includegraphics[width=\unitlength,page=1]{COMMcomb.pdf}}%
    \put(0.64891608,0.71028121){\color[rgb]{0,0,0}\makebox(0,0)[lb]{\smash{}}}%
  \end{picture}%
\endgroup%

%% file: legstabinv.pdf_tex
%% Creator: Inkscape inkscape 0.91, www.inkscape.org
%% Accompanies image file 'legstabinv.pdf' (pdf, eps, ps)
\begingroup%
  \makeatletter%
  \providecommand\color[2][]{%
    \errmessage{(Inkscape) Color is used for the text in Inkscape, but the package 'color.sty' is not loaded}%
    \renewcommand\color[2][]{}%
  }%
  \providecommand\transparent[1]{%
    \errmessage{(Inkscape) Transparency is used (non-zero) for the text in Inkscape, but the package 'transparent.sty' is not loaded}%
    \renewcommand\transparent[1]{}%
  }%
  \providecommand\rotatebox[2]{#2}%
  \ifx\svgwidth\undefined%
    \setlength{\unitlength}{720bp}%
    \ifx\svgscale\undefined%
      \relax%
    \else%
      \setlength{\unitlength}{\unitlength * \real{\svgscale}}%
    \fi%
  \else%
    \setlength{\unitlength}{\svgwidth}%
  \fi%
  \global\let\svgwidth\undefined%
  \global\let\svgscale\undefined%
  \makeatother%
  \begin{picture}(1,1)%
    \put(0,0){\includegraphics[width=\unitlength,page=1]{legstabinv.pdf}}%
    \put(0.81455058,0.83108301){\color[rgb]{0,0,0}\makebox(0,0)[lb]{\smash{$W_0$}}}%
    \put(0.57153815,0.79043597){\color[rgb]{0,0,0}\makebox(0,0)[lb]{\smash{$\alpha_0$}}}%
    \put(0.78481011,0.92929292){\color[rgb]{0,0,0}\makebox(0,0)[lb]{\smash{$\beta_0$}}}%
    \put(0.69717428,0.83108301){\color[rgb]{0,0,0}\makebox(0,0)[lb]{\smash{$Z_0$}}}%
    \put(0.69717428,0.71907685){\color[rgb]{0,0,0}\makebox(0,0)[lb]{\smash{$W_1$}}}%
    \put(0,0){\includegraphics[width=\unitlength,page=2]{legstabinv.pdf}}%
    \put(0.25899502,0.83108301){\color[rgb]{0,0,0}\makebox(0,0)[lb]{\smash{$W_0$}}}%
    \put(0.0159826,0.79043597){\color[rgb]{0,0,0}\makebox(0,0)[lb]{\smash{$\alpha_0$}}}%
    \put(0.22925456,0.92929292){\color[rgb]{0,0,0}\makebox(0,0)[lb]{\smash{$\beta_0$}}}%
    \put(0.14161872,0.83108301){\color[rgb]{0,0,0}\makebox(0,0)[lb]{\smash{$Z_0$}}}%
    \put(0.14161872,0.71907685){\color[rgb]{0,0,0}\makebox(0,0)[lb]{\smash{$W_1$}}}%
    \put(0,0){\includegraphics[width=\unitlength,page=3]{legstabinv.pdf}}%
    \put(0.25899502,0.27552745){\color[rgb]{0,0,0}\makebox(0,0)[lb]{\smash{$W_1$}}}%
    \put(0.0159826,0.23488041){\color[rgb]{0,0,0}\makebox(0,0)[lb]{\smash{$\alpha_0$}}}%
    \put(0.22925456,0.37373736){\color[rgb]{0,0,0}\makebox(0,0)[lb]{\smash{$\beta_0$}}}%
    \put(0.25899502,0.16505565){\color[rgb]{0,0,0}\makebox(0,0)[lb]{\smash{$Z_0$}}}%
    \put(0.14161872,0.1635213){\color[rgb]{0,0,0}\makebox(0,0)[lb]{\smash{$W_0$}}}%
    \put(0,0){\includegraphics[width=\unitlength,page=4]{legstabinv.pdf}}%
    \put(0.81455058,0.27552745){\color[rgb]{0,0,0}\makebox(0,0)[lb]{\smash{$W_1$}}}%
    \put(0.57153815,0.23488041){\color[rgb]{0,0,0}\makebox(0,0)[lb]{\smash{$\alpha_0$}}}%
    \put(0.78481011,0.37373736){\color[rgb]{0,0,0}\makebox(0,0)[lb]{\smash{$\beta_0$}}}%
    \put(0.81685208,0.16428851){\color[rgb]{0,0,0}\makebox(0,0)[lb]{\smash{$Z_0$}}}%
    \put(0.69717428,0.1635213){\color[rgb]{0,0,0}\makebox(0,0)[lb]{\smash{$W_0$}}}%
    \put(0,0){\includegraphics[width=\unitlength,page=5]{legstabinv.pdf}}%
  \end{picture}%
\endgroup%

%% file: one.pdf_tex
%% Creator: Inkscape inkscape 0.91, www.inkscape.org
%% Accompanies image file 'one.pdf' (pdf, eps, ps)
\begingroup%
  \makeatletter%
  \providecommand\color[2][]{%
    \errmessage{(Inkscape) Color is used for the text in Inkscape, but the package 'color.sty' is not loaded}%
    \renewcommand\color[2][]{}%
  }%
  \providecommand\transparent[1]{%
    \errmessage{(Inkscape) Transparency is used (non-zero) for the text in Inkscape, but the package 'transparent.sty' is not loaded}%
    \renewcommand\transparent[1]{}%
  }%
  \providecommand\rotatebox[2]{#2}%
  \ifx\svgwidth\undefined%
    \setlength{\unitlength}{720bp}%
    \ifx\svgscale\undefined%
      \relax%
    \else%
      \setlength{\unitlength}{\unitlength * \real{\svgscale}}%
    \fi%
  \else%
    \setlength{\unitlength}{\svgwidth}%
  \fi%
  \global\let\svgwidth\undefined%
  \global\let\svgscale\undefined%
  \makeatother%
  \begin{picture}(1,0.55555556)%
    \put(0,0){\includegraphics[width=\unitlength,page=1]{one.pdf}}%
    \put(0.42348009,0.09538786){\color[rgb]{0,0,0}\makebox(0,0)[lb]{\smash{$p_1$}}}%
    \put(0.67085951,0.09329138){\color[rgb]{0,0,0}\makebox(0,0)[lb]{\smash{$p_2$}}}%
    \put(0.93291402,0.09119495){\color[rgb]{0,0,0}\makebox(0,0)[lb]{\smash{$p_3$}}}%
    \put(0.07651992,0.42138362){\color[rgb]{0,0,0}\makebox(0,0)[lb]{\smash{$a_1$}}}%
    \put(0.05660378,0.28511531){\color[rgb]{0,0,0}\makebox(0,0)[lb]{\smash{$a_2$}}}%
    \put(0.27358492,0.42662471){\color[rgb]{0,0,0}\makebox(0,0)[lb]{\smash{$a_3$}}}%
    \put(0.25157232,0.28406706){\color[rgb]{0,0,0}\makebox(0,0)[lb]{\smash{$a_4$}}}%
    \put(0.47484277,0.32180293){\color[rgb]{0,0,0}\makebox(0,0)[lb]{\smash{$a_{2,5}$}}}%
    \put(0.74947591,0.31865826){\color[rgb]{0,0,0}\makebox(0,0)[lb]{\smash{$a_{2,6}$}}}%
    \put(0.66037733,0.19287213){\color[rgb]{0,0,0}\makebox(0,0)[lb]{\smash{$a_5$}}}%
    \put(0.92767293,0.1960168){\color[rgb]{0,0,0}\makebox(0,0)[lb]{\smash{$a_6$}}}%
  \end{picture}%
\endgroup%

%% file: two.pdf_tex
%% Creator: Inkscape inkscape 0.91, www.inkscape.org
%% Accompanies image file 'two.pdf' (pdf, eps, ps)
\begingroup%
  \makeatletter%
  \providecommand\color[2][]{%
    \errmessage{(Inkscape) Color is used for the text in Inkscape, but the package 'color.sty' is not loaded}%
    \renewcommand\color[2][]{}%
  }%
  \providecommand\transparent[1]{%
    \errmessage{(Inkscape) Transparency is used (non-zero) for the text in Inkscape, but the package 'transparent.sty' is not loaded}%
    \renewcommand\transparent[1]{}%
  }%
  \providecommand\rotatebox[2]{#2}%
  \ifx\svgwidth\undefined%
    \setlength{\unitlength}{720bp}%
    \ifx\svgscale\undefined%
      \relax%
    \else%
      \setlength{\unitlength}{\unitlength * \real{\svgscale}}%
    \fi%
  \else%
    \setlength{\unitlength}{\svgwidth}%
  \fi%
  \global\let\svgwidth\undefined%
  \global\let\svgscale\undefined%
  \makeatother%
  \begin{picture}(1,0.55555556)%
    \put(0,0){\includegraphics[width=\unitlength,page=1]{two.pdf}}%
    \put(0.45743893,0.49000737){\color[rgb]{0,0,0}\makebox(0,0)[lb]{\smash{$S_{1/2}$}}}%
    \put(0,0){\includegraphics[width=\unitlength,page=2]{two.pdf}}%
  \end{picture}%
\endgroup%

%% file: three.pdf_tex
%% Creator: Inkscape inkscape 0.91, www.inkscape.org
%% Accompanies image file 'three.pdf' (pdf, eps, ps)
\begingroup%
  \makeatletter%
  \providecommand\color[2][]{%
    \errmessage{(Inkscape) Color is used for the text in Inkscape, but the package 'color.sty' is not loaded}%
    \renewcommand\color[2][]{}%
  }%
  \providecommand\transparent[1]{%
    \errmessage{(Inkscape) Transparency is used (non-zero) for the text in Inkscape, but the package 'transparent.sty' is not loaded}%
    \renewcommand\transparent[1]{}%
  }%
  \providecommand\rotatebox[2]{#2}%
  \ifx\svgwidth\undefined%
    \setlength{\unitlength}{720bp}%
    \ifx\svgscale\undefined%
      \relax%
    \else%
      \setlength{\unitlength}{\unitlength * \real{\svgscale}}%
    \fi%
  \else%
    \setlength{\unitlength}{\svgwidth}%
  \fi%
  \global\let\svgwidth\undefined%
  \global\let\svgscale\undefined%
  \makeatother%
  \begin{picture}(1,0.55555556)%
    \put(0,0){\includegraphics[width=\unitlength,page=1]{three.pdf}}%
    \put(0.4988897,0.40784604){\color[rgb]{0,0,0}\makebox(0,0)[lb]{\smash{$S_{1/2}$}}}%
    \put(0,0){\includegraphics[width=\unitlength,page=2]{three.pdf}}%
  \end{picture}%
\endgroup%

%% file: nine.pdf_tex
%% Creator: Inkscape inkscape 0.91, www.inkscape.org
%% Accompanies image file 'nine.pdf' (pdf, eps, ps)
\begingroup%
  \makeatletter%
  \providecommand\color[2][]{%
    \errmessage{(Inkscape) Color is used for the text in Inkscape, but the package 'color.sty' is not loaded}%
    \renewcommand\color[2][]{}%
  }%
  \providecommand\transparent[1]{%
    \errmessage{(Inkscape) Transparency is used (non-zero) for the text in Inkscape, but the package 'transparent.sty' is not loaded}%
    \renewcommand\transparent[1]{}%
  }%
  \providecommand\rotatebox[2]{#2}%
  \ifx\svgwidth\undefined%
    \setlength{\unitlength}{720bp}%
    \ifx\svgscale\undefined%
      \relax%
    \else%
      \setlength{\unitlength}{\unitlength * \real{\svgscale}}%
    \fi%
  \else%
    \setlength{\unitlength}{\svgwidth}%
  \fi%
  \global\let\svgwidth\undefined%
  \global\let\svgscale\undefined%
  \makeatother%
  \begin{picture}(1,0.55555556)%
    \put(0,0){\includegraphics[width=\unitlength,page=1]{nine.pdf}}%
    \put(0.68393785,0.44411551){\color[rgb]{0,0,0}\makebox(0,0)[lb]{\smash{$S_{1/2}$}}}%
    \put(0,0){\includegraphics[width=\unitlength,page=2]{nine.pdf}}%
    \put(0.06735751,0.24130274){\color[rgb]{0,0,0}\makebox(0,0)[lb]{\smash{$w_K$}}}%
    \put(0.17912657,0.24574393){\color[rgb]{0,0,0}\makebox(0,0)[lb]{\smash{$w_K$}}}%
    \put(0.04515174,0.05255375){\color[rgb]{0,0,0}\makebox(0,0)[lb]{\smash{$\alpha_1$}}}%
    \put(0.15544042,0.05551443){\color[rgb]{0,0,0}\makebox(0,0)[lb]{\smash{$\alpha_2$}}}%
  \end{picture}%
\endgroup%

%% file: triangleone.pdf_tex
%% Creator: Inkscape inkscape 0.91, www.inkscape.org
%% Accompanies image file 'triangleone.pdf' (pdf, eps, ps)
\begingroup%
  \makeatletter%
  \providecommand\color[2][]{%
    \errmessage{(Inkscape) Color is used for the text in Inkscape, but the package 'color.sty' is not loaded}%
    \renewcommand\color[2][]{}%
  }%
  \providecommand\transparent[1]{%
    \errmessage{(Inkscape) Transparency is used (non-zero) for the text in Inkscape, but the package 'transparent.sty' is not loaded}%
    \renewcommand\transparent[1]{}%
  }%
  \providecommand\rotatebox[2]{#2}%
  \ifx\svgwidth\undefined%
    \setlength{\unitlength}{720bp}%
    \ifx\svgscale\undefined%
      \relax%
    \else%
      \setlength{\unitlength}{\unitlength * \real{\svgscale}}%
    \fi%
  \else%
    \setlength{\unitlength}{\svgwidth}%
  \fi%
  \global\let\svgwidth\undefined%
  \global\let\svgscale\undefined%
  \makeatother%
  \begin{picture}(1,0.55555556)%
    \put(0,0){\includegraphics[width=\unitlength,page=1]{triangleone.pdf}}%
    \put(0.05329386,0.24352334){\color[rgb]{0,0,0}\makebox(0,0)[lb]{\smash{}}}%
    \put(0,0){\includegraphics[width=\unitlength,page=2]{triangleone.pdf}}%
    \put(0.08655031,0.45347532){\color[rgb]{0,0,0}\makebox(0,0)[lb]{\smash{$R$}}}%
    
    \put(0.87860847,0.44485566){\color[rgb]{0,0,0}\makebox(0,0)[lb]{\smash{$b$}}}%
    \put(0.88008877,0.53367877){\color[rgb]{0,0,0}\makebox(0,0)[lb]{\smash{$0$}}}%
    \put(0.87934862,0.39378233){\color[rgb]{0,0,0}\makebox(0,0)[lb]{\smash{$0$}}}%
    \put(0.81273132,0.47816431){\color[rgb]{0,0,0}\makebox(0,0)[lb]{\smash{$a$}}}%
    \put(0.94226501,0.47594379){\color[rgb]{0,0,0}\makebox(0,0)[lb]{\smash{$c$}}}%
    \put(0.81347148,0.39748336){\color[rgb]{0,0,0}\makebox(0,0)[lb]{\smash{}}}%
    \put(0.80680976,0.40118433){\color[rgb]{0,0,0}\makebox(0,0)[lb]{\smash{$d$}}}%
    \put(0.96299038,0.39822358){\color[rgb]{0,0,0}\makebox(0,0)[lb]{\smash{$e$}}}%
  \end{picture}%
\endgroup%

%% file: triangletwo.pdf_tex
%% Creator: Inkscape inkscape 0.91, www.inkscape.org
%% Accompanies image file 'triangletwo.pdf' (pdf, eps, ps)
\begingroup%
  \makeatletter%
  \providecommand\color[2][]{%
    \errmessage{(Inkscape) Color is used for the text in Inkscape, but the package 'color.sty' is not loaded}%
    \renewcommand\color[2][]{}%
  }%
  \providecommand\transparent[1]{%
    \errmessage{(Inkscape) Transparency is used (non-zero) for the text in Inkscape, but the package 'transparent.sty' is not loaded}%
    \renewcommand\transparent[1]{}%
  }%
  \providecommand\rotatebox[2]{#2}%
  \ifx\svgwidth\undefined%
    \setlength{\unitlength}{720bp}%
    \ifx\svgscale\undefined%
      \relax%
    \else%
      \setlength{\unitlength}{\unitlength * \real{\svgscale}}%
    \fi%
  \else%
    \setlength{\unitlength}{\svgwidth}%
  \fi%
  \global\let\svgwidth\undefined%
  \global\let\svgscale\undefined%
  \makeatother%
  \begin{picture}(1,0.55555556)%
    \put(0,0){\includegraphics[width=\unitlength,page=1]{triangletwo.pdf}}%
    \put(0.05329386,0.24352334){\color[rgb]{0,0,0}\makebox(0,0)[lb]{\smash{}}}%
    \put(0,0){\includegraphics[width=\unitlength,page=2]{triangletwo.pdf}}%
    \put(0.03997039,0.23242043){\color[rgb]{0,0,0}\makebox(0,0)[lb]{\smash{$R$}}}%
    
    \put(0.88230943,0.45299775){\color[rgb]{0,0,0}\makebox(0,0)[lb]{\smash{$1$}}}%
    \put(0.88304959,0.39230195){\color[rgb]{0,0,0}\makebox(0,0)[lb]{\smash{$0$}}}%
    \put(0.83123617,0.47076238){\color[rgb]{0,0,0}\makebox(0,0)[lb]{\smash{$0$}}}%
    \put(0.93486308,0.47372312){\color[rgb]{0,0,0}\makebox(0,0)[lb]{\smash{$0$}}}%
    \put(0.80310879,0.40340493){\color[rgb]{0,0,0}\makebox(0,0)[lb]{\smash{$0$}}}%
    \put(0.96595113,0.40118433){\color[rgb]{0,0,0}\makebox(0,0)[lb]{\smash{$0$}}}%
    \put(0.88082899,0.53886011){\color[rgb]{0,0,0}\makebox(0,0)[lb]{\smash{$0$}}}%
  \end{picture}%
\endgroup%

%% file: diagramT.pdf_tex
%% Creator: Inkscape inkscape 0.91, www.inkscape.org
%% Accompanies image file 'diagramT.pdf' (pdf, eps, ps)
\begingroup%
  \makeatletter%
  \providecommand\color[2][]{%
    \errmessage{(Inkscape) Color is used for the text in Inkscape, but the package 'color.sty' is not loaded}%
    \renewcommand\color[2][]{}%
  }%
  \providecommand\transparent[1]{%
    \errmessage{(Inkscape) Transparency is used (non-zero) for the text in Inkscape, but the package 'transparent.sty' is not loaded}%
    \renewcommand\transparent[1]{}%
  }%
  \providecommand\rotatebox[2]{#2}%
  \ifx\svgwidth\undefined%
    \setlength{\unitlength}{720bp}%
    \ifx\svgscale\undefined%
      \relax%
    \else%
      \setlength{\unitlength}{\unitlength * \real{\svgscale}}%
    \fi%
  \else%
    \setlength{\unitlength}{\svgwidth}%
  \fi%
  \global\let\svgwidth\undefined%
  \global\let\svgscale\undefined%
  \makeatother%
  \begin{picture}(1,0.61111111)%
    \put(0,0){\includegraphics[width=\unitlength,page=1]{diagramT.pdf}}%
    \put(0.08562205,0.37768841){\color[rgb]{0,0,0}\makebox(0,0)[lb]{\smash{$w_K$}}}%
    \put(0.23327979,0.37483434){\color[rgb]{0,0,0}\makebox(0,0)[lb]{\smash{$w_K$}}}%
    \put(0,0){\includegraphics[width=\unitlength,page=2]{diagramT.pdf}}%
  \end{picture}%
\endgroup%

%% file: basisanddiagram.pdf_tex
%% Creator: Inkscape inkscape 0.91, www.inkscape.org
%% Accompanies image file 'basisanddiagram.pdf' (pdf, eps, ps)
\begingroup%
  \makeatletter%
  \providecommand\color[2][]{%
    \errmessage{(Inkscape) Color is used for the text in Inkscape, but the package 'color.sty' is not loaded}%
    \renewcommand\color[2][]{}%
  }%
  \providecommand\transparent[1]{%
    \errmessage{(Inkscape) Transparency is used (non-zero) for the text in Inkscape, but the package 'transparent.sty' is not loaded}%
    \renewcommand\transparent[1]{}%
  }%
  \providecommand\rotatebox[2]{#2}%
  \ifx\svgwidth\undefined%
    \setlength{\unitlength}{720bp}%
    \ifx\svgscale\undefined%
      \relax%
    \else%
      \setlength{\unitlength}{\unitlength * \real{\svgscale}}%
    \fi%
  \else%
    \setlength{\unitlength}{\svgwidth}%
  \fi%
  \global\let\svgwidth\undefined%
  \global\let\svgscale\undefined%
  \makeatother%
  \begin{picture}(1,0.55555556)%
    \put(0,0){\includegraphics[width=\unitlength,page=1]{basisanddiagram.pdf}}%
    \put(0.02758932,0.30062856){\color[rgb]{0,0,0}\makebox(0,0)[lb]{\smash{$p_1$}}}%
    \put(0.09894104,0.29967723){\color[rgb]{0,0,0}\makebox(0,0)[lb]{\smash{$p_2$}}}%
    \put(0,0){\includegraphics[width=\unitlength,page=2]{basisanddiagram.pdf}}%
    \put(0.50802423,0.07610858){\color[rgb]{0,0,0}\makebox(0,0)[lb]{\smash{$-S_0$}}}%
    \put(0,0){\includegraphics[width=\unitlength,page=3]{basisanddiagram.pdf}}%
  \end{picture}%
\endgroup%

%% file: 2diagrams.pdf_tex
%% Creator: Inkscape inkscape 0.91, www.inkscape.org
%% Accompanies image file '2diagrams.pdf' (pdf, eps, ps)
\begingroup%
  \makeatletter%
  \providecommand\color[2][]{%
    \errmessage{(Inkscape) Color is used for the text in Inkscape, but the package 'color.sty' is not loaded}%
    \renewcommand\color[2][]{}%
  }%
  \providecommand\transparent[1]{%
    \errmessage{(Inkscape) Transparency is used (non-zero) for the text in Inkscape, but the package 'transparent.sty' is not loaded}%
    \renewcommand\transparent[1]{}%
  }%
  \providecommand\rotatebox[2]{#2}%
  \ifx\svgwidth\undefined%
    \setlength{\unitlength}{720bp}%
    \ifx\svgscale\undefined%
      \relax%
    \else%
      \setlength{\unitlength}{\unitlength * \real{\svgscale}}%
    \fi%
  \else%
    \setlength{\unitlength}{\svgwidth}%
  \fi%
  \global\let\svgwidth\undefined%
  \global\let\svgscale\undefined%
  \makeatother%
  \begin{picture}(1,0.55555556)%
    \put(0,0){\includegraphics[width=\unitlength,page=1]{2diagrams.pdf}}%
    \put(0.44559587,0.25388602){\color[rgb]{0,0,0}\makebox(0,0)[lb]{\smash{$w_U$}}}%
    \put(0.29237603,0.24870467){\color[rgb]{0,0,0}\makebox(0,0)[lb]{\smash{$z_U$}}}%
    \put(0.35677277,0.42635082){\color[rgb]{0,0,0}\makebox(0,0)[lb]{\smash{$S_{1/2}$}}}%
    \put(0.37231683,0.12953367){\color[rgb]{0,0,0}\makebox(0,0)[lb]{\smash{$-S_{0}$}}}%
    \put(0,0){\includegraphics[width=\unitlength,page=2]{2diagrams.pdf}}%
    \put(0.71132494,0.26720945){\color[rgb]{0,0,0}\makebox(0,0)[lb]{\smash{$w_U$}}}%
    \put(0.65803107,0.26572907){\color[rgb]{0,0,0}\makebox(0,0)[lb]{\smash{$z_U$}}}%
    \put(0.72982978,0.33826793){\color[rgb]{0,0,0}\makebox(0,0)[lb]{\smash{$\mu$}}}%
    \put(0.64765523,0.5129534){\color[rgb]{0,0,0}\makebox(0,0)[lb]{\smash{$\beta _1$}}}%
    \put(0.56550703,0.43227238){\color[rgb]{0,0,0}\makebox(0,0)[lb]{\smash{$\alpha_1$}}}%
    \put(0.61591471,0.39594921){\color[rgb]{0,0,0}\makebox(0,0)[lb]{\smash{$\alpha_2$}}}%
    \put(0.64704888,0.34554149){\color[rgb]{0,0,0}\makebox(0,0)[lb]{\smash{$\alpha_3$}}}%
    \put(0.7625552,0.50479924){\color[rgb]{0,0,0}\makebox(0,0)[lb]{\smash{$\beta _2$}}}%
    \put(0.8767138,0.45365022){\color[rgb]{0,0,0}\makebox(0,0)[lb]{\smash{$\beta _3$}}}%
    \put(0.90437364,0.25574501){\color[rgb]{0,0,0}\makebox(0,0)[lb]{\smash{$\alpha_0$}}}%
    \put(0,0){\includegraphics[width=\unitlength,page=3]{2diagrams.pdf}}%
    \put(0.58117126,0.26538178){\color[rgb]{0,0,0}\makebox(0,0)[lb]{\smash{$\lambda_U$}}}%
  \end{picture}%
\endgroup%

%% file: planardiagrams.pdf_tex
%% Creator: Inkscape inkscape 0.91, www.inkscape.org
%% Accompanies image file 'planardiagrams.pdf' (pdf, eps, ps)
\begingroup%
  \makeatletter%
  \providecommand\color[2][]{%
    \errmessage{(Inkscape) Color is used for the text in Inkscape, but the package 'color.sty' is not loaded}%
    \renewcommand\color[2][]{}%
  }%
  \providecommand\transparent[1]{%
    \errmessage{(Inkscape) Transparency is used (non-zero) for the text in Inkscape, but the package 'transparent.sty' is not loaded}%
    \renewcommand\transparent[1]{}%
  }%
  \providecommand\rotatebox[2]{#2}%
  \ifx\svgwidth\undefined%
    \setlength{\unitlength}{720bp}%
    \ifx\svgscale\undefined%
      \relax%
    \else%
      \setlength{\unitlength}{\unitlength * \real{\svgscale}}%
    \fi%
  \else%
    \setlength{\unitlength}{\svgwidth}%
  \fi%
  \global\let\svgwidth\undefined%
  \global\let\svgscale\undefined%
  \makeatother%
  \begin{picture}(1,0.47777778)%
    \put(0,0){\includegraphics[width=\unitlength,page=1]{planardiagrams.pdf}}%
    \put(0.0296077,0.22871946){\color[rgb]{0,0,0}\makebox(0,0)[lb]{\smash{$w_U$}}}%
    \put(0.42487047,0.23094005){\color[rgb]{0,0,0}\makebox(0,0)[lb]{\smash{$z_U$}}}%
    \put(0.0199852,0.36861583){\color[rgb]{0,0,0}\makebox(0,0)[lb]{\smash{$\mu$}}}%
    \put(0,0){\includegraphics[width=\unitlength,page=2]{planardiagrams.pdf}}%
    \put(0.94398044,0.3212929){\color[rgb]{0,0,0}\makebox(0,0)[lb]{\smash{$w_B$}}}%
    \put(0,0){\includegraphics[width=\unitlength,page=3]{planardiagrams.pdf}}%
    \put(0.92225891,0.24724844){\color[rgb]{0,0,0}\makebox(0,0)[lb]{\smash{$z_B$}}}%
    \put(0.54623739,0.42709104){\color[rgb]{0,0,0}\makebox(0,0)[lb]{\smash{$\beta_1$}}}%
    \put(0.55734022,0.28497407){\color[rgb]{0,0,0}\makebox(0,0)[lb]{\smash{$\alpha_1$}}}%
    \put(0.92235121,0.10042146){\color[rgb]{0,0,0}\makebox(0,0)[lb]{\smash{$\beta_0$}}}%
    \put(0,0){\includegraphics[width=\unitlength,page=4]{planardiagrams.pdf}}%
    \put(0.7026083,0.28497407){\color[rgb]{0,0,0}\makebox(0,0)[lb]{\smash{$\alpha_2$}}}%
    \put(0.81602566,0.28497407){\color[rgb]{0,0,0}\makebox(0,0)[lb]{\smash{$\alpha_3$}}}%
    \put(0.66039605,0.42709104){\color[rgb]{0,0,0}\makebox(0,0)[lb]{\smash{$\beta_2$}}}%
    \put(0.7723308,0.42709104){\color[rgb]{0,0,0}\makebox(0,0)[lb]{\smash{$\beta_3$}}}%
    \put(0.66792684,0.02297998){\color[rgb]{0,0,0}\makebox(0,0)[lb]{\smash{$\alpha_0$}}}%
    \put(0,0){\includegraphics[width=\unitlength,page=5]{planardiagrams.pdf}}%
    \put(0.8191253,0.24833209){\color[rgb]{0,0,0}\makebox(0,0)[lb]{\smash{$\lambda$}}}%
    \put(0,0){\includegraphics[width=\unitlength,page=6]{planardiagrams.pdf}}%
  \end{picture}%
\endgroup%

%% file: multiplicity.pdf_tex
%% Creator: Inkscape inkscape 0.91, www.inkscape.org
%% Accompanies image file 'multiplicity.pdf' (pdf, eps, ps)
\begingroup%
  \makeatletter%
  \providecommand\color[2][]{%
    \errmessage{(Inkscape) Color is used for the text in Inkscape, but the package 'color.sty' is not loaded}%
    \renewcommand\color[2][]{}%
  }%
  \providecommand\transparent[1]{%
    \errmessage{(Inkscape) Transparency is used (non-zero) for the text in Inkscape, but the package 'transparent.sty' is not loaded}%
    \renewcommand\transparent[1]{}%
  }%
  \providecommand\rotatebox[2]{#2}%
  \ifx\svgwidth\undefined%
    \setlength{\unitlength}{416bp}%
    \ifx\svgscale\undefined%
      \relax%
    \else%
      \setlength{\unitlength}{\unitlength * \real{\svgscale}}%
    \fi%
  \else%
    \setlength{\unitlength}{\svgwidth}%
  \fi%
  \global\let\svgwidth\undefined%
  \global\let\svgscale\undefined%
  \makeatother%
  \begin{picture}(1,1)%
    \put(0,0){\includegraphics[width=\unitlength,page=1]{multiplicity.pdf}}%
    \put(0.37841905,0.77122247){\color[rgb]{0,0,0}\makebox(0,0)[lb]{\smash{}}}%
 
    \put(0.77454194,0.29432592){\color[rgb]{0,0,0}\makebox(0,0)[lb]{\smash{$\widetilde{\lambda}$}}}%
    \put(0.22572363,0.87523236){\color[rgb]{0,0,0}\makebox(0,0)[lb]{\smash{$0$}}}%
    \put(0.10290342,0.17371869){\color[rgb]{0,0,0}\makebox(0,0)[lb]{\smash{}}}%
    \put(0.36514118,0.87523236){\color[rgb]{0,0,0}\makebox(0,0)[lb]{\smash{$-1$}}}%
    \put(0.47689652,0.87523236){\color[rgb]{0,0,0}\makebox(0,0)[lb]{\smash{$-2$}}}%
    \put(0.59197129,0.87523236){\color[rgb]{0,0,0}\makebox(0,0)[lb]{\smash{$-3$}}}%
    \put(0.7059396,0.87523236){\color[rgb]{0,0,0}\makebox(0,0)[lb]{\smash{$-4$}}}%
    \put(0.81548198,0.87523236){\color[rgb]{0,0,0}\makebox(0,0)[lb]{\smash{$-5$}}}%
    \put(0.91064001,0.87523236){\color[rgb]{0,0,0}\makebox(0,0)[lb]{\smash{$0$}}}%
    \put(0.8221209,0.64618925){\color[rgb]{0,0,0}\makebox(0,0)[lb]{\smash{$1$}}}%
    \put(0.72143044,0.55877669){\color[rgb]{0,0,0}\makebox(0,0)[lb]{\smash{$2$}}}%
    \put(0.87744534,0.18810303){\color[rgb]{0,0,0}\makebox(0,0)[lb]{\smash{$3$}}}%
    \put(0.13609808,0.55877669){\color[rgb]{0,0,0}\makebox(0,0)[lb]{\smash{$1$}}}%
    \put(0.12835266,0.37288661){\color[rgb]{0,0,0}\makebox(0,0)[lb]{\smash{$2$}}}%
    \put(0.11064884,0.16154734){\color[rgb]{0,0,0}\makebox(0,0)[lb]{\smash{$3$}}}%
  \end{picture}%
\endgroup%

%% file: somecounts.pdf_tex
%% Creator: Inkscape inkscape 0.91, www.inkscape.org
%% Accompanies image file 'somecounts.pdf' (pdf, eps, ps)
\begingroup%
  \makeatletter%
  \providecommand\color[2][]{%
    \errmessage{(Inkscape) Color is used for the text in Inkscape, but the package 'color.sty' is not loaded}%
    \renewcommand\color[2][]{}%
  }%
  \providecommand\transparent[1]{%
    \errmessage{(Inkscape) Transparency is used (non-zero) for the text in Inkscape, but the package 'transparent.sty' is not loaded}%
    \renewcommand\transparent[1]{}%
  }%
  \providecommand\rotatebox[2]{#2}%
  \ifx\svgwidth\undefined%
    \setlength{\unitlength}{720bp}%
    \ifx\svgscale\undefined%
      \relax%
    \else%
      \setlength{\unitlength}{\unitlength * \real{\svgscale}}%
    \fi%
  \else%
    \setlength{\unitlength}{\svgwidth}%
  \fi%
  \global\let\svgwidth\undefined%
  \global\let\svgscale\undefined%
  \makeatother%
  \begin{picture}(1,0.47777778)%
    \put(0,0){\includegraphics[width=\unitlength,page=1]{somecounts.pdf}}%
  \end{picture}%
\endgroup%

%% file: maslovtriple.pdf_tex
%% Creator: Inkscape inkscape 0.91, www.inkscape.org
%% Accompanies image file 'maslovtriple.pdf' (pdf, eps, ps)
\begingroup%
  \makeatletter%
  \providecommand\color[2][]{%
    \errmessage{(Inkscape) Color is used for the text in Inkscape, but the package 'color.sty' is not loaded}%
    \renewcommand\color[2][]{}%
  }%
  \providecommand\transparent[1]{%
    \errmessage{(Inkscape) Transparency is used (non-zero) for the text in Inkscape, but the package 'transparent.sty' is not loaded}%
    \renewcommand\transparent[1]{}%
  }%
  \providecommand\rotatebox[2]{#2}%
  \ifx\svgwidth\undefined%
    \setlength{\unitlength}{728bp}%
    \ifx\svgscale\undefined%
      \relax%
    \else%
      \setlength{\unitlength}{\unitlength * \real{\svgscale}}%
    \fi%
  \else%
    \setlength{\unitlength}{\svgwidth}%
  \fi%
  \global\let\svgwidth\undefined%
  \global\let\svgscale\undefined%
  \makeatother%
  \begin{picture}(1,0.47252747)%
    \put(0,0){\includegraphics[width=\unitlength,page=1]{maslovtriple.pdf}}%
  \end{picture}%
\endgroup%

%% file: gradingshift.pdf_tex
%% Creator: Inkscape inkscape 0.91, www.inkscape.org
%% Accompanies image file 'gradingshift.pdf' (pdf, eps, ps)
\begingroup%
  \makeatletter%
  \providecommand\color[2][]{%
    \errmessage{(Inkscape) Color is used for the text in Inkscape, but the package 'color.sty' is not loaded}%
    \renewcommand\color[2][]{}%
  }%
  \providecommand\transparent[1]{%
    \errmessage{(Inkscape) Transparency is used (non-zero) for the text in Inkscape, but the package 'transparent.sty' is not loaded}%
    \renewcommand\transparent[1]{}%
  }%
  \providecommand\rotatebox[2]{#2}%
  \ifx\svgwidth\undefined%
    \setlength{\unitlength}{368bp}%
    \ifx\svgscale\undefined%
      \relax%
    \else%
      \setlength{\unitlength}{\unitlength * \real{\svgscale}}%
    \fi%
  \else%
    \setlength{\unitlength}{\svgwidth}%
  \fi%
  \global\let\svgwidth\undefined%
  \global\let\svgscale\undefined%
  \makeatother%
  \begin{picture}(1,0.93478261)%
    \put(0,0){\includegraphics[width=\unitlength,page=1]{gradingshift.pdf}}%
  \end{picture}%
\endgroup%

%% file: altchara.pdf_tex
%% Creator: Inkscape inkscape 0.91, www.inkscape.org
%% Accompanies image file 'altchara.pdf' (pdf, eps, ps)
\begingroup%
  \makeatletter%
  \providecommand\color[2][]{%
    \errmessage{(Inkscape) Color is used for the text in Inkscape, but the package 'color.sty' is not loaded}%
    \renewcommand\color[2][]{}%
  }%
  \providecommand\transparent[1]{%
    \errmessage{(Inkscape) Transparency is used (non-zero) for the text in Inkscape, but the package 'transparent.sty' is not loaded}%
    \renewcommand\transparent[1]{}%
  }%
  \providecommand\rotatebox[2]{#2}%
  \ifx\svgwidth\undefined%
    \setlength{\unitlength}{368bp}%
    \ifx\svgscale\undefined%
      \relax%
    \else%
      \setlength{\unitlength}{\unitlength * \real{\svgscale}}%
    \fi%
  \else%
    \setlength{\unitlength}{\svgwidth}%
  \fi%
  \global\let\svgwidth\undefined%
  \global\let\svgscale\undefined%
  \makeatother%
  \begin{picture}(1,0.93478261)%
    \put(0,0){\includegraphics[width=\unitlength,page=1]{altchara.pdf}}%
    \put(0.85102751,0.22113271){\color[rgb]{0,0,0}\makebox(0,0)[lb]{\smash{$\beta_0$}}}%
    \put(0,0){\includegraphics[width=\unitlength,page=2]{altchara.pdf}}%
    \put(0.29960997,0.04641108){\color[rgb]{0,0,0}\makebox(0,0)[lb]{\smash{$\alpha_0$}}}%
    \put(0,0){\includegraphics[width=\unitlength,page=3]{altchara.pdf}}%
    \put(0.33799601,0.83975241){\color[rgb]{0,0,0}\makebox(0,0)[lb]{\smash{$\lambda$}}}%
    \put(0,0){\includegraphics[width=\unitlength,page=4]{altchara.pdf}}%
    \put(0.17381452,0.52876754){\color[rgb]{0,0,0}\makebox(0,0)[lb]{\smash{$z_{-U}$}}}%
    \put(0,0){\includegraphics[width=\unitlength,page=5]{altchara.pdf}}%
    \put(0.65400092,0.67383285){\color[rgb]{0,0,0}\makebox(0,0)[lb]{\smash{$w_{-U}$}}}%
    \put(0,0){\includegraphics[width=\unitlength,page=6]{altchara.pdf}}%
    \put(0.02958992,0.61472627){\color[rgb]{0,0,0}\makebox(0,0)[lb]{\smash{$w_{-U}$}}}%
    \put(0,0){\includegraphics[width=\unitlength,page=7]{altchara.pdf}}%
    \put(0.02214094,0.70410656){\color[rgb]{0,0,0}\makebox(0,0)[lb]{\smash{$z_{-U}$}}}%
  \end{picture}%
\endgroup%

%% file: triple.pdf_tex
%% Creator: Inkscape inkscape 0.91, www.inkscape.org
%% Accompanies image file 'triple.pdf' (pdf, eps, ps)
\begingroup%
  \makeatletter%
  \providecommand\color[2][]{%
    \errmessage{(Inkscape) Color is used for the text in Inkscape, but the package 'color.sty' is not loaded}%
    \renewcommand\color[2][]{}%
  }%
  \providecommand\transparent[1]{%
    \errmessage{(Inkscape) Transparency is used (non-zero) for the text in Inkscape, but the package 'transparent.sty' is not loaded}%
    \renewcommand\transparent[1]{}%
  }%
  \providecommand\rotatebox[2]{#2}%
  \ifx\svgwidth\undefined%
    \setlength{\unitlength}{368bp}%
    \ifx\svgscale\undefined%
      \relax%
    \else%
      \setlength{\unitlength}{\unitlength * \real{\svgscale}}%
    \fi%
  \else%
    \setlength{\unitlength}{\svgwidth}%
  \fi%
  \global\let\svgwidth\undefined%
  \global\let\svgscale\undefined%
  \makeatother%
  \begin{picture}(1,0.93478261)%
    \put(0,0){\includegraphics[width=\unitlength,page=1]{triple.pdf}}%
  \end{picture}%
\endgroup%

%% file: rectproj.pdf_tex
%% Creator: Inkscape inkscape 0.92.1, www.inkscape.org
%% Accompanies image file 'rectproj.pdf' (pdf, eps, ps)
\begingroup%
  \makeatletter%
  \providecommand\color[2][]{%
    \errmessage{(Inkscape) Color is used for the text in Inkscape, but the package 'color.sty' is not loaded}%
    \renewcommand\color[2][]{}%
  }%
  \providecommand\transparent[1]{%
    \errmessage{(Inkscape) Transparency is used (non-zero) for the text in Inkscape, but the package 'transparent.sty' is not loaded}%
    \renewcommand\transparent[1]{}%
  }%
  \providecommand\rotatebox[2]{#2}%
  \ifx\svgwidth\undefined%
    \setlength{\unitlength}{560.00001526bp}%
    \ifx\svgscale\undefined%
      \relax%
    \else%
      \setlength{\unitlength}{\unitlength * \real{\svgscale}}%
    \fi%
  \else%
    \setlength{\unitlength}{\svgwidth}%
  \fi%
  \global\let\svgwidth\undefined%
  \global\let\svgscale\undefined%
  \makeatother%
  \begin{picture}(1,0.56249998)%
    \put(0,0){\includegraphics[width=\unitlength,page=1]{rectproj.pdf}}%
    \put(0.2470899,0.28833541){\color[rgb]{0,0,0}\makebox(0,0)[lb]{\smash{ }}}%
    \put(0,0){\includegraphics[width=\unitlength,page=2]{rectproj.pdf}}%
    \put(0.90168531,0.26193817){\color[rgb]{0,0,0}\makebox(0,0)[lb]{\smash{$-U$}}}%
  \end{picture}%
\endgroup%

%% file: finaltriple.pdf_tex
%% Creator: Inkscape inkscape 0.92.1, www.inkscape.org
%% Accompanies image file 'finaltriple.pdf' (pdf, eps, ps)
\begingroup%
  \makeatletter%
  \providecommand\color[2][]{%
    \errmessage{(Inkscape) Color is used for the text in Inkscape, but the package 'color.sty' is not loaded}%
    \renewcommand\color[2][]{}%
  }%
  \providecommand\transparent[1]{%
    \errmessage{(Inkscape) Transparency is used (non-zero) for the text in Inkscape, but the package 'transparent.sty' is not loaded}%
    \renewcommand\transparent[1]{}%
  }%
  \providecommand\rotatebox[2]{#2}%
  \ifx\svgwidth\undefined%
    \setlength{\unitlength}{720bp}%
    \ifx\svgscale\undefined%
      \relax%
    \else%
      \setlength{\unitlength}{\unitlength * \real{\svgscale}}%
    \fi%
  \else%
    \setlength{\unitlength}{\svgwidth}%
  \fi%
  \global\let\svgwidth\undefined%
  \global\let\svgscale\undefined%
  \makeatother%
  \begin{picture}(1,0.47777777)%
    \put(0,0){\includegraphics[width=\unitlength,page=1]{finaltriple.pdf}}%
  \end{picture}%
\endgroup%